\documentclass{amsart}
\pdfoutput=1

\usepackage{amsfonts, amssymb, amsmath, amsthm}                               
\usepackage{verbatim}                                                         

\newtheorem{theorem}{Theorem} 	      	      	                              
\newtheorem{corollary}[theorem]{Corollary}     	      	      	      	      
\newtheorem{lemma}[theorem]{Lemma}     	       	      	      	      	      
\newtheorem{proposition}[theorem]{Proposition} 	      	      	      	      
\newtheorem{definition}[theorem]{Definition}                                  
\newtheorem*{remark}{Remark}                                                  
\newtheorem*{ackn}{Acknowledgments}                                           

\numberwithin{equation}{section}                                              
\numberwithin{theorem}{section}                                               

\newcommand{\ul}[1]{\underline{#1}}                                           
\newcommand{\mf}[1]{\mathfrak{#1}}                                            
\newcommand{\mc}[1]{\mathcal{#1}}                                             
\newcommand{\ms}[1]{\mathsf{#1}}                                              

\newcommand{\Z}{\mathbb{Z}}                                                   
\newcommand{\R}{\mathbb{R}}                                                   
\newcommand{\C}{\mathbb{C}}                                                   
\newcommand{\Sph}{\mathbb{S}}                                                 

\newcommand{\paren}[1]{\left(#1\right)}                                       
\newcommand{\brak}[1]{\left[#1\right]}                                        
\newcommand{\norm}[1]{\left\|#1\right\|}                                      

\DeclareMathOperator{\real}{Re}                                               
\DeclareMathOperator{\imag}{Im}                                               

\DeclareMathOperator{\trace}{tr}                                              

\newcommand{\lapl}{\Delta}                                                    
\newcommand{\cint}{{\begingroup\textstyle\smallint\endgroup}}                 

\newcommand{\mind}{/\mspace{-9mu}\gamma}                                      
\newcommand{\vind}{/\mspace{-9mu}\epsilon}                                    
\newcommand{\nasla}{/\mspace{-11mu}\nabla}                                    
\newcommand{\lasl}{/\mspace{-12mu}\lapl}                                      
\newcommand{\sauss}{/\mspace{-11mu}\mc{K}}                                    
\newcommand{\sodge}{/\mspace{-11mu}\mc{D}}                                    
\newcommand{\trase}{/\mspace{-14mu}\trace}                                    
\newcommand{\sorm}[1]{/\mspace{-11mu}#1}                                      

\newcommand{\mbar}{\bar{\gamma}}                                              
\newcommand{\vbar}{\bar{\epsilon}}                                            

\newcommand{\ass}[2]{{\bf (#1)}${}_{#2}$}                                     

\allowdisplaybreaks

\begin{document}

\title[Null Cones to Infinity]{On the Geometry of Null Cones to Infinity Under Curvature Flux Bounds}
\author{Spyros Alexakis, Arick Shao}

\address{Department of Mathematics, University of Toronto, Toronto, ON, Canada M5S 2E4}
\email{alexakis@math.utoronto.ca, ashao@math.utoronto.ca}


\begin{abstract}
The main objective of this paper is to control the geometry of a future outgoing truncated null cone extending smoothly toward infinity in an Einstein-vacuum spacetime.
In particular, we wish to do this under minimal regularity assumptions, namely, at the (weighted) $L^2$-curvature level.
We show that if the curvature flux and the data on an initial sphere of the cone are sufficiently close to the corresponding values in a standard Minkowski or Schwarzschild null cone, then we can obtain quantitative bounds on the geometry of the entire infinite cone.
The same bounds also imply the existence of limits at infinity, along the null cone, of naturally scaled geometric quantities.
In \cite{alex_shao:bondi}, we will apply these results in order to control various physical quantities---e.g., the Bondi energy and (linear and angular) momenta---associated with such infinite null cones in vacuum spacetimes.
\end{abstract}

\maketitle

\section{Introduction} \label{sec.intro}

This is the first of two papers on infinite null cones in vacuum spacetimes. 
The immediate aim of our project is to obtain control for physical quantities (the Bondi energy, the linear and angular momenta, and the rate of energy loss) associated to the intersection of our cone with future null infinity, under very weak assumptions, that is, at the level of the $L^2$-norm of the spacetime curvature.
This is achieved in \cite{alex_shao:bondi}, which relies on the present paper in an essential way.
In this paper, we complete the first part of our program---we quantitatively control the geometry of infinite null cones, under these same assumptions.

A key component of our argument is a renormalization of the null structure equations governing the geometry of infinite null cones.
In particular, this renormalization is carried out in a manner that preserves the essential covariant characteristics of our equations.
This will be vital due to the low regularity of our setting.

\vspace{8pt}

Consider a $4$-dimensional Einstein-vacuum spacetime $(M, g)$, and let $\mc{N}$ be a smooth null hypersurface in $M$.
The \emph{curvature flux} of $\mc{N}$ generally refers to an $L^2$-norm along $\mc{N}$ of certain components of the spacetime curvature tensor.
Such quantities arise as fluxes for an energy current built out of the Bel-Robinson tensor (via contraction against three causal vector fields), and hence are fundamental for dealing with local energy estimates involving the curvature.

A number of results in general relativity have relied heavily on this flux of curvature or its variations.
\footnote{For example, one can take $L^2$-norms of derivatives of the spacetime curvature.}
Examples include the stability of Minkowski spacetime \cite{bie_zip:stb_mink, chr_kl:stb_mink, kl_nic:stb_mink}, breakdown criteria for the Einstein equations \cite{kl_rod:bdc, parl:bdc, shao:bdc_nv, shao:bdc_nvp, wang:ibdc}, the formation of trapped surfaces \cite{chr:gr_bh, kl_rod:fbh}, and the recent resolution of the $L^2$-curvature conjecture \cite{kl_rod_szf:blc, szf:par1, szf:par2, szf:par3, szf:par4}.
Regarding the breakdown criteria results, a major component of their proofs is precisely that of controlling the geometry of null cones (beginning from a single point) by its curvature flux.
Similarly, for the $L^2$-curvature conjecture, an important step is to control the geometry of null hyperplanes by a curvature flux (although this is not the central issue in this problem).

Here, we consider instead \emph{infinite} future-directed truncated null cones $\mc{N}$.
The corresponding objective is to control the geometry of $\mc{N}$ by a corresponding \emph{weighted} curvature flux on $\mc{N}$.
The main result of this paper accomplishes this task for a certain class of null cones and choice of weights; see further details below.

For such infinite null cones, one can additionally inquire whether the control one has on this geometry is sufficient to bound quantities of physical interest defined at infinity along our cone, such as the Bondi energy and the rate of energy loss.
The control of such quantities at null infinity is a point of interest in global problems; for example, it was obtained by Christodoulou and Klainerman in the stability of Minkowski spacetime in \cite{chr_kl:stb_mink}.
It turns out that we \emph{can} control these quantities of physical interest on one cone, purely in terms of this weighted curvature flux (as opposed to quantities requiring higher regularity).
This is the primary focus of \cite{alex_shao:bondi}, which relies on the present paper in an essential way. 

\vspace{8pt}

Our result (Theorem \ref{thm.nc_phys} below) is perturbative in nature; it applies to any infinite null cone $\mc{N}$ that is sufficiently close to the rotationally symmetric null cones in the Schwarzschild (and hence Minkowski) spacetimes.
More specifically, we assume:
\begin{itemize} 
\item Certain components of the spacetime curvature on $\mc{N}$ (with respect to a geodesic foliation) are $L^2$-close, with suitable weights, to their corresponding values in a Schwarzschild spacetime.

\item On the initial sphere of $\mc{N}$, the connection coefficients for the cone (with respect to the same foliation) are likewise sufficiently close, in the appropriate norms, to the expected values in the same Schwarzschild spacetime.
\end{itemize}
The objective is to prove the following results:
\begin{itemize}
\item The connection coefficients will remain uniformly close, in the appropriate norms, to their Schwarzschild values on all of $\mc{N}$.

\item The same connection coefficients, with suitable weights, have limits at infinity in the appropriate spaces that can also be controlled.
\end{itemize}
In near-Minkowski settings, these $L^2$-curvature quantities arose as weighted fluxes of curvature from the Bel-Robinson tensor and were used extensively in \cite{bie_zip:stb_mink, chr_kl:stb_mink, kl_nic:stb_mink}.
\footnote{More specifically, these fluxes can be obtained by contracting the Bel-Robinson tensor with the vector fields $\mathbf{T}$, $\mathbf{T}$, and $\mathbf{K}$, where $\mathbf{T}$ and $\mathbf{K}$ are suitable adaptations of the time translation vector field and the Morawetz vector field, respectively, in Minkowski spacetime.}

The analytical framework of this paper is analogous to that of \cite{kl_rod:cg}, but now considered for an infinite null cone.
In particular, our starting point is the system of structure equations on a null cone for a geodesic foliation, under the Einstein-vacuum equations.
For our argument, we apply a renormalization of our system of equations, in which we introduce an appropriate reweighting of each geometric quantity within the system.
\footnote{In the near-Minkowski case, these weights are consistent with those of \cite{kl_nic:stb_mink}.}
In this renormalized setting, we obtain a corresponding set of structure equations that is formally similar to the structure equations for the finite null cone studied in \cite{kl_rod:cg}.
From a geometric perspective, the analysis, as well as the resulting estimates and limits, are more naturally expressed in terms of this new system.
Moreover, these results will readily translate to corresponding results for the original ``physical" system.

This renormalization is well-suited for studying the structure equations at the low regularities we deal with here.
One of the main difficulties in working at the $L^2$-curvature level is the lack of regularity of natural coordinate systems and their associated frames.
As a result of this, one often must work covariantly;
in particular, the analysis on the null structure equations is at the level of tensors.
An essential feature of the renormalization that we introduce is that this covariant 
formulation of the equations \emph{transforms naturally} into the renormalized setting.
\footnote{The precise formulation of this covariant structure of the equations is given in Section \ref{sec.fol}.}

Furthermore, although we work only with geodesic foliations in this paper, the renormalization scheme used here, as well as the framework and the estimates from \cite{shao:stt}, should be readily applicable to other similar situations.
\footnote{It is interesting to remark that one can recover physical quantities such as the Bondi energy by working exclusively with geodesic foliations; this follows from this paper and \cite{alex_shao:bondi}.}
These include time-foliated null hypersurfaces, as in \cite{chr_kl:stb_mink, kl_rod_szf:blc}, or double-null-foliated settings, such as \cite{kl_nic:stb_mink}.
For these reasons, we hope that the ideas introduced here may prove useful in other more global problems at the $L^2$-curvature level.

For instance, as mentioned before, one of the conclusions of the global stability of Minkowski spacetime, \cite{chr_kl:stb_mink}, is control on the Bondi mass at future null infinity.
Our results in \cite{alex_shao:bondi} and in this paper could be applied to a similar effect if such a global stability result were to be extended to the $L^2$-curvature level.

\vspace{8pt}

The techniques for controlling the geometry of null cones by the curvature flux were first developed in \cite{kl_rod:cg}, which dealt with geodesically foliated finite null cones in a vacuum spacetime.
There are several variations of this result addressing similar situations, including null cones beginning from a point \cite{wang:cg, wang:cgp}, time-foliated null cones \cite{parl:bdc, wang:ibdc}, and non-vacuum spacetimes \cite{shao:bdc_nv}.
\footnote{Many of the technical aspects of these analyses were more recently unified in \cite{shao:stt}.}

Because of the low regularity of the system, in proving the results of \cite{kl_rod:cg} (as well as its variants), one required many technical developments; see \cite{kl_rod:glp, kl_rod:stt}.
Significant simplifications and generalizations of these developments were obtained more recently by the second author in \cite{shao:stt}.
We will take full advantage of this in our work; see Section \ref{sec.intro_tech} for additional discussions.
Furthermore, our analytical framework, in particular the precise formulation of our renormalized system, will be expressed in the language developed in \cite{shao:stt}; see Sections \ref{sec.geom} and \ref{sec.fol} for details.

\subsection{The Main Results} \label{sec.intro_thm}

Throughout, we let $(M, g)$ denote a $4$-dimensional vacuum spacetime, and we let $\mc{N}$ denote a future-directed, smooth, and geodesically foliated null cone in $M$, emanating from a $2$-sphere $\mc{S}$.

For our analysis, we will consider $\mc{N}$ as a smoothly parametrized foliation, $\mc{N} \simeq [s_0, \infty) \times \Sph^2$.
The first parameter in this product refers to the chosen affine parameter for the null geodesic generators of $\mc{N}$, while the copies of $\Sph^2$ are the level sets of the affine parameter.
As each level sphere of $\mc{N}$ is spacelike, we can consider the Riemannian metric $\mind$ on the spheres induced from the spacetime metric $g$.
Furthermore, we normalize our affine parameter $s$ on $\mc{N}$ so that the initial sphere $\mc{S}$ corresponds to the level set $s = s_0$ and has area $4 \pi s_0^2$.

The next step is to describe the objects of our analysis:
\begin{itemize}
\item \emph{Connection coefficients}: These are, as usual, quantities that correspond to one derivative of the metric.
More accurately, these are spacetime covariant derivatives of certain adapted null frames on $\mc{N}$.

\item \emph{Curvature components}: These refer to the spacetime curvature of $(M, g)$, decomposed in terms of the same null frames on $\mc{N}$ as above.
\end{itemize}
The main idea is, as in \cite{chr:gr_bh, chr_kl:stb_mink, kl_rod:cg}, to reinterpret the aforementioned connection and curvature quantities as ``horizontal" tensorial quantities, i.e., tensor fields on $\mc{N}$ which are everywhere tangent to the level spheres of $\mc{N}$.
Consequently, we can treat each of these connection and curvature quantities as a smoothly varying family of tensor fields on $\Sph^2$, parametrized by the affine parameter $s$.

One main feature of this system is that the geometries of the of level spheres of $\mc{N}$ also evolve.
In other words, the horizontal tensor field $\mind$ on $\mc{N}$ constructed from the metrics induced from $g$ also evolves as a function of $s$.
A consequence of this is that the norms, the elliptic operators, and the evolutionary operators that we will consider will also evolve depending on the affine parameter.
Throughout this paper, we will analyze such horizontal tensorial quantities on this evolving geometric setting using the formalisms developed in \cite{shao:stt}.

\subsubsection{Connection and Curvature Decompositions}

We now discuss in further detail the horizontal decompositions for the connection and curvature quantities.
We begin with the connection coefficients, defined with respect to our geodesic foliation.

The most important connection quantity is the \emph{intrinsic null second fundamental form}, $\chi$, which is defined as the second fundamental form of the level spheres of $\mc{N}$, in the future null direction tangent to $\mc{N}$.
Intuitively, $\chi$ determines how $\mind$, and hence the geometry of $\mc{N}$, evolves as one moves along $\mc{N}$ in this future null direction.
$\chi$ can be further decomposed into its trace and traceless parts, i.e., the \emph{expansion} $\trase \chi$ and the \emph{shear} $\hat{\chi}$.
In particular, $\trase \chi$ describes how the area element of the level spheres of $\mc{N}$ evolve, and is related to the formation of null conjugate points.

Similarly, the \emph{extrinsic null second fundamental form}, $\ul{\chi}$, represents the second fundamental forms of the level spheres of $\mc{N}$ in the transverse future null direction orthogonal to these level spheres.
Like for $\chi$, one can also decompose $\ul{\chi}$ into its trace and traceless parts.
As our analysis is concerned only with $\mc{N}$ itself, $\ul{\chi}$ lacks the same intrinsic significance as $\chi$ in our setting.
However, $\ul{\chi}$ will be shown to decay less than the other connection coefficients.
In fact, that the level spheres of $\mc{N}$ fail to be asymptotically round as one approaches infinity is due to this lack of decay for $\trase \ul{\chi}$.
This will play a central role in \cite{alex_shao:bondi}.

The final connection coefficient in the geodesic foliation is the \emph{torsion}, $\zeta$.
In the geodesic foliation, this quantity can be roughly interpreted as the failure of the transverse null direction orthogonal to the level spheres of $\mc{N}$ to evolve in a parallel fashion along the null direction that is tangent to $\mc{N}$.

Next are the curvature quantities, which represent the various components of the spacetime curvature $R$ on $\mc{N}$.
To define these, one first takes two future null vector fields, $L$ and $\ul{L}$, with the former representing the future tangent null direction in $\mc{N}$, and with the latter representing the future transverse null direction normal to the level spheres of $\mc{N}$.
The curvature components are then defined by contracting $R$ with one or more instances of $L$ and $\ul{L}$, and by requiring that the remaining components are horizontal, i.e., tangent to the level spheres.
In keeping with notational traditions, we denote the resulting components by $\alpha$, $\beta$, $\rho$, $\sigma$, $\ul{\beta}$, and $\ul{\alpha}$.
Since Einstein-vacuum spacetimes are by definition Ricci-flat, these curvature quantities comprise all the independent components of $R$.

Finally, in our analysis, we will also require an additional scalar quantity $\mu$ on $\mc{N}$, called the \emph{mass aspect function}.
$\mu$ is defined directly from the connection coeffients and the curvature component $\rho$.
In particular, this quantity is closely related to the Hawking masses of the level spheres; see \cite{chr_kl:stb_mink, kl_rod:cg}.

For precise definitions of all the above quantities, see Section \ref{sec.nc_rc}.

\subsubsection{A Rough Theorem Statement}

Recall the main result of this paper roughly states that if the weighted curvature flux of $\mc{N}$ and initial data on $\mc{S}$ are sufficiently close to their corresponding values in a Schwarzschild spacetime, then the connection coefficients on $\mc{N}$ remain close to their Schwarzschild values.
Below, we further clarify the meanings of ``weighted curvature flux" and ``initial data".

With respect to the geodesic foliation, the weighted curvature flux of $\mc{N}$ is
\begin{align*}
\mc{F} &= s_0^{-\frac{3}{2}} \| s^2 \alpha \|_{ L^2 (\mc{N}) } + s_0^{-\frac{3}{2}} \| s^2 \beta \|_{ L^2 (\mc{N}) } + s_0^{-\frac{1}{2}} \| s \rho \|_{ L^2 (\mc{N}) } \\
&\qquad + s_0^{-\frac{1}{2}} \| s \sigma \|_{ L^2 (\mc{N}) } + s_0^\frac{1}{2} \| \ul{\beta} \|_{ L^2 (\mc{N}) } \text{,}
\end{align*}
Note this formula includes all the spacetime curvature quantities except for $\ul{\alpha}$.
In particular, the excluded component $\ul{\alpha}$ is the only component of $R$ which does not contain any $L$-components.
In \cite{chr_kl:stb_mink, kl_nic:stb_mink}, for example, such quantities were intimately tied to energy estimates involving the spacetime curvature.

\begin{remark}
The main heuristic for the weights within $\mc{F}$ is that the affine parameter $s$ will remain comparable to the radii of the level spheres of $\mc{N}$.
In contrast with similar developments in \cite{bie_zip:stb_mink, chr_kl:stb_mink, kl_nic:stb_mink}, we use $s$ here rather than the actual radius, as it is an easier quantity to manipulate algebraically.
\end{remark}

In our setting, however, we are interested not in the curvature flux itself, but rather in its deviation from the Schwarzschild values, i.e., the quantity
\begin{align*}
\delta \mc{F} &= s_0^{-\frac{3}{2}} \| s^2 ( \alpha - \alpha^S ) \|_{ L^2 (\mc{N}) } + s_0^{-\frac{3}{2}} \| s^2 ( \beta - \beta^S ) \|_{ L^2 (\mc{N}) } + s_0^{-\frac{1}{2}} \| s ( \rho - \rho^S ) \|_{ L^2 (\mc{N}) } \\
&\qquad + s_0^{-\frac{1}{2}} \| s ( \sigma - \sigma^S ) \|_{ L^2 (\mc{N}) } + s_0^\frac{1}{2} \| \ul{\beta} - \ul{\beta}^S \|_{ L^2 (\mc{N}) } \text{,}
\end{align*}
where $\alpha^S$, $\beta^S$, $\rho^S$, $\sigma^S$, and $\ul{\beta}^S$ denote the values of the curvature components in a Schwarzschild spacetime, with mass $0 \leq m < s_0 / 2$ (the latter bound ensures $\mc{S}$ represents a sphere in the outer region).
Standard computations show that the only nonvanishing Schwarzschild component here is $\rho^S$; see Section \ref{sec.nc_ms} for details.

Finally, the initial data for $\mc{S}$ reflects the deviation of $\chi$, $\ul{\chi}$, $\zeta$, and $\mu$ from their Schwarzschild values (see Section \ref{sec.nc_ms}), in the appropriate norms.
More specifically,
\begin{align*}
\delta \mc{I} &= s_0 \| \trase \chi - ( \trase \chi )^S \|_{ L^\infty (\mc{S}) } + s_0^\frac{1}{2} \| \chi - \chi^S \|_{ \ms{H} (\mc{S}) } + s_0^\frac{1}{2} \| \zeta - \zeta^S \|_{ \ms{H} (\mc{S}) } \\
\notag &\qquad + \| \ul{\chi} - \ul{\chi}^S \|_{ \ms{B} (\mc{S}) } + s_0 \| \nasla ( \trase \chi ) - [ \nasla ( \trase \chi ) ]_S \|_{ \ms{B} (\mc{S}) } + s_0 \| \mu - \mu^S \|_{ \ms{B} (\mc{S}) } \text{.}
\end{align*}
In the above, $\ms{H}$ is a (geometric tensorial) $H^{1/2}$-type norm on $\mc{S}$, while $\ms{B}$ is a similar zero-derivative Besov-type norm.
An unfortunate by-product of working at the curvature flux level is that such Besov norms are required in the analysis.

We now give a very rough statement of the main theorem.

\begin{theorem} \label{thm.rough_est}
Assume that
\[ \delta \mc{F} + \delta \mc{I} \leq \Gamma \text{.} \]
Suppose $\Gamma$ is sufficiently small with respect to the geometry of $\mc{S}$, that is, the weighted curvature and initial data of $\mc{N}$ remain close to their Schwarzschild values.
Then, the geometry of $\mc{N}$ remains close to that of the corresponding Schwarzschild null cone.
To be more specific, the connection deviations
\[ \chi - \chi^S \text{,} \qquad \ul{\chi} - \ul{\chi}^S \text{,} \qquad \zeta - \zeta^S \text{,} \qquad \mu - \mu^S \]
will also be bounded by $\Gamma$, in the appropriate norms.
Furthermore, up to a rescaling, the geometries of the level spheres of $\mc{N}$ remain close to that of $\mc{S}$.
\end{theorem}

For the precise version of the theorem, see Theorem \ref{thm.nc_phys}, in Section \ref{sec.main_pthm}.
Next, one can use the estimates within Theorem \ref{thm.rough_est} to generate asymptotic limits at infinity, of both the geometries of the level spheres of $\mc{N}$ and the connection coefficients.
A rough statement of this result is stated in the subsequent theorem.
For the precise statements, see Corollary \ref{thm.nc_renorm_lim}, in Section \ref{sec.main_rthm}.

\begin{theorem} \label{thm.rough_lim}
Assume the hypotheses, and hence the conclusions, of Theorem \ref{thm.rough_est}.
Then, as $s \nearrow \infty$, the geometries of the level spheres of $\mc{N}$ converge, after an appropriate rescaling, to a rough limiting geometry on $\Sph^2$.
Furthermore, certain appropriate rescalings of $\chi - \chi^S$, $\ul{\chi} - \ul{\chi}^S$, $\zeta - \zeta^S$, and $\mu - \mu^S$ will have limits as $s \nearrow \infty$, with respect to the appropriate normed spaces.
In particular, the Hawking masses of the level spheres of $\mc{N}$ have a limit as $s \nearrow \infty$.
\end{theorem}

From Theorem \ref{thm.rough_est}, we can in particular control the \emph{Hawking masses} of the level spheres of $\mc{N}$.
In \cite[Ch. 17]{chr_kl:stb_mink}, for spacetimes that were sufficiently near-Minkowski, and for certain (time-foliated) null cones extending to infinity, it was shown that:
\begin{itemize}
\item The Hawking masses of the level spheres of such a null cone converge to some finite nonnegative limit at null infinity.

\item Moreover, this limit is the Bondi mass associated with this cone.
\end{itemize}
In the minimal setting of this paper, we can also show that the corresponding Hawking masses of $\mc{N}$ are bounded and have a controlled limit at infinity.

In our case, it is not clear a priori that this limit of the Hawking masses corresponds to any notion of Bondi energy or mass.
As discussed in \cite{sauter:penrose}, for instance, this connection between Hawking and Bondi masses depends closely on the spheres foliating $\mc{N}$ becoming asymptotically round at infinity.
In the companion paper \cite{alex_shao:bondi}, however, we will construct, under the same assumptions, such an asymptotically round family of spheres in $\mc{N}$, in order to control the Bondi energy.

We discuss the background and some main ideas of the proof in the remainder of the introduction.
For now, some remarks are in order.
The first concerns the significance and role of the low regularity regime in which we work. 

\begin{remark}
A ``poor-man's version" of our main theorems (and also of the sequel \cite{alex_shao:bondi}) would be to assume higher regularity for the curvature coefficients on $\mc{N}$.
For instance, suppose one has control for up to two derivatives of $\alpha$, $\beta$, $\rho$, $\sigma$, and $\beta$ (in directions tangent to $\mc{N}$, and with the appropriate weights) in $L^2 (\mc{N})$, along with corresponding control for the initial data on $\mc{S}$.
Then, using elementary methods, one can fairly easily derive (weighted) pointwise bounds for the connection coefficients on $\mc{N}$, and hence arrive at the conclusions of Theorems \ref{thm.rough_est} and \ref{thm.rough_lim}.
\footnote{This is in fact part of the argument for the global stability of Minkowski spacetime, \cite{chr_kl:stb_mink, kl_nic:stb_mink}.}

This derivation can be briefly summarized as follows.
Under uniform regularity bounds on the induced metrics $\mind_s$ on the level spheres of $s$, one can obtain from Sobolev estimates pointwise decay (in inverse powers of $s$) for the curvature components.
Then, treating the structure equations relating the curvature and connection coefficients (in particular, \eqref{eq.structure_ev} and \eqref{eq.structure_evd}) as ODEs in the null parameter $s$, we can derive the desired bounds on the connection coefficients. 
\footnote{For the uniform regularity on the metrics which initiates this derivation, one also requires a bootstrap argument analogous to that of the present paper.}
In particular, one avoids at this regularity all the technical issues required at the $L^2$-curvature level, in particular Besov spaces and geometric estimates under low regularity.

In working purely at the $L^2$-curvature level here, we establish a stronger result: the geometry of infinite null cones can break down only if the \emph{undifferentiated} curvature itself (and not just the derivatives of the curvature) deviates from the corresponding Schwarzschild values.
Unfortunately, this greatly complicates the proof, due not only to the necessity of Besov spaces and estimates, but also because one must rely more carefully on all the structure hidden within the Einstein equations.
As mentioned before, one motivation for advancing in this direction is that our result would be closely connected to any potential extension of the global stability of Minkowski spacetime to the $L^2$-curvature level.
\end{remark}

\begin{remark}
In our main theorems (and also in \cite{alex_shao:bondi}), we deal with null hypersurfaces that are close to the shear-free truncated null cones in the Schwarzschild spacetimes.
This is partly out of convenience, since the connection and curvature quantities are explicitly given and well-known on these canonical Schwarzschild surfaces.
In particular, this makes the tasks of formulating and measuring closeness easier.

It is likely that analogous results can be proved for null cones near Kerr spacetimes.
(However, our result already captures surfaces that are close to certain cones in Kerr spacetimes with small angular momentum, since the latter are themselves close to the shear-free cones in Schwarzschild). 
Furthermore, it is conceivable that the perturbative result derived here could be true for smooth, infinite null hypersurfaces which are close (at the $L^2$-curvature level) to a fixed null hypersurface in a general asymptotically flat spacetime.
While the methods of analysis would likely be analogous to that of this paper, the differential equations governing the null geometry would tend to be unwieldy in this general setting.
\end{remark}

\begin{remark}
In view of recent works by Luk and Rodnianski, \cite{luk_rod:grav_wave, luk_rod:grav_wave_inter}, on gravitational waves, one may ask whether Theorems \ref{thm.rough_est} and \ref{thm.rough_lim} hold even without assumptions on the curvature component $\alpha$.
As far as we know, the smallness of $\alpha$ seems to be necessary, as the low regularity of our setting forces us to utilize \emph{all} of the structure equations for the connection and curvature components.
\end{remark}

\subsubsection{The Renormalization Procedure} \label{sec.intro_outline}

We now give a brief outline of how the proof of Theorem \ref{thm.rough_est} proceeds.
Primarily, we wish to convert our setting to one which can be treated by methods analogous to \cite{kl_rod:cg, kl_rod:glp, kl_rod:stt}.
Moreover, we want the general results developed in \cite{shao:stt} to be applicable to our new setting.
Both of these objectives are accomplished by adopting certain \emph{renormalizations} to our system.

This first step is to convert $\mc{N}$ from an \emph{infinite cone} into a \emph{finite cylinder}.
To do this, we rescale the metrics $\mind$ on the level spheres of $\mc{N}$, so that they have almost constant area.
In practice, this allows us to analyze all the level spheres of $\mc{N}$ in a uniform manner.
Next, we adopt a change of the evolutionary variable to convert the infinite interval $[s_0, \infty)$ to a finite interval $[0, 1)$.
We also make corresponding renormalizations for both the curvature and the connection coefficients on $\mc{N}$.
For details behind the specific rescalings and transformations, see Section \ref{sec.nc_renorm}.

\begin{remark}
Note we have the freedom to choose weights for each curvature and connection quantity.
Though the chosen weights correspond to \cite{kl_nic:stb_mink}, some do not reflect those one would obtain from the usual conformal compactification of spacetime.
\end{remark}

Another key component of this process is the construction of a \emph{covariant system} (in the sense of \cite{shao:stt}) with respect to our finite null cylinder.
For this, we introduce connections on $\mc{N}$, compatible with the \emph{rescaled} metrics and adapted to the \emph{finite} evolutionary variable $t \in [0, 1)$.
This defines a notion of covariant differentiation on $\mc{N}$ that is adapted to our renormalized system.

Combining all these steps results in a new equivalent system in terms of the renormalized geometry, connection coefficients, and curvature components.
Moreover, this new system is formally similar to the original physical system for a \emph{finite} null cone.
Thus, the analysis we perform can in large part reduce to the ideas developed in \cite{kl_rod:cg}, along with some modifications and simplifications from \cite{shao:stt}.

As in \cite{kl_rod:cg}, we establish our desired estimates through an elaborate bootstrap argument.
We take as bootstrap assumptions some of the estimates on the connection coefficients that we wish to prove.
This is an important step, as these assumptions are required in order to validate many of the tools of analysis that are used.
\footnote{These assumptions imply that the geometries of the renormalized level spheres of $\mc{N}$ change very little.  As a result of this, various Sobolev, product, and elliptic estimates from \cite{shao:stt} can be applied in a uniform manner to all the level spheres of $\mc{N}$.  In particular, the aforementioned uniformity implies that $s$ remains comparable to the radii of the unrenormalized level spheres.}
From these assumptions, we can eventually derive all the desired estimates on the connection coefficients.
To close the bootstrap argument, we also obtain strictly improved versions of the bootstrap assumptions.
This implies that the estimates we have obtained in fact hold without our bootstrap assumptions.

\begin{remark}
For further discussions on the intricacies of the bootstrap argument itself, the reader is referred to the introductions of \cite{kl_rod:cg, wang:cg}.
\end{remark}

From this analysis, we obtain \emph{unweighted} estimates for the \emph{renormalized} connection coefficients on $\mc{N}$, in terms of the renormalized geometry.
The precise results of this procedure are stated as Theorem \ref{thm.nc_renorm}, i.e., the ``renormalized main theorem".
By inverting our renormalization procedure, we can restate these results in terms of the \emph{physical} geometry and connection coefficients.
This results precisely in the desired \emph{weighted} estimates, thereby proving Theorems \ref{thm.rough_est} and \ref{thm.nc_phys}.

We also remark that the limits obtained at infinity are most easily stated in terms of the renormalized system, since the geometries of the renormalized level spheres remain close to that of the initial sphere.
In particular, Corollary \ref{thm.nc_renorm_lim} (the precise version of Theorem \ref{thm.rough_lim}) is expressed entirely in the renormalized picture.
Furthermore, in the sequel \cite{alex_shao:bondi}, in which we control the Bondi energy, much of the analysis is once again performed in this renormalized setting.

Aside from the immediate problem, we propose that this renormalization process also provides a template for analyzing other geometric situations.
In general, this process transforms a foliation so that the evolutionary parameter has finite length and the geometries of the leaves of the foliation change very little.
In this paper, the upshot is that the renormalized system satisfies abstract assumptions which validate a wide range of estimates established in \cite{shao:stt}.

Finally, while the analysis here applies to only a single null cone, it is hoped that adaptations of this renormalization argument could also be used for studying regions of spacetimes (for example, the double-null foliations used in \cite{chr:gr_bh, kl_nic:stb_mink}).
In particular, because of these considerations, this procedure perhaps may also be applied toward analyzing existence theorems, again possibly at the $L^2$-curvature level, for the Einstein equations extending up to (a part of) null infinity.

\subsection{Technical Improvements} \label{sec.intro_tech}

Although we use the same template in our argument as in previous works (\cite{kl_rod:cg, parl:bdc, shao:bdc_nv, wang:cg, wang:cgp, wang:tbdc}), we also improve upon some of the techniques used in the above works.
Below, we briefly discuss the technical innovations employed in this paper and in \cite{shao:stt}.

\subsubsection{Bilinear Product Estimates}

In \cite{kl_rod:stt}, various bilinear product estimates, essential for the main argument in \cite{kl_rod:cg}, were established.
Due to the lack of geometric regularity, the proofs required the construction and application of a geometric tensorial Littlewood-Paley theory based on the heat flow; see \cite{kl_rod:glp} for this development.
This made the proofs in \cite{kl_rod:stt} both lengthy and highly technical.
Furthermore, the proofs of these estimates relied heavily upon the specific setting (geodesic foliation, finite null cones, vacuum spacetimes, etc.).
As a result, although these techniques can be adapted to other related settings, one must in principle redo these arguments for each separate setting.
For example, this was the case in \cite{wang:cg, wang:cgp}, which dealt instead with null cones beginning from a point.

In \cite{shao:stt}, a simpler and more systematic method for deriving these bilinear product and sharp trace estimates was presented.
In contrast to \cite{kl_rod:stt}, which resorted to the geometric Littlewood-Paley theory of \cite{kl_rod:glp}, these estimates were obtained in \cite{shao:stt} by reducing them to their (much simpler) Euclidean analogues.
That this reduction is possible is a consequence of two new observations:
\begin{itemize}
\item From the Codazzi equations, applied to the level spheres of $\mc{N}$, the curl of $\chi$ has slightly better properties than other derivatives of $\chi$.

\item This improved regularity for the curl of $\chi$ yields strictly better regularity for certain parallel-transported frames.
\end{itemize}
That these special frames are more regular than the usual coordinate frames allows us to reduce these geometric and tensorial product estimates to their Euclidean and scalar counterparts, which can be proved using \emph{classical} Littlewood-Paley theory.

Furthermore, in \cite{shao:stt}, the estimates are stated in terms of abstract foliations satisfying certain regularity assumptions.
The main advantage of this presentation is that these assumptions apply not only to the setting of this paper, but also to the settings of \cite{kl_rod:cg, parl:bdc, shao:bdc_nv, wang:cg, wang:cgp, wang:tbdc}.
Consequently, for any reasonable variation of the ``null cone with bounded curvature flux" problem, one can very quickly and easily validate a whole family of tensorial product estimates using \cite{shao:stt}.

\subsubsection{Elliptic Estimates}

In \cite{kl_rod:cg, parl:bdc, shao:bdc_nv, wang:cg, wang:cgp, wang:tbdc}, a major difficulty arose from the fact that the Gauss curvatures of the level spheres of $\mc{N}$ were highly irregular.
One derived only an $H^{-1/2}$-type bound on these Gauss curvatures; moreover, this was achieved only after a highly nontrivial argument containing, for instance, a technical commutator estimate involving heat flows.
Thus, many elliptic estimates that were straightforward in more regular settings became both technical and lengthy.
The most difficult examples were Besov elliptic estimates for the symmetric Hodge operators, which now required a much more delicate analysis.
In particular, the estimates contained additional error terms, which then added considerable length and complexity to the overall argument; see \cite{wang:cg} for details and discussions.

In \cite{shao:stt}, it was shown that these additional error terms were in fact unnecessary.
Furthermore, these estimates, without the error terms, could be proved using a far shorter argument than before.
The main new observation in the null cone setting is that the only part of the Gauss curvature that is not $L^2$-controlled can be expressed as a divergence (of $\zeta$).
This allowed for a conformal transformation into a different metric for which the Gauss curvature is entirely $L^2$-controlled.
In effect, one absorbs the low-regularity term into the chosen conformal factor.
\footnote{This technique also has similar applications in \cite{alex_shao:bondi}.}

The advantage gained from this transformation is that all the desired Besov-elliptic estimates can be derived far more easily with respect to this regularized metric.
Furthermore, the Hodge operators under consideration are conformally invariant, and a brief but careful analysis shows that these estimates can in fact be transferred back to the original metric.
Moreover, similar to the bilinear product estimates, the elliptic estimates in \cite{shao:stt} were stated in an abstract setting that applies not only to this paper, but also to the settings of previous works.

\subsubsection{Infinite Decompositions}

In previous arguments, an elaborate infinite decomposition of tensor fields was required within the main bootstrap argument in order to apply the necessary bilinear and sharp trace estimates.
In this process, such tensor fields (e.g., the gradients of $\hat{\chi}$ and $\zeta$) were decomposed into ``good" and ``bad" parts.
While the ``good" parts can be properly controlled, the ``bad" parts must again be decomposed into ``good" and ``bad" parts.
This led to infinite iterations, which converged in suitable norms to the final desired decompositions.

In this paper, we demonstrate that this cumbersome process can be avoided.
This simplification revolves around the following basic ideas:
\begin{itemize}
\item The exact infinite decompositions of these fields were irrelevant.
The main point is that some decomposition with the necessary estimates exists.

\item One can construct norms to quantify the \emph{existence} of the necessary decompositions, without needing to explicitly specify the decompositions.
\end{itemize}
The natural norm to capture such decompositions is the ``sum norm", described in Section \ref{sec.fol_norms}.
\footnote{This is, in fact, a general construction for arbitrary normed spaces.}
By using these sums norms in the main bootstrap assumptions and argument, one can avoid this process of obtaining explicit infinite decompositions.

More specifically, for argument, we decompose tensor fields as follows:
\begin{itemize}
\item The first part contains terms which can be controlled in an $L^2$-norm along all of $\mc{N}$, with the additional caveat that one can trade a null derivative for a spherical derivative.
A quantitative version of this property can be captured using a special norm (the $N^{0\star}_{t, x}$-norm in Section \ref{sec.fol_norms}).

\item The remaining terms will lack this derivative trading property.
These terms will be controlled by an infinitesimally stronger Besov-type norm on $\mc{N}$.
\end{itemize}
The main idea is to impose an additional bootstrap assumption in the main argument (via the sum norm) stating that such a decomposition exists, with sufficient bounds by the above two norms.
From explicit decompositions, one obtains as before the ``good" terms, which can be controlled in these norms.
However, one no longer needs to decompose again the remaining ``bad" terms, as these can now be handled using the bootstrap assumptions.
From this process, we can immediately recover a strictly improved version of this new bootstrap assumption.

\subsection{Notations} \label{sec.intro_not}

Here, we list basic notational conventions we will use, many of them borrowed from \cite{shao:stt}.
First, given nonnegative real numbers $X, Y, c_1, \dots, c_m$:
\begin{itemize}
\item $X \lesssim_{c_1, \dots, c_m} Y$ means that $X \leq c Y$ for some constant $c > 0$ depending on $c_1, \dots, c_m$.
If no $c_i$'s are given, then the constant $c$ is universal.

\item Similarly, we write $X \simeq_{c_1, \dots, c_m} Y$ to mean that both of the following statements hold: $X \lesssim_{c_1, \dots, c_m} Y$ and $Y \lesssim_{c_1, \dots, c_m} X$.
\end{itemize}
To shorten notations, \emph{we will generally omit the dependence of constants (i.e., the $c_i$'s in the above) in inequalities within proofs of statements.}

Next, we will use the following symbols to denote various constants.
These constants will be used throughout the paper to represent ranks of tensor fields as well as parameters in various regularity conditions.
\begin{itemize}
\item Let $r, r_1, r_2$ and $l, l_1, l_2$ denote non-negative integers.

\item Let $C > 1$ and $B > 0$ denote real constants.

\item Let $N > 0$ denote an integer constant.
\end{itemize}

Finally, for a general manifold $M$:
\begin{itemize}
\item Let $\mc{C}^\infty M$ denote the space of all smooth real-valued functions on $M$.

\item For a vector bundle $\mc{V}$ over $M$ and $z \in M$, we let $\mc{V}_z$ denote the fiber of $\mc{V}$ at $z$.
Moreover, we let $\mc{C}^\infty \mc{V}$ be the space of all smooth sections of $\mc{V}$.

\item Let $T^r_l M$ denote the \emph{tensor bundle} over $M$ of rank $(r, l)$, for which the fiber $( T^r_l M )_z$ at any $z \in M$ is the space of tensors of rank $(r, l)$ at $z$.
\footnote{Here, $r$ is the contravariant rank, and $l$ is the covariant rank.}
Thus, $\mc{C}^\infty T^r_l M$ is the space of smooth tensor fields of rank $(r, l)$ on $M$.
\end{itemize}
We will often use standard index notation to describe tensor and tensor fields on $M$.
Indices, denoted using lowercase Latin letters, will be with respect to fixed frames and coframes.
In accordance with Einstein summation notation, repeated indices indicate summations over all allowable index values.

\begin{ackn}
The first author was supported by NSERC grants 488916 and 489103, as well as a Sloan Fellowship.
The authors also wish to thank Sergiu Klainerman for interesting and helpful conversations that contributed to this report.
\end{ackn}

\section{Geometric Preliminaries} \label{sec.geom}

In this section, we review some background involving the analysis of tensor fields on a $2$-dimensional Riemannian manifold.
The contents here briefly summarize many of the basic developments detailed in \cite{shao:stt}.
Throughout, we let $\mc{S}$ be a surface diffeomorphic to $\Sph^2$, with $h$ a Riemannian metric on $\mc{S}$.

\subsection{Riemannian Structures} \label{sec.geom_riem}

By the conventions in Section \ref{sec.intro_not}, we can think of $h$ as an element of $\mc{C}^\infty T^0_2 \mc{S}$.
Let $h^{-1} \in \mc{C}^\infty T^2_0 \mc{S}$ denote the metric dual of $h$.
As usual, within index notation, $h^{-1}$ is written as simply $h$, but with superscript indices.
Since $\mc{S}$ is compact, we can fix an orientation for $\mc{S}$.
As a result, $h$ and this orientation induce a volume form $\omega \in \mc{C}^\infty T^0_2 \mc{S}$ on $\mc{S}$.

Recall $h$ and $h^{-1}$ define pointwise tensorial inner products and norms on $\mc{S}$.
More specifically, for any $F, G \in \mc{C}^\infty T^r_l \mc{S}$, we define
\[ \langle F, G \rangle = h_{a_1 b_1} \dots h_{a_r b_r} h^{c_1 d_1} \dots h^{c_l d_l} F^{a_1 \dots a_r}{}_{c_1 \dots c_l} G^{b_1 \dots b_r}{}_{d_1 \dots d_l} \in C^\infty \mc{S} \text{,} \]
i.e., the bundle metric on $T^r_l \mc{S}$ induced by $h$.
We also define the pointwise norm:
\footnote{In the scalar case $r = l = 0$, the inner product $\langle \cdot, \cdot \rangle$ is simply multiplication of functions, and the norm $| \cdot |$ is the absolute value.
In particular, these are independent of $h$.}
\[ | F | = \langle F, F \rangle^\frac{1}{2} \text{.} \]
We can now use $h$ and $\omega$ to define standard integral norms:
\[ \| F \|_{ L^q_x }^q = \int_\mc{S} | F |^q d \omega \text{,} \qquad \| F \|_{ L^\infty_x } = \sup_{ \mc{S} } | F | \text{,} \qquad q \in [1, \infty) \text{.} \]

Following standard conventions, we let $\nabla$ and $\lapl$ denote the Levi-Civita connection and the Bochner Laplacian with respect to $h$, respectively.
Higher-order differentials are defined iteratively: $\nabla^{k+1} = \nabla \nabla^k$ for any positive integer $k$.
Furthermore, we let $\mc{K} \in \mc{C}^\infty \mc{S}$ denote the Gauss curvature of $(\mc{S}, h)$.

Next, we recall the symmetric Hodge operators on spherical surfaces, as defined in \cite{chr_kl:stb_mink, kl_rod:cg}.
We begin by defining the vector bundles on which these operators act.
The rank-$0$ and rank-$1$ bundles are defined as
\[ H_0 \mc{S} = \mc{C}^\infty \mc{S} \otimes \C \text{,} \qquad H_1 \mc{S} = \mc{C}^\infty T^0_1 \mc{S} \text{.} \]
Note the sections of $H_0 \mc{S}$ are precisely the complex-valued smooth scalar functions on $\mc{S}$.
In addition, we define $H_2 \mc{S}$ to be the vector bundle over $\mc{S}$ of all covariant symmetric $h$-traceless horizontal $2$-tensors on $\mc{S}$.
\footnote{In index notation, $A \in \mc{C}^\infty T^0_2 \mc{S}$ is in $\mc{C}^\infty H_2 \mc{S}$ iff $A_{ba} = A_{ab}$ and $h^{ab} A_{ab} \equiv 0$.}

\begin{remark}
Note that $H_0 \mc{S}$ and $H_1 \mc{S}$ are independent of $h$, while $H_2 \mc{S}$ is not.
\end{remark}

The symmetric Hodge operators are defined as follows:
\begin{alignat*}{3}
\mc{D}_1 &: \mc{C}^\infty H_1 \mc{S} \rightarrow \mc{C}^\infty H_0 \mc{S} \text{,} &\qquad \mc{D}_1 X &= h^{ab} \nabla_a X_b - i \cdot \omega^{ab} \nabla_a X_b \text{,} \\
\mc{D}_2 &: \mc{C}^\infty H_2 \mc{S} \rightarrow \mc{C}^\infty H_1 \mc{S} \text{,} &\qquad ( \mc{D}_2 X )_a &= h^{bc} \nabla_b X_{ac} \text{,} \\
\mc{D}_1^\ast &: \mc{C}^\infty H_0 \mc{S} \rightarrow \mc{C}^\infty H_1 \mc{S} \text{,} &\qquad ( \mc{D}_1^\ast X )_a &= - \nabla_a ( \real X ) - \omega_a{}^c \nabla_c ( \imag X ) \text{,} \\
\mc{D}_2^\ast &: \mc{C}^\infty H_1 \mc{S} \rightarrow \mc{C}^\infty H_2 \mc{S} \text{,} &\qquad -2 ( \mc{D}_2^\ast X )_{ab} &= \nabla_a X_b + \nabla_b X_a - h_{ab} h^{cd} \nabla_c X_d \text{.}
\end{alignat*}
Direct computations show that the $\mc{D}_i^\ast$'s are the $L^2$-adjoints of the $\mc{D}_i$'s (with respect to $h$).
In addition, we can compute the following identities:
\begin{alignat}{3}
\label{eq.hodge_sq} \mc{D}_1 \mc{D}_1^\ast &= - \lapl \text{,} &\qquad \mc{D}_1^\ast \mc{D}_1 &= - \lapl + \mc{K} \text{,} \\
\notag \mc{D}_2 \mc{D}_2^\ast &= - \frac{1}{2} \lapl - \frac{1}{2} \mc{K} \text{,} &\qquad \mc{D}_2^\ast \mc{D}_2 &= - \frac{1}{2} \lapl + \mc{K} \text{.}
\end{alignat}

Finally, we review some basic formulas for rescaling $h$.
Fix $\lambda \in \R$, and consider the Riemannian metric $\bar{h} = e^{2 \lambda} h \in \mc{C}^\infty T^0_2 \mc{S}$.
We will use the following notational conventions: geometric objects and norms defined with respect to $\bar{h}$ will be denoted with a ``bar" over the symbol.
For example, $\bar{\omega} = e^{2 \lambda} \omega$ denotes the volume form associated with $\bar{h}$ (with the same orientation).
In terms of index notations,
\[ \bar{h}_{ab} = e^{2 \lambda} h_{ab} \text{,} \qquad \bar{\omega}_{ab} = e^{2 \lambda} \omega_{ab} \text{,} \qquad \bar{h}^{ab} = e^{-2 \lambda} h^{ab} \text{,} \qquad \bar{\omega}^{ab} = e^{-2 \lambda} \omega^{ab} \text{.} \]
Moreover, such rescalings leave the Levi-Civita connection unchanged,
\[ \bar{\nabla} F = \nabla F \text{,} \qquad F \in \mc{C}^\infty T^r_l \mc{S} \text{,} \]
while it rescales the curvature by a constant factor,
\[ \bar{\mc{K}} = e^{-2 \lambda} \mc{K} \text{.} \]
The Hodge operators for $h$ and $\bar{h}$ also obey similar formulas:
\[ \bar{\mc{D}}_1 = e^{-2 \lambda} \mc{D}_1 \text{,} \qquad \bar{\mc{D}}_2 = e^{-2 \lambda} \mc{D}_2 \text{,} \qquad \bar{\mc{D}}_1^\ast = \mc{D}_1^\ast \text{.} \]

\subsection{Geometric Littlewood-Paley Theory} \label{sec.geom_glp}

We next review the geometric invariant Littlewood-Paley (abbreviated L-P) theory, based on spectral decompositions of the (Bochner) Laplacian.
For additional discussions, see \cite{shao:stt}.

\begin{remark}
An alternative approach is to use the geometric L-P theory of \cite{kl_rod:glp}, based on the heat flow.
This was done in previous works involving null cones with bounded curvature flux, cf. \cite{kl_rod:cg, parl:bdc, shao:bdc_nv, wang:cg, wang:cgp}.
However, the spectral version, whenever applicable, is much easier to rigorously construct and utilize.
\end{remark}

For technical purposes, we consider the Hilbert space $L^2 T^r_l \mc{S}$, defined as the completion of $\mc{C}^\infty T^r_l \mc{S}$ with respect to the $L^2$-norm on $(\mc{S}, h)$.
Consider $-\lapl$ as a positive self-adjoint unbounded operator on $L^2 T^r_l \mc{S}$, which has a spectral decomposition
\[ -\lapl = \int_0^\infty \lambda \cdot d E_\lambda \text{.} \]

As in \cite{shao:stt}, the spectral L-P operators can be constructed as follows:
\begin{itemize}
\item Fix a function $\varsigma \in \mc{C}^\infty \R$, supported in the region $1/2 \leq | \xi | \leq 2$, satisfying
\[ \sum_{k \in \Z} \varsigma ( 2^{-2k} \xi ) = 1 \text{,} \qquad \xi \in \R \setminus \{ 0 \} \text{.} \]

\item For each $k \in \Z$, we define the L-P operators on $L^2 T^r_l \mc{S}$ by
\[ P_k = \varsigma ( - 2^{-2k} \lapl ) \text{,} \qquad P_- = \chi_{ \{ 0 \} } ( -\lapl ) \text{,} \]
where $\chi_{ \{0\} }$ is the characteristic function for the set containing only $0$.
In particular, $P_-$ is precisely the $L^2$-projection onto the kernel of $\lapl$.

\item Given any $k \in \Z$, we can define (in the strong operator topology)
\[ P_{< k} = P_- + \sum_{l < k} P_l \text{.} \]
In addition, letting $I$ denote the identity operator on $L^2 T^r_l \mc{S}$, we have
\[ I = P_{< 0} + \sum_{k \geq 0} P_k \text{.} \]
\end{itemize}
These L-P operators are fully invariant and tensorial, and they satisfy many of the same properties as the classical L-P operators on Euclidean spaces (at least with respect to the $L^2$-norm).
For details, see \cite[Sect. 2.2]{shao:stt}.

In this paper, we will not need to deal directly with these L-P operators.
Instead, we need them in order to define geometric, tensorial Besov norms.
Given $a \in [1, \infty)$ and $s \in \R$, we define for each $F \in \mc{C}^\infty T^r_l \mc{S}$ the norms
\begin{align*}
\| F \|_{ B^{a, s}_{\ell, x} }^a &= \sum_{k \geq 0} 2^{ask} \| P_k F \|_{ L^2_x }^a + \| P_{< 0} F \|_{ L^2_x }^a \text{,} \\
\| F \|_{ B^{\infty, s}_{\ell, x} } &= \max \left( \sup_{k \geq 0} 2^{sk} \| P_k F \|_{ L^2_x }, \| P_{< 0} F \|_{ L^2_x } \right) \text{.}
\end{align*}
These are the direct analogues of the standard $B^s_{2, a}$-norms in Euclidean space.
As we are mainly interested in the case $a = 1$, we define the shorthand
\[ \| F \|_{ B^s_x } = \| F \|_{ B^{1, s}_{\ell, x} } \text{.} \]

Next, given $s \in \R$, we can define the standard fractional Sobolev norms
\[ \| F \|_{ H^s_x } = \| \Lambda^s F \|_{ L^2_x } \text{,} \qquad F \in \mc{C}^\infty T^r_l \mc{S} \text{,} \]
where $\Lambda^s = ( I - \lapl )^\frac{s}{2}$.
We can relate these to the aforementioned Besov norms.

\begin{proposition} \label{thm.sobolev}
If $s \in \R$ and $F \in \mc{C}^\infty T^r_l \mc{N}$, then
\begin{equation} \label{eq.sobolev} \| F \|_{ H^s_x } \simeq_s \| F \|_{ B^{2, s}_{\ell, x} } \text{.} \end{equation}
\end{proposition}

\begin{proof}
This follows from the spectral properties of $\Lambda^s$.
\end{proof}

\subsection{Regularity Conditions} \label{sec.geom_reg}

We now discuss the regularity conditions that we will impose on $(\mc{S}, h)$.
The point is that all such $(\mc{S}, h)$ satisfying these properties can be controlled in a uniform way.
One example is Sobolev-type estimates, which can be applied with a common Sobolev constant for all such $(\mc{S}, h)$.

We will use the same conditions that were defined in \cite[Sect. 2.4]{shao:stt}.

\begin{definition} \label{def.ass_r0}
$(\mc{S}, h)$ satisfies \ass{r0}{C, N}, with data $\{ U_i, \varphi_i, \eta_i \}_{i = 1}^N$, iff:
\begin{itemize}
\item The area $| \mc{S} |$ of $(\mc{S}, h)$ satisfies
\[ C^{-1} \leq | \mc{S} | \leq C \text{.} \]

\item The $(U_i, \varphi_i)$'s, where $1 \leq i \leq N$, are local coordinate systems on $\mc{S}$ that cover $\mc{S}$.
Moreover, each $\varphi_i (U_i)$ is a bounded neighborhood in $\R^2$.

\item The $\eta_i$'s form a partition of unity of $\mc{S}$, subordinate to the $U_i$'s, such that
\[ 0 \leq \eta_i \leq 1 \text{,} \qquad | \partial^i_a \eta_i | \leq C \text{,} \qquad a, b \in \{ 1, 2 \} \text{,} \]
for each $1 \leq i \leq N$, where $\partial^i_1, \partial^i_2$ denote the $\varphi_i$-coordinate vector fields.

\item For each $1 \leq i \leq N$, we have on $U_i$ the uniform positivity property
\[ C^{-1} | \xi |^2 \leq \sum_{a, b = 1}^2 h_{ab} \xi^a \xi^b \leq C | \xi |^2 \text{,} \qquad \xi \in \R^2 \text{,} \]
where we have indexed with respect to the $\varphi_i$-coordinate system on $U_i$.
\end{itemize}
\end{definition}

\begin{definition} \label{def.ass_r1}
$(\mc{S}, h)$ satisfies \ass{r1}{C, N}, with data $\{ U_i, \varphi_i, \eta_i, \tilde{\eta}_i, e^i \}_{i = 1}^N$, iff:
\begin{itemize}
\item $(\mc{S}, h)$ satisfies \ass{r0}{C, N}, with data $\{ U_i, \varphi_i, \eta_i \}_{i = 1}^N$.

\item For any $1 \leq i \leq N$, we have that $e^i = ( e^i_1, e^i_2 ) \in \mc{C}^\infty T^1_0 \mc{S} \times \mc{C}^\infty T^1_0 \mc{S}$ forms an orthonormal frame on $U_i$ and satisfies the estimates
\[ \| \nabla e^i_a \|_{ L^4_x } \leq C \text{,} \qquad a \in \{ 1, 2 \} \text{.} \]

\item For any $1 \leq i \leq N$, we have that $\tilde{\eta}_i \in \mc{C}^\infty \mc{S}$ is supported within $U_i$, is identically $1$ on the support of $\eta_i$, and satisfies the estimates
\[ 0 \leq \tilde{\eta}_i \leq 1 \text{,} \qquad | \partial^i_a \tilde{\eta}_i | \leq C \text{,} \qquad a \in \{ 1, 2 \} \text{.} \]

\item For each $1 \leq i \leq N$, the area density
\[ \vartheta_i = \sqrt{ h_{11} h_{22} - h_{12}^2 } \in \mc{C}^\infty U_i \text{,} \]
where we have indexed with respect to the $\varphi_i$-coordinates, satisfies
\[ \| \nabla \vartheta_i \|_{ L^2_x } \leq C \text{.} \]
\end{itemize}
\end{definition}

\begin{definition} \label{def.ass_r2}
$(\mc{S}, h)$ satisfies \ass{r2}{C, N}, with data $\{ U_i, \varphi_i, \eta_i, \tilde{\eta}_i, e^i \}_{i = 1}^N$, iff:
\begin{itemize}
\item $(\mc{S}, h)$ satisfies \ass{r1}{C, N}, with data $\{ U_i, \varphi_i, \eta_i, \tilde{\eta}_i, e^i \}_{i = 1}^N$.

\item For each $1 \leq i \leq N$, the $\varphi_i$-coordinate vector fields $\partial^i_1, \partial^i_2$ satisfy
\[ \| \nabla \partial^i_a \|_{ L^2_x } \leq C \text{,} \qquad a \in \{ 1, 2 \} \text{.} \]

\item For each $1 \leq i \leq N$, the second \emph{coordinate} derivatives of $\eta_i$ satisfy
\[ | \partial^i_a \partial^i_b \eta_i | \leq C \text{,} \qquad a, b \in \{ 1, 2 \} \text{.} \]
\end{itemize}
\end{definition}

The conditions in the \ass{r0}{}, \ass{r1}{}, and \ass{r2}{} assumptions were required explicitly in \cite{shao:stt} for various estimates that we will need here.
On the other hand, here we will not encounter most of these conditions directly.
We include their precise statements in order to provide a similarly precise statement of the main results of this paper.

\begin{remark}
Since $\mc{S}$ is compact, $(\mc{S}, h)$ trivially satisfies \ass{r2}{C, N} for some $C$, $N$.
\end{remark}

We now list some estimates resulting from these conditions.
The first batch involves tensorial Sobolev-type estimates, which were proved in \cite[Sect. 2.5]{shao:stt}. 
\footnote{See also \cite[Cor. 2.4]{kl_rod:glp}, on which the proofs in \cite{shao:stt} were based.}

\begin{proposition} \label{thm.gns_ineq}
Suppose $(\mc{S}, h)$ satisfies \ass{r0}{C, N}, and let $F \in \mc{C}^\infty T^r_l \mc{S}$.
\begin{itemize}
\item If $q \in (2, \infty)$, then
\begin{align}
\label{eq.gns_1} \| F \|_{ L^q_x } &\lesssim_{C, N, q} \| \nabla F \|_{ L^2_x }^{ 1 - \frac{2}{q} } \| F \|_{ L^2_x }^\frac{2}{q} + \| F \|_{ L^2_x } \text{,} \\
\label{eq.gns_1s} \| F \|_{ L^\infty_x } &\lesssim_{C, N, q} \| \nabla F \|_{ L^q_x }^\frac{2}{q} \| F \|_{ L^q_x }^{ 1 - \frac{2}{q} } + \| F \|_{ L^q_x } \text{.}
\end{align}

\item Moreover, the following estimate holds:
\begin{equation} \label{eq.gns_2} \| F \|_{ L^\infty_x } \lesssim_{C, N} \| \nabla^2 F \|_{ L^2_x }^\frac{1}{2} \| F \|_{ L^2_x }^\frac{1}{2} + \| F \|_{ L^2_x } \text{.} \end{equation}
\end{itemize}
\end{proposition}

With the \ass{r1}{} condition, we can also establish certain sharp Sobolev embeddings involving fractional derivatives.
For details, see \cite[Sect. 3.5]{shao:stt}.

\begin{proposition} \label{thm.sob_frac_sh}
Assume $(\mc{S}, h)$ satisfies \ass{r1}{C, N}.
If $F \in \mc{C}^\infty T^r_l \mc{S}$, then
\begin{equation} \label{eq.sob_frac_sh} \| F \|_{ L^4_x } \lesssim_{C, N, r, l} \| F \|_{ H^{1/2}_x } \text{.} \end{equation}
\end{proposition}

\subsection{Curvature Regularity} \label{sec.geom_curv}

Besides \ass{r0}{}, \ass{r1}{}, and \ass{r2}{}, we require one more regularity assumption, introduced in \cite[Sect. 6.1]{shao:stt}, related to the curvature of $\mc{S}$.

\begin{definition} \label{def.ass_k}
$(\mc{S}, h)$ satisfies \ass{k}{C, D}, with data $(f, W, V)$, iff:
\begin{itemize}
\item $f \in \mc{C}^\infty \mc{S}$ satisfies, for any $x \in \mc{S}$, the bounds
\[ C^{-1} \leq f |_x \leq C \text{.} \]

\item $V \in \mc{C}^\infty T^0_1 \mc{S}$ and $W \in \mc{C}^\infty \mc{S}$ satisfy
\[ \| V \|_{ H^{1/2}_x } \leq D \text{,} \qquad \| W \|_{ L^2_x } \leq D \text{.} \]

\item $\mc{K}$ can be decomposed in the form
\[ \mc{K} - f = h^{ab} \nabla_a V_b + W \text{.} \]
\end{itemize}
Moreover, we will only consider the case in which $D$ is very small.
\end{definition}

\begin{remark}
In other words, $\mc{K}$ is comparable to $1$, except for a ``good" error term $W$ that is $L^2$-bounded, and a ``bad" error term, which is not $L^2$-bounded but can be expressed as a divergence of an $H^{1/2}$-controlled $1$-form $V$.
\end{remark}

Although the \ass{k}{} condition places only very weak restrictions on the curvature, it is sufficient to establish several elliptic estimates.

\begin{proposition} \label{thm.div_curl_est}
Assume $(\mc{S}, h)$ satisfies \ass{r1}{C, N} and \ass{k}{C, D}, with $D \ll 1$ sufficiently small.
Then, for any $F \in \mc{C}^\infty T^r_{l+1} \mc{S}$,
\footnote{Here, $h^{ab} \nabla_a F_b \in \mc{C}^\infty T^r_l \mc{S}$ refers to the metric contraction of $\nabla F$ in the derivative component and a fixed covariant component of $F$.  The expression $\omega^{ab} \nabla_a F_b$ is defined similarly.}
\begin{equation} \label{eq.div_curl_est} \| \nabla F \|_{ L^2_x } \lesssim_{C, N, r, l} \| h^{ab} \nabla_a F_b \|_{ L^2_x } + \| \omega^{ab} \nabla_a F_b \|_{ L^2_x } + \| F \|_{ L^2_x } \text{,} \end{equation}
\end{proposition}

\begin{proof}
See \cite[Sect. 6.1]{shao:stt}.
\end{proof}

Next, we derive similar elliptic estimates for the symmetric Hodge operators.

\begin{proposition} \label{thm.hodge_est}
Assume $(\mc{S}, h)$ satisfies \ass{r1}{C, N} and \ass{k}{C, D}, with $D \ll 1$ sufficiently small.
Then, the following Hodge-elliptic estimates hold:
\begin{itemize}
\item If $X \in \mc{C}^\infty H_1 \mc{S}$, then
\begin{equation} \label{eq.hodge_est_D1} \| \nabla X \|_{ L^2_x } + \| X \|_{ L^2_x } \lesssim_{C, N} \| \mc{D}_1 X \|_{ L^2_x } \text{.} \end{equation}

\item If $X \in \mc{C}^\infty H_2 \mc{S}$, then
\begin{equation} \label{eq.hodge_est_D2} \| \nabla X \|_{ L^2_x } + \| X \|_{ L^2_x } \lesssim_{C, N} \| \mc{D}_2 X \|_{ L^2_x } \text{.} \end{equation}

\item If $X \in \mc{C}^\infty H_0 \mc{S}$, then
\begin{equation} \label{eq.hodge_est_D1a} \| \nabla X \|_{ L^2_x } \simeq \| \mc{D}_1^\ast X \|_{ L^2_x } \text{.} \end{equation}

\item If $X \in \mc{C}^\infty H_1 \mc{S}$, then
\begin{equation} \label{eq.hodge_est_D2a} \| \nabla X \|_{ L^2_x } \lesssim_{C, N} \| \mc{D}_2^\ast X \|_{ L^2_x } + \| X \|_{ L^2_x } \text{.} \end{equation}
\end{itemize}
\end{proposition}

\begin{proof}
See \cite[Sect. 6.2]{shao:stt}.
\end{proof}

Assuming for the moment the setting of Proposition \ref{thm.hodge_est}, then \eqref{eq.hodge_est_D1} and \eqref{eq.hodge_est_D2} imply that $\mc{D}_1$ and $\mc{D}_2$ are one-to-one, and that both operators have $L^2$-bounded inverses.
Furthermore, we can extend these inverses to $L^2$-bounded operators
\[ \mc{D}_i^{-1} : \mc{C}^\infty H_{i-1} \mc{S} \rightarrow \mc{C}^\infty H_i \mc{S} \text{,} \qquad i \in \{ 1, 2 \} \text{,} \]
by defining $\mc{D}_i^{-1} X$ to be the (actual) inverse of $\mc{D}_i$ acting on the $L^2$-orthogonal projection of $X$ onto the (closed) range of $\mc{D}_i$.
If we let $\mc{P}_i$ denote this $L^2$-projection onto the range of $\mc{D}_i$, then by the above definitions,
\[ \mc{D}_i^{-1} \mc{D}_i = I \text{,} \qquad \mc{D}_i \mc{D}_i^{-1} = \mc{P}_i \text{.} \]

One can also use the above to partially invert the $\mc{D}_i^\ast$'s.
Since $\mc{D}_i$ is injective, then $\mc{D}_i^\ast$ is surjective, and its inverse image of any element of $\mc{C}^\infty H_i \mc{S}$ is a coset of the nullspace of $\mc{D}_i^\ast$.
\footnote{To be fully rigorous, we must invoke some functional analytic technicalities and consider the $\mc{D}_i$'s and $\mc{D}_i^\ast$'s as densely defined unbounded operators on the appropriate $L^2$-spaces.}
Since the nullspace of $\mc{D}_i^\ast$ is the orthogonal complement of the range of $\mc{D}_i$, then we can define $\mc{D}_i^{\ast -1} X$ to be the unique element of the corresponding inverse image that is in the range of $\mc{D}_i$.
In summary, we have
\[ \mc{D}_i^{\ast -1} \mc{D}_i^\ast = \mc{P}_i \text{,} \qquad \mc{D}_i^\ast \mc{D}_i^{\ast -1} = I \text{.} \]

For further details on the inverse Hodge operators, see \cite[Sect. 6.2]{shao:stt}.
For our purposes, we will need the following estimates, which follow from Proposition \ref{thm.hodge_est}.

\begin{proposition} \label{thm.hodge_inv_est}
Assume $(\mc{S}, h)$ satisfies \ass{r1}{C, N} and \ass{k}{C, D}, with $D \ll 1$ sufficiently small.
If $\mc{D}$ denotes any one of the operators $\mc{D}_1$, $\mc{D}_2$, $\mc{D}_1^\ast$, $\mc{D}_2^\ast$, and if $X$ is a smooth section of the appropriate Hodge bundle on $\mc{S}$, then
\begin{equation} \label{eq.hodge_inv_est} \| \nabla \mc{D}^{-1} X \|_{ L^2_x } + \| \mc{D}^{-1} X \|_{ L^2_x } \lesssim_{C, N} \| X \|_{ L^2_x } \text{.} \end{equation}
\end{proposition}

\begin{proof}
See \cite[Sect. 6.2]{shao:stt}.
\end{proof}

\section{Spherical Foliations} \label{sec.fol}

In Section \ref{sec.geom}, our discussions were restricted to a single manifold $\mc{S}$ that was diffeomorphic to $\Sph^2$.
Here, we will discuss \emph{one-parameter foliations} of such $2$-spheres.
More specifically, our background setting will be the product
\[ \mc{N} = [0, \delta] \times \mc{S} \text{,} \qquad \delta > 0 \text{.} \]
Let $t$ be the natural projection onto the first component:
\[ t: \mc{N} \rightarrow [0, \delta] \text{,} \qquad t (\tau, x) = \tau \text{.} \]
Throughout this section, we will let $\tau$ denote an arbitrary element of $[0, \delta]$.
Given such a $\tau$, we let $\mc{S}_\tau$ denote the level set $\mc{S}_\tau = \{ \tau \} \times \mc{S}$ of $t$:

Although we will work only on the $3$-manifold with boundary $\mc{N}$, we will always implicitly assume that all our objects can be smoothly extended beyond the boundaries $\mc{S}_0$ and $\mc{S}_\delta$.
Thus, our full setting, on which all our objects of analysis are defined, is the extended foliation $\mc{N}^\prime = (-\varepsilon, \delta + \varepsilon) \times \mc{S}$, for some $\varepsilon > 0$.

\begin{remark}
We consider only intervals of the form $[0, \delta]$ here in order to make some estimates easier to state.
However, the formalisms throughout this section can be easily adapted if $[0, \delta]$ is replaced by any other closed, open, or half-open interval.
We will consider such adaptations throughout Section \ref{sec.nc}.
\end{remark}

We note that most of the upcoming formalisms of this section will be special cases of those discussed in \cite{shao:stt}.
For further details on this formalism, see the more detailed development of the material in \cite{shao:stt}.

\subsection{Horizontal Structures} \label{sec.fol_hor}

Given our foliation $\mc{N}$, the first task is to construct objects on $\mc{N}$ which represent smoothly varying aggregations of objects on the $\mc{S}_\tau$'s.
For this, we begin by defining the diffeomorphisms
\[ \Xi_\tau: \mc{S}_\tau \leftrightarrow \mc{S} \text{,} \qquad \Xi_\tau (\tau, x) = x \text{,} \]
which identify $\mc{S}_\tau$ with $\mc{S}$.
From $\Xi_\tau$, we can construct natural identifications
\[ \Xi_\tau^\ast: \mc{C}^\infty T^r_l \mc{S}_\tau \leftrightarrow \mc{C}^\infty T^r_l \mc{S} \text{.} \]
We will use these identifications repeatedly in our basic constructions.

As in \cite{shao:stt}, we let $\ul{T}^r_l \mc{N}$ denote the \emph{horizontal tensor bundle} of rank $(r, l)$, i.e., the vector bundle over $\mc{N}$ for which the fiber at each $(\tau, x) \in \mc{N}$ is
\[ ( \ul{T}^r_l \mc{N} )_{ (\tau, x) } = ( T^r_l \mc{S}_\tau )_{ (\tau, x) } \text{.} \]
Note in particular that $\ul{T}^0_0 \mc{N}$ can be identified with $\mc{C}^\infty \mc{N}$.
A section $A \in \mc{C}^\infty \ul{T}^r_l \mc{N}$ is called a \emph{horizontal tensor field}.
Given $\tau$, we will let
\[ A [\tau] = \Xi_\tau^\ast ( A | \mc{S}_\tau ) \in \mc{C}^\infty T^r_l \mc{S} \]
denote the tensor field on $\mc{S}$ corresponding to the restriction of $A$ to $\mc{S}_\tau$.

We impose on $\mc{N}$ a \emph{horizontal metric} $\gamma \in \mc{C}^\infty \ul{T}^0_2 \mc{N}$, such that each $\gamma [\tau]$ defines a Riemannian metric on $\mc{S}$.
We also define $\gamma^{-1} \in \mc{C}^\infty \ul{T}^2_0 \mc{N}$ so that each $\gamma^{-1} [\tau]$ is the dual $( \gamma [\tau] )^{-1}$ to $\gamma [\tau]$.
A fixed orientation on $\mc{S}$ induces an orientation on each $\mc{S}_\tau$.
From this, we can generate a \emph{horizontal volume form} $\epsilon \in \mc{C}^\infty \ul{T}^0_2 \mc{N}$, defined such that each $\epsilon [\tau]$ represents the volume form on $\mc{S}$ associated with $\gamma [\tau]$.

Families of objects and operators on $\mc{S}$ parametrized by $\tau$ can be aggregated into corresponding horizontal objects on $\mc{N}$.
Relevant examples include the following:
\begin{itemize}
\item The fields $\gamma$, $\gamma^{-1}$, and $\epsilon$ above are the most basic examples.

\item Tensor products can be similarly aggregated into an analogous product of horizontal fields.
Given $\Psi_i \in \mc{C}^\infty \ul{T}^{r_i}_{l_i} \mc{N}$, where $i \in \{ 1, 2 \}$, we define
\[ \Psi_1 \otimes \Psi_2 \in \mc{C}^\infty \ul{T}^{r_1 + r_2}_{l_1 + l_2} \mc{N} \text{,} \qquad ( \Psi_1 \otimes \Psi_2 ) [\tau] = \Psi_1 [\tau] \otimes \Psi_2 [\tau] \text{.} \]

\item The pointwise tensor norms $| \cdot |$ with respect to the $\gamma [\tau]$'s lift to a corresponding pointwise norm for horizontal tensors with respect to $\gamma$:
\[ | \cdot | : \mc{C}^\infty \ul{T}^r_l \mc{N} \rightarrow \mc{C}^\infty \mc{N} \text{,} \qquad | \Psi | [\tau] = | \Psi [\tau] | \text{.} \]
The pointwise inner products also lift analogously.

\item The Levi-Civita connections $\nabla$ with respect to the $\gamma [\tau]$'s can be similarly aggregated into a single \emph{horizontal covariant differential} operator
\[ \nabla : \mc{C}^\infty \ul{T}^r_l \mc{N} \rightarrow \mc{C}^\infty \ul{T}^r_{l+1} \mc{N} \text{.} \]
The operators $\nabla^k$, $k > 1$, are defined similarly, as is the Laplacian:
\[ \lapl : \mc{C}^\infty \ul{T}^r_l \mc{N} \rightarrow \mc{C}^\infty \ul{T}^r_l \mc{N} \text{.} \]

\item The horizontal \emph{Gauss curvature} is the function $\mc{K} \in \mc{C}^\infty \mc{N}$ such that each $\mc{K} [\tau]$ is precisely the Gauss curvature associated with $\gamma [\tau]$.

\item We can also aggregate the geometric L-P operators $P_k$, $P_{< k}$, $P_-$, so that they act on $\mc{C}^\infty \ul{T}^r_l \mc{N}$, with respect to $\gamma [\tau]$ on each $\mc{S}_\tau$.
\end{itemize}

Relations involving horizontal tensors are sometimes more easily described using index notation.
We will use the same indexing conventions as one would use for tensors on $\mc{S}$ or the $\mc{S}_\tau$'s.
We will use lowercase Latin indices to denote components of a horizontal tensor field, with repeated indices indicating summations.

Finally, the Hodge bundles defined in Section \ref{sec.geom_riem} can be lifted to horizontal objects on $\mc{N}$.
Let $\ul{H}_i \mc{N}$, where $i \in \{ 0, 1, 2 \}$, denote the natural vector bundle over $\mc{N}$, for which the fiber at each $(\tau, x) \in \mc{N}$ is $( \ul{H}_i \mc{N} )_{ (\tau, x) } = ( H_i \mc{S}_\tau )_x$.
Furthermore, for $j \in \{ 1, 2 \}$, we define the aggregated Hodge operators
\[ \mc{D}_j : \mc{C}^\infty \ul{H}_j \mc{N} \rightarrow \mc{C}^\infty \ul{H}_{j-1} \mc{N} \text{,} \qquad \mc{D}_j^\ast : \mc{C}^\infty \ul{H}_{j-1} \mc{N} \rightarrow \mc{C}^\infty \ul{H}_j \mc{N} \text{,} \]
to behave like the corresponding Hodge operators on each $( \mc{S}, \gamma [\tau] )$.

\subsection{Evolution} \label{sec.fol_ev}

Given $A \in \mc{C}^\infty \ul{T}^r_l \mc{N}$, we define its \emph{vertical Lie derivative} as
\footnote{Note that the operator $\mf{L}_t$ is independent of $\gamma$ and $\epsilon$.}
\[ \mf{L}_t A \in \mc{C}^\infty \ul{T}^r_l \mc{N} \text{,} \qquad \mf{L}_t A [\tau] = \lim_{ \tau^\prime \rightarrow \tau } \frac{ A [\tau^\prime] - A [\tau] }{ \tau^\prime - \tau } \text{.} \]
In particular, we define the \emph{second fundamental form}
\[ k = \frac{1}{2} \mf{L}_t \gamma \in \mc{C}^\infty \ul{T}^0_2 \mc{N} \text{,} \]
which describes the evolution of the geometries of the $\mc{S}_\tau$'s.
Of particular importance is the \emph{mean curvature}, or \emph{expansion}, of $k$:
\[ \trace k = \gamma^{ab} k_{ab} \in C^\infty \mc{N} \text{.} \]

\begin{remark}
$\mf{L}_t$ can alternately be defined as the (standard) Lie derivative with respect to the lift of the vector field $d/dt$ on $[0, \delta]$ to $\mc{N}$.
\end{remark}

Using $\mf{L}_t$ and $k$, we define a corresponding $\gamma$-covariant derivative along the $t$-direction.
Given $\Psi \in \mc{C}^\infty \ul{T}^r_l \mc{N}$, we define $\nabla_t \Psi \in \mc{C}^\infty \ul{T}^r_l \mc{N}$ by
\begin{equation} \label{eq.vderiv_cov} \nabla_t \Psi_{u_1 \dots u_l}^{v_1 \dots v_r} = \mf{L}_t \Psi_{u_1 \dots u_l}^{v_1 \dots v_r} - \sum_{i = 1}^l \gamma^{cd} k_{u_i c} \Psi_{u_1 \hat{d}_i u_l}^{v_1 \dots v_r} + \sum_{j = 1}^r \gamma^{c v_j} k_{cd} \Psi_{u_1 \dots u_l}^{v_1 \hat{d}_j v_r} \text{.} \end{equation}
The notation $u_1 \hat{d}_i u_l$ means $u_1 \dots u_l$, but with $u_i$ replaced by $d$; similar conventions apply for $v_1 \hat{d}_j v_r$.
Note $\nabla_t$ and $\mf{L}_t$ coincide for scalar fields, and $\nabla_t \gamma$, $\nabla_t \gamma^{-1}$, and $\nabla_t \epsilon$ vanish identically.
In addition, we define the following curl of $k$:
\[ \mf{C} \in \mc{C}^\infty \ul{T}^0_3 \mc{N} \text{,} \qquad \mf{C}_{abc} = \nabla_b k_{ac} - \nabla_c k_{ab} \text{,} \]

Next, for $\Psi \in \mc{C}^\infty \ul{T}^r_l \mc{N}$, we define its \emph{covariant integral} $\cint^t_\tau \Psi$ from $\tau$ to be the unique element of $\mc{C}^\infty \ul{T}^r_l \mc{N}$ satisfying both $\nabla_t \cint^t_\tau \Psi = \Psi$ and $( \cint^t_\tau \Psi ) [\tau] \equiv 0$.
In the scalar case $r = l = 0$, this is simply the standard integral with respect to $t$ from $t = \tau$.
These operators $\cint^t_\tau$ can also be defined more explicitly using frames and coframes which are ``$t$-parallel"; see \cite{shao:stt} for details.

\begin{remark}
The above is a slight generalization of concepts found in \cite[Sect. 4.2]{shao:stt}, in which only the case $\tau = 0$ was considered.
In this paper, we require the cases $\tau = 0$ and $\tau = \delta$.
The case of general $\tau$ reduces to the case $\tau = 0$ by a transformation of the $t$-variable (translation and a possible reflection).
Thus, properties that hold in the case $\tau = 0$ generally still hold for other $\tau$.
\end{remark}

Later, we will need the following construction.
Consider smooth functions
\[ \eta_+, \eta_-: [0, \delta] \rightarrow [0, 1] \text{,} \qquad \eta_+ (\tau) = \begin{cases} 0 & 0 \leq \tau \leq \frac{\delta}{3} \text{,} \\ 1 & \frac{2 \delta}{3} \leq \tau \leq \delta \text{,} \end{cases} \qquad \eta_- = 1 - \eta_+ \text{.} \]
In particular, we can define $\eta_+, \eta_-$ as rescalings of the case $\delta = 1$, so that
\[ | \eta_+^\prime (\tau) | \lesssim \delta^{-1} \text{,} \qquad | \eta_-^\prime (\tau) | \lesssim \delta^{-1} \text{.} \]
Now, given $\Psi \in \mc{C}^\infty \ul{T}^r_l \mc{N}$, we define the integral operator
\begin{equation} \label{eq.cint_ex} \cint^t_\star \Psi = \cint^t_0 ( \eta_+ \Psi ) - \cint^t_\delta ( \eta_- \Psi ) \text{.} \end{equation}
Note that we have the identities
\begin{equation} \label{eq.cint_ex_id} \nabla_t \cint^t_\star \Psi = \Psi \text{,} \qquad \cint^t_\star \nabla_t \Psi = \Psi - \cint_0^t ( \eta_+^\prime \Psi ) + \cint_\delta^t ( \eta_-^\prime \Psi ) \text{.} \end{equation}
This above construction defines a particular covariant $t$-antiderivative of $\Psi$ which we will need much later, in the proof of our main result.
More specifically, these integral operators cut off the end-times $\tau = 0$ and $\tau = \delta$ without introducing a factor that blows up as one approaches either end-time.

Finally, given $F \in \mc{C}^\infty T^r_l \mc{S}$, we define the $t$-parallel transport $\mf{p} F \in \mc{C}^\infty \ul{T}^r_l \mc{N}$ of $F$ (from $0$) to be the unique element of $\mc{C}^\infty T^r_l \mc{N}$ satisfying
\[ \nabla_t ( \mf{p} F ) \equiv 0 \text{,} \qquad \mf{p} F [0] = F \text{.} \]
Note in particular that $| \mf{p} F | |_{ (\tau, x) } = | F | |_x$ for any $x \in \mc{S}$.

Of particular importance are the following commutation formulas.

\begin{proposition} \label{thm.comm}
If $\Psi \in \mc{C}^\infty \ul{T}^r_l \mc{N}$ and $F \in \mc{C}^\infty T^r_l \mc{S}$, then
\begin{align}
\label{eq.comm} [ \mf{L}_t, \nabla_a ] \Psi_{u_1 \dots u_l}^{v_1 \dots v_r} &= - \sum_{i = 1}^l \gamma^{cd} ( \nabla_a k_{u_i c} + \nabla_{u_i} k_{ac} - \nabla_c k_{a u_i} ) \Psi_{u_1 \hat{d}_i u_l}^{v_1 \dots v_r} \\
\notag &\qquad + \sum_{j = 1}^r \gamma^{c v_j} ( \nabla_a k_{d c} + \nabla_d k_{ac} - \nabla_c k_{ad} ) \Psi_{u_1 \dots u_l}^{v_1 \hat{d}_j v_r} \text{,} \\
\notag [ \nabla_t, \nabla_a ] \Psi_{u_1 \dots u_l}^{v_1 \dots v_r} &= - \gamma^{cd} k_{ac} \nabla_d \Psi_{u_1 \dots u_l}^{v_1 \dots v_r} - \sum_{i = 1}^l \gamma^{cd} \mf{C}_{a u_i c} \Psi_{u_1 \hat{d}_i u_l}^{v_1 \dots v_r} \\
\notag &\qquad + \sum_{j = 1}^r \gamma^{c v_j} \mf{C}_{adc} \Psi_{u_1 \dots u_l}^{v_1 \hat{d}_j v_r} \text{,} \\
\notag [ \cint^t_0, \nabla_a ] \Psi_{u_1 \dots u_l}^{v_1 \dots v_r} &= \gamma^{cd} \cint^t_0 ( k_{ac} \nabla_d \cint^t_0 \Psi_{u_1 \dots u_l}^{v_1 \dots v_r} ) + \sum_{i = 1}^l \gamma^{cd} \cint^t_0 ( \mf{C}_{a u_i c} \cint^t_0 \Psi_{u_1 \hat{d}_i u_l}^{v_1 \dots v_r} ) \\
\notag &\qquad - \sum_{j = 1}^r \gamma^{c v_j} \cint^t_0 ( \mf{C}_{adc} \cint^t_0 \Psi_{u_1 \dots u_l}^{v_1 \hat{d}_j v_r} ) \text{.}
\end{align}
\end{proposition}

\begin{proof}
See \cite[Sect. 4.1-4.2]{shao:stt}.
\end{proof}

Finally, we define the \emph{Jacobian} (of $\epsilon$ with respect to $\epsilon [0]$) as
\[ \mc{J} = \exp \cint^t_0 ( \trace k ) \in \mc{C}^\infty \mc{N} \text{.} \]
Note $\mc{J}$ acts as a change of measure quantity, as it satisfies (see \cite[Sect. 4.2]{shao:stt})
\begin{equation} \label{eq.jacobian} \epsilon [\tau] = \mc{J} [\tau] \cdot \epsilon [0] \text{,} \qquad \nabla_t \mc{J} = \trace k \cdot \mc{J} \text{.} \end{equation}

\subsection{Integral Norms} \label{sec.fol_norms}

Given $\Psi \in \mc{C}^\infty \ul{T}^r_l \mc{N}$, we assume any norm over $\Psi [\tau]$ to be with respect to $\gamma [\tau]$.
Define now the following iterated integral norms:
\begin{itemize}
\item If $p \in [1, \infty)$ and $q \in [1, \infty]$, then we define
\begin{align*}
\| \Psi \|_{ L^{p, q}_{t, x} } = \paren{ \int_0^\delta \| \Psi [\tau] \|_{ L^q_x }^p d \tau }^\frac{1}{p} \text{,} \qquad \| \Psi \|_{ L^{\infty, q}_{t, x} } = \sup_{ 0 \leq \tau \leq \delta } \| \Psi [\tau] \|_{ L^q_x } \text{.}
\end{align*}

\item We can also reverse the order of integration.
Given $p, q \in [1, \infty)$, we define
\begin{align*}
\| \Psi \|_{ L^{q, p}_{x, t} } &= \brak{ \int_{ \mc{S} } \paren{ \int_0^\delta \left. | \Psi |^p \mc{J}^\frac{p}{q} \right|_{ (\tau, x) } d \tau }^\frac{q}{p} d \epsilon [0]_x }^\frac{1}{q} \text{,} \\
\| \Psi \|_{ L^{q, \infty}_{x, t} } &= \brak{ \int_{ \mc{S} } \paren{ \left. \sup_{ 0 \leq \tau \leq \delta } | \Psi | \mc{J}^\frac{1}{q} \right|_{ (\tau, x) } }^q d \epsilon [0]_x }^\frac{1}{q} \text{.}
\end{align*}
Furthermore, when $q = \infty$, we define
\begin{align*}
\| \Psi \|_{ L^{\infty, p}_{x, t} } = \sup_{ x \in \mc{S} } \paren{ \int_0^\delta \left. | \Psi |^p \right|_{ (\tau, x) } d \tau }^\frac{1}{p} \text{,} \qquad \| \Psi \|_{ L^{\infty, \infty}_{x, t} } = \sup_{ x \in \mc{S} } \sup_{ 0 \leq \tau \leq \delta } | \Psi | |_{ (\tau, x) } \text{.}
\end{align*}

\item Given any $a \in [1, \infty)$, $s \in \R$, and $p \in [1, \infty]$, we define
\footnote{In this notation, the parameters $a$, $p$, $s$ refer to the summability of the L-P components, the integrability of the $t$-component, and the differentiability of the spatial components, respectively.
The order ``$\ell, t, x$" refers to the relative order of integration and summation.}
\begin{align*}
\| \Psi \|_{ B^{a, p, s}_{\ell, t, x} }^a &= \sum_{k \geq 0} 2^{ask} \| P_k \Psi \|_{ L^{p, 2}_{t, x} }^a + \| P_{< 0} \Psi \|_{ L^{p, 2}_{t, x} }^a \text{,} \\
\| \Psi \|_{ B^{\infty, p, s}_{\ell, t, x} } &= \max \left( \sup_{k \geq 0} 2^{sk} \| P_k \Psi \|_{ L^{p, 2}_{t, x} }, \| P_{< 0} \Psi \|_{ L^{p, 2}_{t, x} } \right) \text{.}
\end{align*}
We also define for convenience the shorthands
\[ \| \Psi \|_{ B^{p, s}_{t, x} } = \| \Psi \|_{ B^{1, p, s}_{k, t, x} } \text{,} \qquad \| \Psi \|_{ H^{p, s}_{t, x} } = \| \Psi \|_{ B^{2, p, s}_{k, t, x} } \text{.} \]
\end{itemize}
Note that all the above norms were used in \cite{shao:stt}.

Next, we consider first-order Sobolev norms on $\mc{N}$ containing both horizontal and $t$-derivatives.
We define the following, the first of which was used in \cite{shao:stt}:
\begin{align*}
\| \Psi \|_{ N^1_{t, x} } &= \| \nabla_t \Psi \|_{ L^{2, 2}_{t, x} } + \| \nabla \Psi \|_{ L^{2, 2}_{t, x} } + \| \Psi \|_{ L^{2, 2}_{t, x} } \text{,} \\
\| \Psi \|_{ N^{1i}_{t, x} } &= \| \Psi \|_{ N^1_{t, x} } + \| \Psi [0] \|_{ H^{1/2}_x } \text{.}
\end{align*}

In addition, we define the norm
\begin{align*}
\| \Psi \|_{ N^{0 \star}_{t, x} } = \inf \{ \| \Phi \|_{ N^{1i}_{t, x} } \mid \nabla_t \Phi = \Psi \} \text{.}
\end{align*}
This measures the smallest $N^{1i}_{t, x}$-norm of any $t$-antiderivative of $\Psi$.
More specifically, we take a $t$-antiderivative and then a spatial derivative of $\Psi$; note we have an additional degree of freedom, in that we can choose the optimal $t$-antiderivative.
In other words, we are essentially trading a $t$-derivative for a spatial derivative.
This norm is, in particular, well suited for taking advantage of the underlying structure behind the null Bianchi identities; see Propositions \ref{thm.structure_eq} and \ref{thm.structure_renorm}.

Finally, we review some methods for combining existing norms to produce new useful norms.
First, let $X$ and $Y$ denote vector spaces, with norms $\| \cdot \|_X$ and $\| \cdot \|_Y$.
\begin{itemize}
\item A natural norm on $X \cap Y$ is the following:
\[ \| v \|_{ X \cap Y } = \| v \|_X + \| v \|_Y \text{.} \]

\item Suppose $X$ and $Y$ are subspaces of a larger vector space $Z$.
Then, one defines a natural norm on $X + Y$ by the formula
\[ \| v \|_{ X + Y } = \inf \{ \| v_X \|_X + \| v_Y \|_Y \mid v = v_X + v_Y \text{, } v_X \in X \text{, } v_Y \in Y \} \text{.} \]
In other words, this norm indicates the smallest way one can decompose $v$ as a sum of two vectors, one in $X$ and one in $Y$.
\end{itemize}

For example, for a norm controlling both the $N^{1i}_{t, x}$- and $L^{\infty, 2}_{x, t}$-norms, we take
\[ \| \Psi \|_{ N^{1i}_{t, x} \cap L^{\infty, 2}_{x, t} } = \| \Psi \|_{ N^{1i}_{t, x} } + \| \Psi \|_{ L^{\infty, 2}_{x, t} } \text{.} \]
As for sum norms, we will often refer to the quantities
\[ \| \Psi \|_{ N^{0\star}_{t, x} + B^{2, 0}_{t, x} } \text{.} \]
This norm will serve an important purpose in the proof of our main results.
Like in \cite{kl_rod:cg, parl:bdc, shao:bdc_nv, wang:cg, wang:cgp}, the proof revolves around an elaborate bootstrap argument.
However, we can take advantage of this specific sum norm in order to simplify some technical portions of the argument, as this norm is well-adapted to the decompositions required for various Besov estimates.
In particular, by inserting additional bootstrap assumptions in terms of this norm, we can avoid the infinite decomposition process that was required in \cite{kl_rod:cg, parl:bdc, shao:bdc_nv, wang:cg, wang:cgp}; see Section \ref{sec.proof_outline}.

\subsection{Conformal Transformations} \label{sec.fol_conf}

We turn briefly to the topic of conformal transformations of $\gamma$.
Consider another horizontal metric $\mbar \in \mc{C}^\infty \ul{T}^0_2 \mc{N}$, defined
\[ \mbar = e^{2 u} \gamma \text{,} \qquad u \in \mc{C}^\infty \mc{N} \text{.} \]
We can view this as a family of conformal transformations for the $\gamma [\tau]$'s, such that the conformal factors $e^{2 u} [\tau]$ also vary smoothly with respect to $\tau$.
If we let $\vbar$ denote the volume form associated with $\mbar$, then we have
\[ \mbar_{ab} = e^{2 u} \gamma_{ab} \text{,} \qquad \vbar_{ab} = e^{2 u} \epsilon_{ab} \text{,} \qquad \mbar^{ab} = e^{-2 u} \gamma^{ab} \text{,} \qquad \vbar^{ab} = e^{-2 u} \epsilon^{ab} \text{.} \]
As before, we denote objects with respect to $\bar{\gamma}$ with a ``bar" over the symbol.

\begin{remark}
In this paper, we are interested only in the special case in which $u$ is constant on each $\mc{S}_\tau$.
Thus, on every $(\mc{S}, \gamma [\tau])$, we are simply performing the rescaling described in Section \ref{sec.geom_riem}, except that the scale depends on the $t$-variable.
\end{remark}

Since $\gamma$ and $\mbar$ are horizontal metrics, each has its own associated second fundamental form.
A direct computation yields the following relation between them:
\begin{align}
\label{eq.conf_sff} \bar{k} = e^{2 u} k + \nabla_t u \cdot \mbar \text{,} \qquad \bar{\trace} \bar{k} = \trace k + 2 \nabla_t u \text{.}
\end{align}

\begin{remark}
On the other hand, the traceless part of the second fundamental form is conformally invariant.
More specifically, the following formula holds:
\[ \bar{k} - \frac{1}{2} ( \bar{\trace} \bar{k} ) \mbar = e^{2 u} \left[ k - \frac{1}{2} ( \trace k ) \gamma \right] \text{.} \]
\end{remark}

The following proposition compares the $t$-covariant derivatives with respect to $\gamma$ and $\bar{\gamma}$, as well as the corresponding $t$-covariant integrals.

\begin{proposition} \label{thm.conf_ev}
If $\Psi \in \mc{C}^\infty \ul{T}^r_l \mc{N}$, then the following relations hold:
\begin{equation} \label{eq.conf_ev} \bar{\nabla}_t [ e^{(l - r) u} \Psi ] = e^{(l - r) u} \nabla_t \Psi \text{,} \qquad \bar{\cint}^t_0 [ e^{(l - r) u} \Psi ] = e^{(l - r) u} \cint^t_0 \Psi \text{.} \end{equation}
\end{proposition}

\begin{proof}
We have from definition and \eqref{eq.conf_sff} that
\begin{align*}
\bar{\nabla}_t \Psi_{c_1 \dots c_l}^{d_1 \dots d_r} &= \nabla_t \Psi_{c_1 \dots c_l}^{d_1 \dots d_r} - \sum_{i = 1}^l ( \mbar^{ab} \bar{k}_{c_i a} - \gamma^{ab} k_{c_i a} ) \Psi_{c_1 \hat{b}_i c_l}^{d_1 \dots d_r} \\
&\qquad + \sum_{j = 1}^r ( \mbar^{a d_j} \bar{k}_{a b} - \gamma^{a d_j} k_{a b} ) \Psi_{c_1 \dots c_l}^{d_1 \hat{b}_j d_r} \\
&= \nabla_t \Psi_{c_1 \dots c_l}^{d_1 \dots d_r} + (r - l) \cdot \nabla_t u \cdot \Psi_{c_1 \dots c_l}^{d_1 \dots d_r} \text{.}
\end{align*}
The first identity in \eqref{eq.conf_ev} now follows immediately from the above.

For the second equality, we begin by applying the first equality to $\cint^t_0 \Psi$:
\[ \bar{\nabla}_t [ e^{(l - r) u} \cint^t_0 \Psi ] = e^{(l - r) u} \nabla_t \cint^t_0 \Psi = e^{(l - r) u} \Psi \text{.} \]
Applying $\bar{\cint}^t_0 \cdot$ to this yields the identity, since $e^{(l - r) u} \bar{\cint}^t_0 \Psi [0]$ vanishes.
\end{proof}

\subsection{Evolutionary Assumptions} \label{sec.fol_reg}

In Section \ref{sec.geom_reg}, we defined various regularity conditions on a single Riemannian surface in order to derive estimates.
The main point is not the estimates themselves, as much as the fact that we uniformly controlled the \emph{constants} of these estimates by various parameters.
Here, we define conditions on the evolution of $\gamma$ so that such regularity properties on $\mc{S}_0$ can be propagated to all the $\mc{S}_\tau$'s.
As a result of these, the estimates of Section \ref{sec.geom} hold on each $(\mc{S}, \gamma [\tau])$, \emph{with the same constants}.
This will play an important role in the bootstrap argument that forms the foundation of the proof of our main theorem.

Our evolutionary assumptions will be given as integral bounds on $k$ and its derivative.
The specific bounds we will reference are the following:
\begin{align}
\label{eqr.sff} \| k \|_{ L^{\infty, 1}_{x, t} } &\leq B \text{,} \\
\label{eqr.sffd_tr} \| \nabla (\trace k) \|_{ L^{2, 1}_{x, t} } &\leq B \text{,} \\
\label{eqr.sffd} \| \nabla k \|_{ L^{2, 1}_{x, t} } &\leq B \text{,} \\
\label{eqr.sffcurl} \inf \{ \| \Phi \|_{ L^{4, \infty}_{x, t} } \mid \Phi \in \mc{C}^\infty \ul{T}^0_3 \mc{N} \text{, } \nabla_t \Phi = \mf{C} \} &\leq B \text{.}
\end{align}
These are mostly the same bounds that were used in \cite{shao:stt}.
Recall that $\mf{C}$, in \eqref{eqr.sffcurl}, is the curl of $k$, which was defined in Section \ref{sec.fol_ev} and in \cite{shao:stt}.

We begin with some simple consequences of \eqref{eqr.sff}.

\begin{proposition} \label{thm.int_ineq}
Assume \eqref{eqr.sff}.
If $q, p_1, p_2 \in [1, \infty]$ and $q_1, q_2 \in [q, \infty]$ satisfy
\[ q_1^{-1} + q_2^{-1} = q^{-1} \text{,} \qquad p_1^{-1} + p_2^{-1} = 1 \text{,} \]
then the following integral estimates hold for any $\Psi, \Phi \in \mc{C}^\infty \ul{T}^r_l \mc{N}$.
\begin{align}
\label{eq.int_ineq} \| \cint^t_0 \Psi \|_{ L^{q, \infty}_{x, t} } &\lesssim_B \| \Psi \|_{ L^{q, 1}_{x, t} } \text{,} \\
\notag \| \cint^t_0 ( \Psi \otimes \Phi ) \|_{ L^{q, \infty}_{x, t} } &\lesssim_B \| \Psi \|_{ L^{q_1, p_1}_{x, t} } \| \Phi \|_{ L^{q_2, p_2}_{x, t} } \text{.}
\end{align}
\end{proposition}

\begin{proof}
See \cite[Sect. 4.3]{shao:stt}.
\end{proof}

As in \cite[Sect. 4.5]{shao:stt}, we introduce evolutionary regularity assumptions for $(\mc{N}, \gamma)$, which combine the \ass{r0}{}, \ass{r1}{}, and \ass{r2}{} conditions in Section \ref{sec.geom_reg} with \eqref{eqr.sff}-\eqref{eqr.sffcurl}.

\begin{definition} \label{def.ass_F}
For convenience, we define the following conditions:
\begin{itemize}
\item $(\mc{N}, \gamma)$ satisfies \ass{F0}{C, N, B}, with data $\{ U_i, \varphi_i, \eta_i \}_{i = 1}^N$, iff $(\mc{S}, \gamma [0])$ satisfies \ass{r0}{C, N}, with the same data, and \eqref{eqr.sff} holds.

\item $(\mc{N}, \gamma)$ satisfies \ass{F1}{C, N, B}, with data $\{ U_i, \varphi_i, \eta_i, \tilde{\eta}_i, e^i \}_{i = 1}^N$, iff $(\mc{S}, \gamma [0])$ satisfies \ass{r1}{C, N} with the same data, and \eqref{eqr.sff}, \eqref{eqr.sffd_tr}, and \eqref{eqr.sffcurl} hold.

\item $(\mc{N}, \gamma)$ satisfies \ass{F2}{C, N, B}, with data $\{ U_i, \varphi_i, \eta_i, \tilde{\eta}_i, e^i \}_{i = 1}^N$, iff $(\mc{S}, \gamma [0])$ satisfies \ass{r2}{C, N} with the same data, and \eqref{eqr.sff}, \eqref{eqr.sffd}, and \eqref{eqr.sffcurl} hold.
\end{itemize}
\end{definition}

In this paper, we will not need to further mention the data ($U_i$, $\eta_i$, etc.) associated with these regularity conditions.
The precise data was required in \cite{shao:stt} for various technical constructions.
The important parameters for this paper are the constants $C$, $N$, $B$, associated with some existing data.
We list the full definitions here in order to maintain consistency with the development in \cite{shao:stt}.

\begin{remark}
In the proof of our main theorem, the above conditions will be derived as consequences of the bootstrap assumptions.
\end{remark}

The above conditions imply that the \ass{r0}{}, \ass{r1}{}, and \ass{r2}{} conditions can be propagated from $( \mc{S}, \gamma [0] )$ to all the $( \mc{S}, \gamma [\tau] )$'s.
To be more specific, we state an abridged version of the result proved in \cite[Sect. 4.5]{shao:stt}.

\begin{proposition} \label{thm.reg_prop}
The following statements hold:
\begin{itemize}
\item If $(\mc{N}, \gamma)$ satisfies \ass{F0}{C, N, B}, then every $(\mc{S}, \gamma [\tau])$, where $\tau \in [0, \delta]$, satisfies \ass{r0}{C^\prime, N}, for some constant $C^\prime$ depending on $C$, $N$, and $B$.

\item If $(\mc{N}, \gamma)$ satisfies \ass{F1}{C, N, B}, then every $(\mc{S}, \gamma [\tau])$, where $\tau \in [0, \delta]$, satisfies \ass{r1}{C^\prime, N}, for some constant $C^\prime$ depending on $C$, $N$, and $B$.

\item If $(\mc{N}, \gamma)$ satisfies \ass{F2}{C, N, B}, then every $(\mc{S}, \gamma [\tau])$, where $\tau \in [0, \delta]$, satisfies \ass{r2}{C^\prime, N}, for some constant $C^\prime$ depending on $C$, $N$, and $B$.
\end{itemize}
\end{proposition}

\begin{proof}
See \cite[Sect. 4.5]{shao:stt}.
\end{proof}

That the regularity properties of Section \ref{sec.geom_reg} can be propagated to all the $\mc{S}_\tau$'s implies that various integral norms on the $\mc{S}_\tau$'s are in fact comparable to each other.
Here, we state these results in a form that we will use later.

\begin{proposition} \label{thm.norm_comp}
Let $\Psi \in \mc{C}^\infty \ul{T}^r_l \mc{N}$ be $t$-parallel, i.e., that $\nabla_t \Psi \equiv 0$.
\begin{itemize}
\item If \eqref{eqr.sff} holds, then for any $q \in [1, \infty]$,
\begin{equation} \label{eq.norm_comp_int} \| \Psi \|_{ L^{q, \infty}_{x, t} } \lesssim_B \| \Psi [0] \|_{ L^q_x } \text{.} \end{equation}

\item If $(\mc{N}, \gamma)$ satisfies \ass{F1}{C, N, B}, then for any $a \in [1, \infty]$ and $s \in (-1, 1)$,
\begin{align}
\label{eq.norm_comp_besov} \| \Psi \|_{ B^{a, \infty, s}_{\ell, t, x} } &\lesssim_{C, N, B, s, r, l} \| \Psi [0] \|_{ B^{a, s}_{\ell, x} } \text{.}
\end{align}
\end{itemize}
\end{proposition}

\begin{proof}
\eqref{eq.norm_comp_int} is immediate, since $| \Psi |$ is independent of $t$ and $\mc{J} \simeq 1$ on $\mc{N}$ due to \eqref{eqr.sff}.
\footnote{On the last point, see \cite[Sect. 4.3]{shao:stt}.}
The proof of \eqref{eq.norm_comp_besov} can be found in \cite[Sect. 5.1]{shao:stt}.
\end{proof}

We now consider Sobolev-type estimates involving all of $\mc{N}$.

\begin{proposition} \label{thm.nsob_ineq}
Assume $(\mc{N}, \gamma)$ satisfies \ass{F0}{C, N, B}.
If $\Psi \in \mc{C}^\infty \ul{T}^r_l \mc{N}$, then
\begin{align}
\label{eq.nsob_ineq} \| \Psi \|_{ L^{2, \infty}_{x, t} } &\lesssim_B \| \Psi [0] \|_{ L^2_x } + \| \nabla_t \Psi \|_{ L^{2, 2}_{t, x} }^\frac{1}{2} \| \Psi \|_{ L^{2, 2}_{t, x} }^\frac{1}{2} \text{,} \\
\notag \| \Psi \|_{ L^{4, \infty}_{x, t} } &\lesssim_{C, N, B} \| \Psi [0] \|_{ L^4_x } + \| \nabla_t \Psi \|_{ L^{2, 2}_{t, x} }^\frac{1}{2} ( \| \nabla \Psi \|_{ L^{2, 2}_{t, x} } + \| \Psi \|_{ L^{2, 2}_{t, x} } )^\frac{1}{2} \text{.}
\end{align}
\end{proposition}

\begin{proof}
See \cite[Sect. 4.6]{shao:stt}.
\end{proof}

\begin{proposition} \label{thm.nsob_trace}
Assume $(\mc{N}, \gamma)$ satisfies \ass{F1}{C, N, B}.
If $\Psi \in \mc{C}^\infty \ul{T}^r_l \mc{N}$, then
\begin{equation} \label{eq.nsob_trace} \| \Psi \|_{ H^{\infty, 1/2}_{t, x} } \lesssim_{C, N, B, r, l} \| \Psi [0] \|_{ H^{1/2}_x } + \| \nabla_t \Psi \|_{ L^{2, 2}_{t, x} }^\frac{1}{2} ( \| \nabla \Psi \|_{ L^{2, 2}_{t, x} } + \| \Psi \|_{ L^{2, 2}_{t, x} } )^\frac{1}{2} \text{.} \end{equation}
\end{proposition}

\begin{proof}
See \cite[Sect. 5.1]{shao:stt}.
\end{proof}

\subsection{Bilinear Product Estimates} \label{sec.fol_prod}

The following bilinear product estimates, all proved in \cite{shao:stt}, will be essential to the proof of our main results.

\begin{theorem} \label{thm.est_prod}
Assume that $(\mc{N}, \gamma)$ satisfies \ass{F1}{C, N, B}.
Furthermore, fix horizontal tensor fields $\Psi \in \mc{C}^\infty \ul{T}^{r_1}_{l_1} \mc{N}$ and $\Phi \in \mc{C}^\infty \ul{T}^{r_2}_{l_2} \mc{N}$.
\begin{itemize}
\item If $a \in [1, \infty]$, $s \in (-1, 1)$, and $\Psi$ is $t$-parallel (i.e., $\nabla_t \Psi \equiv 0$), then
\begin{align}
\label{eq.est_prod_imp} \| \Phi \otimes \Psi \|_{ B^{a, 2, s}_{\ell, t, x} } &\lesssim_{ C, N, B, s, r_1, l_1, r_2, l_2 } ( \| \nabla \Phi \|_{ L^{2, 2}_{t, x} } + \| \Phi \|_{ L^{\infty, 2}_{x, t} } ) \| \Psi [0] \|_{ B^{a, s}_{\ell, x} } \text{.}
\end{align}

\item If $a \in [1, \infty]$ and $s \in (-1, 1)$, then
\begin{align}
\label{eq.est_trace_sh} \| \cint^t_0 ( \Phi \otimes \Psi ) \|_{ B^{a, \infty, s}_{\ell, t, x} } &\lesssim_{ C, N, B, s, r_1, l_1, r_2, l_2 } ( \| \nabla \Phi \|_{ L^{2, 2}_{t, x} } + \| \Phi \|_{ L^{\infty, 2}_{x, t} } ) \| \Psi \|_{ B^{a, 2, s}_{\ell, t, x} } \text{,} \\
\label{eq.est_trace_shp} \| \Phi \otimes \cint^t_0 \Psi \|_{ B^{a, 2, s}_{\ell, t, x} } &\lesssim_{ C, N, B, s, r_1, l_1, r_2, l_2 } ( \| \nabla \Phi \|_{ L^{2, 2}_{t, x} } + \| \Phi \|_{ L^{\infty, 2}_{x, t} } ) \| \Psi \|_{ B^{a, 1, s}_{\ell, t, x} } \text{.}
\end{align}

\item The following estimates hold:
\begin{align}
\label{eq.est_prod_ex} \| \Phi \otimes \Psi \|_{ B^{\infty, 0}_{t, x} } &\lesssim_{ C, N, B, r_1, l_1, r_2, l_2 } \| \Phi \|_{ N^{1i}_{t, x} } \| \Psi \|_{ N^{1i}_{t, x} } \text{,} \\
\label{eq.est_trace_ex} \| \cint^t_0 ( \nabla_t \Phi \otimes \Psi ) \|_{ B^{\infty, 0}_{t, x} } &\lesssim_{ C, N, B, r_1, l_1, r_2, l_2 } \| \Phi \|_{ N^{1i}_{t, x} } \| \Psi \|_{ N^{1i}_{t, x} } \text{.}
\end{align}
\end{itemize}
\end{theorem}

\begin{proof}
See \cite[Sect. 5.2-5.3]{shao:stt}.
\end{proof}

By aggregating norms, we can combine some of the above estimates:

\begin{corollary} \label{thm.est_prod_cor}
Assume that $(\mc{N}, \gamma)$ satisfies \ass{F1}{C, N, B}.
Furthermore, fix horizontal tensor fields $\Psi \in \mc{C}^\infty \ul{T}^{r_1}_{l_1} \mc{N}$ and $\Phi \in \mc{C}^\infty \ul{T}^{r_2}_{l_2} \mc{N}$.
\begin{align}
\label{eq.est_trace_sh_ex} \| \cint^t_0 ( \Phi \otimes \Psi ) \|_{ B^{\infty, 0}_{t, x} } &\lesssim_{ C, N, B, r_1, r_2, l_1, l_2 } \| \Phi \|_{ N^{0 \star}_{t, x} + B^{2, 0}_{t, x} } \| \Psi \|_{ N^{1i}_{t, x} \cap L^{\infty, 2}_{x, t} } \text{,} \\
\label{eq.est_trace_shp_ex} \| \Phi \otimes \cint^t_0 \Psi \|_{ B^{2, 0}_{t, x} } &\lesssim_{ C, N, B, r_1, r_2, l_1, l_2 } ( 1 + \delta^\frac{1}{2} ) \| \Phi \|_{ N^{1i}_{t, x} \cap L^{\infty, 2}_{x, t} } \| \Psi \|_{ N^{0 \star}_{t, x} + B^{2, 0}_{t, x} } \text{.}
\end{align}
\end{corollary}

\begin{proof}
We begin with \eqref{eq.est_trace_sh_ex}.
Given a decomposition $\Phi = \Phi_1 + \Phi_2$,
\begin{align*}
\| \cint^t_0 ( \Phi \otimes \Psi ) \|_{ B^{\infty, 0}_{t, x} } &\lesssim \| \cint^t_0 ( \Phi_1 \otimes \Psi ) \|_{ B^{\infty, 0}_{t, x} } + \| \cint^t_0 ( \Phi_2 \otimes \Psi ) \|_{ B^{\infty, 0}_{t, x} } \\
&\lesssim \| \cint^t_0 ( \Phi_1 \otimes \Psi ) \|_{ B^{\infty, 0}_{t, x} } + \| \Phi_2 \|_{ B^{2, 0}_{t, x} } \| \Psi \|_{ N^{1i}_{t, x} \cap L^{\infty, 2}_{t, x} } \text{,}
\end{align*}
where we applied \eqref{eq.est_trace_sh}.
If $\Phi_0 \in \mc{C}^\infty \ul{T}^{r_2}_{l_2} \mc{N}$ satisfies $\nabla_t \Phi_0 = \Phi_1$, then by \eqref{eq.est_trace_ex},
\begin{align*}
\| \cint^t_0 ( \Phi_1 \otimes \Psi ) \|_{ B^{\infty, 0}_{t, x} } &= \| \cint^t_0 ( \nabla_t \Phi_0 \otimes \Psi ) \|_{ B^{\infty, 0}_{t, x} } \lesssim \| \Phi_0 \|_{ N^{1i}_{t, x} } \| \Psi \|_{ N^{1i}_{t, x} } \text{.}
\end{align*}
Varying over all possible $\Phi_0$'s, the above implies
\begin{align*}
\| \cint^t_0 ( \Phi_1 \otimes \Psi ) \|_{ B^{\infty, 0}_{t, x} } &\lesssim \| \Phi_1 \|_{ N^{0 \star}_{t, x} } \| \Psi \|_{ N^{1i}_{t, x} \cap L^{\infty, 2}_{x, t} } \text{,} \\
\| \cint^t_0 ( \Phi \otimes \Psi ) \|_{ B^{\infty, 0}_{t, x} } &\lesssim ( \| \Phi_1 \|_{ N^{0 \star}_{t, x} } + \| \Phi_2 \|_{ B^{2, 0}_{t, x} } ) \| \Psi \|_{ N^{1i}_{t, x} \cap L^{\infty, 2}_{x, t} } \text{.}
\end{align*}
Varying over all decompositions $\Phi = \Phi_1 + \Phi_2$ yields \eqref{eq.est_trace_sh_ex}.

For \eqref{eq.est_trace_shp_ex}, we begin by assuming a similar decomposition $\Psi = \Psi_1 + \Psi_2$, so
\begin{align*}
\| \Phi \otimes \cint^t_0 \Psi \|_{ B^{2, 0}_{t, x} } &\lesssim \| \Phi \otimes \cint^t_0 \Psi_1 \|_{ B^{2, 0}_{t, x} } + \| \Phi \otimes \cint^t_0 \Psi_2 \|_{ B^{2, 0}_{t, x} } \\
&\lesssim \| \Phi \otimes \cint^t_0 \Psi_1 \|_{ B^{2, 0}_{t, x} } + \delta^\frac{1}{2} \| \Phi \|_{ N^{1i}_{t, x} \cap L^{\infty, 2}_{x, t} } \| \Psi_2 \|_{ B^{2, 0}_{t, x} } \text{,}
\end{align*}
where we also applied \eqref{eq.est_trace_shp}.
For any $t$-parallel $\Theta \in \mc{C}^\infty \ul{T}^r_l \mc{N}$, we can bound
\begin{align*}
\| \Phi \otimes \cint^t_0 \Psi_1 \|_{ B^{2, 0}_{t, x} } \lesssim \| \Phi \otimes ( \cint^t_0 \Psi_1 + \Theta ) \|_{ B^{2, 0}_{t, x} } + \| \Phi \otimes \Theta \|_{ B^{2, 0}_{t, x} } \text{.}
\end{align*}
Applying \eqref{eq.est_prod_imp} and \eqref{eq.est_prod_ex} to the terms on the right-hand side yields
\begin{align*}
\| \Phi \otimes \cint^t_0 \Psi_1 \|_{ B^{2, 0}_{t, x} } &\lesssim \| \Phi \|_{ N^{1i}_{t, x} \cap L^{\infty, 2}_{x, t} } ( \delta^\frac{1}{2} \| \cint^t_0 \Psi_1 + \Theta \|_{ N^{1i}_{t, x} } + \| \Theta [0] \|_{ B^0_x } ) 
\end{align*}
Moreover, since $\Theta [0] = ( \cint^t_0 \Psi + \Theta ) [0]$, then
\[ \| \Theta [0] \|_{ B^0_x } \lesssim \| \Theta [0] \|_{ H^{1/2}_x } \lesssim \| \cint^t_0 \Psi_1 + \Theta \|_{ N^{1i}_{t, x} } \text{.} \]
Note any (covariant) $t$-antiderivative of $\Psi_1$ can be written as $\cint^t_0 \Psi_1 + \Theta$ for some $t$-parallel $\Theta \in \mc{C}^\infty \ul{T}^r_l \mc{N}$.
Thus, combining the above developments, we have
\begin{align*}
\| \Phi \otimes \cint^t_0 \Psi_1 \|_{ B^{2, 0}_{t, x} } &\lesssim ( 1 + \delta^\frac{1}{2} ) \| \Phi \|_{ N^{1i}_{t, x} \cap L^{\infty, 2}_{x, t} } \| \Psi_1 \|_{ N^{0 \star}_{t, x} } \text{,} \\
\| \Phi \otimes \cint^t_0 \Psi \|_{ B^{2, 0}_{t, x} } &\lesssim ( 1 + \delta^\frac{1}{2} ) \| \Phi \|_{ N^{1i}_{t, x} \cap L^{\infty, 2}_{x, t} } \| \Psi \|_{ N^{0 \star}_{t, x} + B^{1, 0}_{t, x} } \text{,}
\end{align*}
which completes the proof of \eqref{eq.est_trace_shp_ex}.
\end{proof}

The following \emph{sharp trace estimate} is also established in \cite{shao:stt} (and is a variation and simplification of similar estimates found in \cite{kl_rod:cg, wang:cg}).

\begin{theorem} \label{thm.sharp_trace}
Assume that $(\mc{N}, \gamma)$ satisfies \ass{F2}{C, N, B}.
Let $\Psi \in \mc{C}^\infty \ul{T}^r_l \mc{N}$, and suppose $\Psi_1, \Psi_2 \in \mc{C}^\infty \ul{T}^r_{l + 1} \mc{N}$ are such that the decomposition
\begin{align}
\label{eq.sharp_trace_decomp} \nabla \Psi = \nabla_t \Psi_1 + \Psi_2
\end{align}
holds.
Then, we have the following estimate:
\begin{align}
\label{eq.sharp_trace} \| \Psi \|_{ L^{\infty, 2}_{x, t} } &\lesssim_{ C, N, B, r, l } ( 1 + \| k \|_{ L^{2, \infty}_{x, t} } ) ( \| \Psi \|_{ N^{1i}_{t, x} } + \| \Psi_1 \|_{ N^{1i}_{t, x} } + \| \Psi_2 \|_{ B^{2, 0}_{t, x} } ) \text{.}
\end{align}
\end{theorem}

\begin{proof}
See \cite[Sect. 5.4]{shao:stt}.
\end{proof}

Combining Theorem \ref{thm.sharp_trace} with our aggregated norms and taking the infimum over all decompositions of the form \eqref{eq.sharp_trace_decomp} results in the following:

\begin{corollary} \label{thm.sharp_trace_ex}
If $(\mc{N}, \gamma)$ satisfies \ass{F2}{C, N, B}, and if $\Psi \in \mc{C}^\infty \ul{T}^r_l \mc{N}$, then
\begin{align}
\label{eq.sharp_trace_ex} \| \Psi \|_{ L^{\infty, 2}_{x, t} } &\lesssim_{ C, N, B, r, l } ( 1 + \| k \|_{ L^{2, \infty}_{x, t} } ) ( \| \Psi \|_{ N^{1i}_{t, x} } + \| \nabla \Psi \|_{ N^{0 \star}_{t, x} + B^{2, 0}_{t, x} } ) \text{.}
\end{align}
\end{corollary}

\subsection{Besov-Elliptic Estimates} \label{sec.fol_curv}

In Section \ref{sec.geom_curv}, we cited $L^2$-elliptic estimates on a surface $(\mc{S}, h)$, given a weak regularity condition for the associated Gauss curvature.
Here, we state analogous estimates in geometric Besov spaces.

First, we port the \ass{k}{} condition in Section \ref{sec.geom_curv} to the foliation setting.

\begin{definition} \label{def.ass_K}
$(\mc{N}, \gamma)$ satisfies \ass{K}{C, D}, with data $( f, W, V )$, iff:
\begin{itemize}
\item $f \in \mc{C}^\infty \mc{N}$ satisfies the uniform estimate
\[ C^{-1} \leq f |_{(\tau, x)} \leq C \text{,} \qquad (\tau, x) \in \mc{N} \text{.} \]

\item $V \in \mc{C}^\infty \ul{T}^0_1 \mc{N}$ and $W \in \mc{C}^\infty \mc{N}$ satisfy
\[ \| V \|_{ H^{\infty, 1/2}_{t, x} } \leq D \text{,} \qquad \| W \|_{ L^{\infty, 2}_{t, x} } \leq D \text{.} \]

\item $\mc{K}$ can be decomposed in the form
\[ \mc{K} - f = \gamma^{ab} \nabla_a V_b + W \text{.} \]
\end{itemize}
\end{definition}

Note that \ass{K}{C, D} implies that every $( \mc{S}, \gamma [\tau] )$ satisfies \ass{k}{C, D^\prime}, for some constant $D^\prime$ depending on $C$ and $D$.
Moreover, if $D$ is very small, then so is $D^\prime$.

Given the \ass{K}{} condition, we can prove integrated Besov-elliptic estimates involving operators of the form $\nabla \mc{D}^{-1}$, where $\mc{D}$ is any of the Hodge operators.

\begin{theorem} \label{thm.besov_impr}
Assume $(\mc{N}, \gamma)$ satisfies \ass{F1}{C, N} and \ass{K}{C, D}, with $D \ll 1$ sufficiently small.
In addition, suppose $p \in [1, \infty]$.
\begin{itemize}
\item If $\Psi \in \mc{C}^\infty \ul{T}^r_l \mc{N}$, then
\begin{equation} \label{eq.besov_impr_sh} \| \Psi \|_{ L^{p, \infty}_{t, x} } \lesssim_{ C, N, r, l } \| \nabla \Psi \|_{ B^{p, 0}_{t, x} } + \| \Psi \|_{ L^{p, 2}_{t, x} } \text{.} \end{equation}

\item If $a \in [1, \infty]$, if $\mc{D}$ is any one of the operators $\mc{D}_1$, $\mc{D}_2$, $\mc{D}_1^\ast$, and if $\xi$ is a smooth section of the appropriate Hodge bundle on $\mc{N}$, then
\begin{equation} \label{eq.besov_impr_bdd} \| \nabla \mc{D}^{-1} \xi \|_{ B^{a, p, 0}_{\ell, t, x} } \lesssim_{C, N} \| \xi \|_{ B^{a, p, 0}_{\ell, t, x} } \text{.} \end{equation}
\end{itemize}
\end{theorem}

\begin{proof}
See \cite[Sect. 6.5]{shao:stt}.
\end{proof}

Finally, the \ass{K}{} condition implies weak control for the Gauss curvatures.

\begin{proposition} \label{thm.curv_sob}
Assume $(\mc{N}, \gamma)$ satisfies \ass{R1}{C, N} and \ass{K}{C, D}, the latter with data $(f, W, V)$, and with $D \ll 1$ sufficiently small.
Then,
\begin{equation} \label{eq.curv_sob} \| \mc{K} - f \|_{ H^{\infty, -1/2}_{t, x} } \lesssim D \text{.} \end{equation}
\end{proposition}

\begin{proof}
See \cite[Sect. 6.5]{shao:stt}.
\end{proof}

\section{Null Cones to Infinity} \label{sec.nc}

In this section, we describe in detail the setting that we will consider: a geodesically foliated smooth null cone extending ``toward infinity".
In particular, we define the associated connection and curvature quantities, and we list the null structure equations which relate these quantities.
Finally, we apply a renormalization in order to transform our physical system into another which is more sensible to analyze.

Throughout, we will let $(M, g)$ denote a four-dimensional orientable Lorentzian manifold which satisfies the \emph{Einstein-vacuum equations},
\[ \operatorname{Ric}_g \equiv 0 \text{.} \]

\subsection{Geodesically Foliated Null Cones} \label{sec.nc_gf}

Although one can foliate a (truncated) null cone in many different ways---for example, foliations using a time or an optical function have been used in some applications, cf. \cite{chr:gr_bh, chr_kl:stb_mink, kl_nic:stb_mink, kl_rod:bdc, parl:bdc, shao:bdc_nv}---the geodesic foliation is the simplest algebraically.
Indeed, the geodesically foliated setting contains the least number of connection coefficients to work with, and it does not depend on quantities external to the geometry of the spacetime $(M, g)$.

First, let $\mc{S}$ be a compact $2$-dimensional spacelike (i.e., Riemannian) submanifold of $M$ that is diffeomorphic to $\Sph^2$.
This will serve as the initial sphere, or base, of our null cone.
To construct our cone, we define the following:
\begin{itemize}
\item Let $s_0 \in (0, \infty)$, and suppose $\mc{S}$ has area $4 \pi s_0^2$.
\footnote{In other words, $s_0$ is the ``radius" of the initial sphere $\mc{S}$.}

\item For each point $p \in \mc{S}$, we let $\ell_p$ denote a future-directed null tangent vector at $p$ that is orthogonal to $\mc{S}$.
Furthermore, we choose these vectors such that these $\ell_p$'s vary smoothly with respect to $p$.

\item For each $p \in \mc{S}$, we let $\lambda_p$ denote the future-directed null geodesic
\[ \lambda_p : [s_0, \infty) \rightarrow M \text{,} \qquad \lambda_p (s_0) = p \text{,} \qquad \lambda_p^\prime (s_0) = \ell_p \text{.} \]
\end{itemize}

Suppose every $\lambda_p$, $p \in \mc{S}$, is well-defined on the entire half-line $[s_0, \infty)$.
Then, we can define the future null cone $\mc{N}$ beginning from $\mc{S}$ as the set
\[ \{ \lambda_p (v) \mid p \in \mc{S} \text{, } v \in [s_0, \infty) \} \]
traced out by the $\lambda_p$'s.
In addition, we assume the following for $\mc{N}$:
\begin{itemize}
\item No two distinct $\lambda_p$'s intersect.

\item This family $\{ \lambda_p \mid p \in \mc{S} \}$ of null geodesics has no null conjugate points.
\end{itemize}
Under these assumptions, $\mc{N}$ forms a smooth null hypersurface of $M$; see \cite{haw_el:gr}.
In other words, $\mc{N}$ is a regular outgoing future null cone extending to infinity.

Let $s: \mc{N} \rightarrow \R$ denote the \emph{affine parameter}, satisfying
\[ s ( \lambda_p (v) ) = v \text{,} \qquad p \in \mc{S} \text{,} \quad v \in [s_0, \infty) \text{.} \]
In particular, $s$ is a smooth function, and its level sets
\[ \mc{S}_v = \{ q \in \mc{N} \mid s(q) = v \} \]
are diffeomorphic to $\mc{S}$.
As a result, we can reformulate $\mc{N}$ as
\begin{equation} \label{eq.nc_fol} \mc{N} = \bigcup_{s_0 \leq v < \infty} \mc{S}_v \simeq [s_0, \infty) \times \mc{S} \simeq [s_0, \infty) \times \Sph^2 \text{,} \end{equation}
known as a \emph{geodesic foliation} of $\mc{N}$.
We will consider $\mc{N}$ both as a smooth null hypersurface of $M$ and as a spherical foliation, depending on context.
In the latter characterization, we retrieve a special case of the abstract setting of Section \ref{sec.fol}.

We also define the following vector fields on $\mc{N}$:
\begin{itemize}
\item We define the \emph{tangent null vector field} $L$ on $\mc{N}$ as
\[ L |_{ \lambda_p (v) } = \lambda_p^\prime (v) \text{,} \qquad p \in \mc{S} \text{,} \quad v \in [s_0, \infty) \text{.} \]
By definition, $L$ is geodesic, and $L s \equiv 1$ everywhere.
In particular, $L$ is transverse to the $\mc{S}_v$'s, so the $\mc{S}_v$'s are Riemannian submanifolds of $\mc{N}$.

\item Let $\ul{L}$ denote the \emph{conjugate null vector field} on $\mc{N}$, which is orthogonal to every $\mc{S}_v$ and satisfies $g (L, \ul{L}) \equiv -2$.
Note that $\ul{L}$ is transverse to $\mc{N}$, and that $\ul{L}$ is uniquely defined from the geodesic foliation.
\end{itemize}
By combining the null vector fields $L$ and $\ul{L}$ with local ($g$-)orthonormal frames tangent to the $\mc{S}_v$'s, one obtains the usual \emph{local null frames} on $\mc{N}$.

Considering $\mc{N}$ now as the foliation \eqref{eq.nc_fol}, we can define the horizontal metric $\mind \in \mc{C}^\infty \ul{T}^0_2 \mc{N}$ to correspond to the metrics on the $\mc{S}_v$'s induced by $g$.
Furthermore, by fixing an orientation for $\mc{S}$, we define a natural orientation for each $\mc{S}_v$.
These orientations define the volume form $\vind \in \mc{C}^\infty \ul{T}^0_2 \mc{N}$ associated with $\mind$.

From now on, objects defined with respect to $\mind$ and $\vind$ will be denoted with a ``slash".
For example, the covariant derivative with respect to $\mind$ is denoted $\nasla$.

\begin{remark}
In this setting, the evolutionary derivative operator is $\nasla_s$, with respect to $\mind$ and the $s$-foliation of $\mc{N}$.
From the definition, one can show that $\nasla_s$, as defined in Section \ref{sec.fol_ev}, coincides with the operator $\nasla_L$, defined as the projection of the \emph{spacetime} covariant derivative $D_L$ to the $\mc{S}_v$'s.
This derivative $\nasla_L$ is the operator that was utilized in previous works, e.g., \cite{kl_rod:cg, kl_rod:stt, parl:bdc, shao:bdc_nv, wang:cg, wang:cgp}.
\end{remark}

\subsection{Connection and Curvature} \label{sec.nc_rc}

Next, we define the \emph{Ricci coefficients} on $\mc{N}$.
These are connection quantities, expressed as horizontal tensor fields, that describe the derivatives of $L$ and $\ul{L}$ in directions tangent to $\mc{N}$.
Throughout, we let $D$ denote the restriction of the spacetime ($g$-)Levi-Civita connection to $\mc{N}$.

In the geodesic foliation, the Ricci coefficients are the following:
\begin{itemize}
\item Define the \emph{null second fundamental forms} $\chi, \ul{\chi} \in \mc{C}^\infty \ul{T}^0_2 \mc{N}$ by
\[ \chi (X, Y) = g ( D_X L, Y ) \text{,} \qquad \ul{\chi} (X, Y) = g ( D_X \ul{L}, Y ) \text{,} \qquad X, Y \in \mc{C}^\infty \ul{T}^1_0 \mc{N} \text{.} \]
Since $L$ and $\ul{L}$ are orthogonal to the $\mc{S}_v$'s, both $\chi$ and $\ul{\chi}$ are symmetric.
In particular, the trace and traceless parts of $\chi$ (with respect to $\mind$),
\[ \trase \chi = \mind^{ab} \chi_{ab} \text{,} \qquad \hat{\chi} = \chi - \frac{1}{2} (\trase \chi) \mind \text{,} \]
are often called the \emph{expansion} and \emph{shear} of $\mc{N}$, respectively.
The same trace-traceless decomposition can also be done for $\ul{\chi}$.

\item Define the \emph{torsion} $\zeta \in \mc{C}^\infty \ul{T}^0_1 \mc{N}$ by
\[ \zeta (X) = \frac{1}{2} g ( D_X L, \ul{L} ) \text{,} \qquad X \in \mc{C}^\infty \ul{T}^1_0 \mc{N} \text{.} \]
\end{itemize}
Furthermore, one can explicitly compute the associated $\mind$-second fundamental form:
\[ /\mspace{-9mu}k = \chi \text{.} \]
For details, see, e.g., the proof of \cite[Lemma 2.26]{kl_rod:cg}.

Now, let $R$ denote the spacetime Riemann curvature tensor associated with $g$.
We can then define the following null curvature components,
\begin{alignat*}{5}
\alpha, \ul{\alpha} &\in \mc{C}^\infty \ul{T}^0_2 \mc{N} \text{,} &\qquad \alpha (X, Y) &= R (L, X, L, Y) \text{,} &\qquad \ul{\alpha} (X, Y) &= R (\ul{L}, X, \ul{L}, Y) \text{,} \\
\beta, \ul{\beta} &\in \mc{C}^\infty \ul{T}^0_1 \mc{N} \text{,} &\qquad \beta (X) &= \frac{1}{2} R (L, X, L, \ul{L}) \text{,} &\qquad \ul{\beta} (X) &= \frac{1}{2} R (\ul{L}, X, \ul{L}, L) \text{,} \\
\rho, \sigma &\in \mc{C}^\infty \mc{N} \text{,} &\qquad \rho &= \frac{1}{4} R (L, \ul{L}, L, \ul{L}) \text{,} &\qquad \sigma &= \frac{1}{4} {}^\star R (L, \ul{L}, L, \ul{L}) \text{,}
\end{alignat*}
where ${}^\star R$ denotes the left (spacetime) Hodge dual of $R$ (with respect to a fixed orientation of $M$).
In the Einstein-vacuum setting, these curvature components comprise all the independent components of $R$.

Finally, we define the \emph{mass aspect function} on $\mc{N}$ by
\[ \mu \in \mc{C}^\infty \mc{N} \text{,} \qquad \mu = - \mind^{ab} \nasla_a \zeta_b - \rho + \frac{1}{2} \mind^{ac} \mind^{bd} \hat{\chi}_{ab} \ul{\hat{\chi}}_{cd} \text{.} \]
This quantity plays a essential role in the proof of our main theorem and is related to the Hawking masses of the $\mc{S}_v$'s; see \cite{chr_kl:stb_mink}.

The Ricci and curvature coefficients are related to each other via a family of geometric differential equations, known as the \emph{null structure equations}.
We now state these identites in terms of the induced metrics $\mind$ and volume forms $\vind$.
For details and derivations, see, for example, \cite{chr_kl:stb_mink, kl_rod:cg}.

\begin{proposition} \label{thm.structure_eq}
Assume that all geometric quantities are defined with respect to $\mind$, $\vind$, and $s$.
Then, the following structure equations hold on $\mc{N}$.
\begin{itemize}
\item Evolution equations:
\begin{align}
\label{eq.structure_ev} \nasla_s \chi_{ab} &= - \mind^{cd} \chi_{ac} \chi_{bd} - \alpha_{ab} \text{,} \\
\notag \nasla_s \zeta_a &= - 2 \mind^{bc} \chi_{ab} \zeta_c - \beta_a \text{,} \\
\notag \nasla_s \ul{\chi}_{ab} &= - ( \nasla_a \zeta_b + \nasla_b \zeta_a ) - \frac{1}{2} \mind^{cd} ( \chi_{ac} \ul{\chi}_{bd} + \chi_{bc} \ul{\chi}_{ad} ) + 2 \zeta_a \zeta_b + \rho \mind_{ab} \text{.}
\end{align}

\item Elliptic equations:
\begin{align}
\label{eq.structure_ell} \sodge_2 \hat{\chi}_a &= - \beta_a + \frac{1}{2} \nasla_a ( \trase \chi ) + \frac{1}{2} ( \trase \chi ) \zeta_a - \mind^{bc} \hat{\chi}_{ab} \zeta_c \text{,} \\
\notag \sodge_1 \zeta &= - ( \rho + i \sigma ) - \mu + \frac{1}{2} ( \mind^{ac} + i \vind^{ac} ) \mind^{bd} \hat{\chi}_{ab} \ul{\hat{\chi}}_{cd} \text{.}
\end{align}

\item Gauss-Codazzi equations:
\begin{align}
\label{eq.structure_gc} \nasla_b \chi_{ac} - \nasla_c \chi_{ab} &= - \vind_{bc} \vind_a{}^d \beta_d + \chi_{ab} \zeta_c - \chi_{ac} \zeta_b \text{,} \\
\notag \sauss &= - \rho + \frac{1}{2} ( \mind^{ac} \mind^{bd} - \mind^{ab} \mind^{cd} ) \chi_{ab} \ul{\chi}_{cd} \text{.}
\end{align}

\item Derivative evolution equations:
\begin{align}
\label{eq.structure_evd} \nasla_s \nasla_a ( \trase \chi ) &= - \mind^{bc} \chi_{ab} \nasla_c ( \trase \chi ) - 2 \mind^{bc} \mind^{de} \chi_{bd} \nasla_a \chi_{ce} \text{,} \\
\notag \nasla_s \mu &= - \frac{3}{2} ( \trase \chi ) \mu - 2 \mind^{ab} \zeta_a \beta_b + 2 \mind^{ab} \nasla_a ( \trase \chi ) \zeta_b \\
\notag &\qquad + 2 \mind^{ab} \mind^{cd} \hat{\chi}_{ac} \nasla_b \zeta_d - 2 \mind^{ab} \mind^{cd} \hat{\chi}_{ac} \zeta_b \zeta_d \\
\notag &\qquad + \frac{3}{2} \mind^{ab} ( \trase \chi ) \zeta_a \zeta_b - \frac{1}{4} \mind^{ab} \mind^{cd} ( \trase \ul{\chi} ) \hat{\chi}_{ac} \hat{\chi}_{bd} \text{.}
\end{align}

\item Null Bianchi equations:
\begin{align}
\label{eq.structure_nb} \nasla_s \beta_a &= \sodge_2 \alpha_a - 2 ( \trase \chi ) \beta_a + \mind^{bc} \zeta_b \alpha_{ac} \text{,} \\
\notag \nasla_s ( \rho + i \sigma ) &= \sodge_1 \beta - \frac{3}{2} ( \trase \chi ) ( \rho + i \sigma ) - ( \mind^{ab} - i \vind^{ab} ) ( \zeta_a \beta_b + \frac{1}{2} \mind^{cd} \ul{\chi}_{ac} \alpha_{bd} ) \text{,} \\
\notag \nasla_s \ul{\beta}_a &= \sodge_1^\ast ( \rho - i \sigma ) - ( \trase \chi ) \ul{\beta}_a + 3 \zeta_a \rho - 3 \vind_a{}^b \zeta_b \sigma + 2 \mind^{bc} \ul{\hat{\chi}}_{ab} \beta_c \text{.}
\end{align}
\end{itemize}
\end{proposition}

\subsection{Minkowski and Schwarzschild Null Cones} \label{sec.nc_ms}

Recall that the Minkowski spacetime can be represented as $(M, g) = (\R^{1+3}, \eta)$, where $\eta$ is the Minkowski metric, which in polar coordinates can be written as
\[ \eta = -dt^2 + dr^2 + r^2 d \Omega \text{,} \]
where $d \Omega$ is the standard Euclidean metric on $\Sph^2$.
The standard future outgoing (truncated) null cones in Minkowski spacetime are given by
\[ \mc{N} = \{ t - r = c \text{,} \quad r \geq r_0 \} \text{,} \qquad c \in \R \text{,} \quad r_0 > 0 \text{.} \]
An affine parameter for $\mc{N}$ that is compatible with the general setting of Section \ref{sec.nc_gf} is $s = r$, i.e., the radial function.
The associated null vector fields $L$ and $\ul{L}$ are
\[ L = \partial_t + \partial_r \text{,} \qquad \ul{L} = \partial_t - \partial_r \text{.} \]

On $\mc{N}$, the induced horizontal metric is given by $\mind [v] = v^2 d \Omega$.
The values of the Ricci coefficients for this foliation of $\mc{N}$ are well-known:
\[ \chi = r^{-1} \mind \text{,} \qquad \ul{\chi} = - r^{-1} \mind \text{,} \qquad \zeta \equiv 0 \text{.} \]
Moreover, since Minkowski spacetime has zero curvature, all the curvature components $\alpha$, $\beta$, $\rho$, $\sigma$, $\ul{\alpha}$ (with respect to this foliation) vanish.

These constructions can be generalized to Schwarzschild spacetimes.
We summarize the relevant computations for this case below; for further details, see \cite{hol:ult_schw}.
Fix a mass value $m \geq 0$ (the case $m = 0$ reduces to the above Minkowski setting).
The outer region, in which our null cone will lie, can be expressed as
\[ M = \R \times (2 m, \infty) \times \Sph^2 \text{,} \]
with the Schwarzschild metric $g$, given in standard coordinates by
\[ g = - \paren{ 1 - \frac{2m}{r} } dt^2 + \paren{ 1 - \frac{2m}{r} }^{-1} dr^2 + r^2 d \Omega \text{.} \]

To describe the canonical null cones, we first recall the ``tortoise" coordinate
\[ r^\ast = r + 2 m \log \paren{ \frac{r}{2m} - 1 } \text{.} \]
Then, the standard null cones can be expressed as
\[ \mc{N} = \{ t - r^\ast = c \text{,} \quad r \geq r_0 \} \text{,} \qquad c \in \R \text{,} \quad r_0 > 2 m \text{.} \]
Like in the Minkowski setting, one can take the radial function $s = r$ (not $r^\ast$) as an affine parameter for $\mc{N}$, which is compatible with the development in Section \ref{sec.nc_gf}.
Moreover, the associated null vector fields are
\[ L = \paren{ 1 - \frac{2m}{r} }^{-1} \partial_t + \partial_r \text{,} \qquad \ul{L} = \partial_t - \paren{ 1 - \frac{2m}{r} } \partial_r \text{.} \]

The induced metric is once again $\mind [v] = v^2 d \Omega^2$.
Next, one can compute the Ricci and curvature coefficients corresponding to the above affine foliation of $\mc{N}$.
For the Ricci coefficients, one obtains the following values:
\[ \chi = r^{-1} \mind \text{,} \qquad \ul{\chi} = - r^{-1} \paren{ 1 - \frac{2m}{r} } \mind \text{,} \qquad \zeta \equiv 0 \text{.} \]
The only nonvanishing curvature coefficient is $\rho$, whose value is
\[ \rho \equiv - \frac{2m}{r^3} \text{.} \]
Combining the above with the definition of $\mu$, we see that
\[ \mu \equiv \frac{2m}{r^3} \text{.} \]

In the remainder of this section, we will focus exclusively on near-Minkowski and near-Schwarzschild null cones.
More specifically, we will consider geodesically foliated null cones whose associated Ricci and curvature coefficients deviate very little, in some sense, from the above known Minkowski and Schwarzschild values.
The main theorems will state, roughly, that if the curvature components of an infinite truncated null cone $\mc{N}$ are close to their Schwarzschild values (for some fixed mass $m \geq 0$), and if the Ricci coefficients are also initially close to their Schwarzschild values (with the same $m$), then the Ricci coefficients will remain close to the corresponding Schwarzschild values on all of $\mc{N}$.
\footnote{This closeness will, of course, be at the level of curvature flux.}

\subsection{The Renormalized System} \label{sec.nc_renorm}

We return now to the abstract setting of Sections \ref{sec.nc_gf} and \ref{sec.nc_rc}.
To develop our main results, we must accomplish the following:
\begin{itemize}
\item Quantify the deviation of $\mc{N}$ from a Schwarzschild null cone.

\item Perform analysis on these deviations, which are expected to be small.
\end{itemize}
For this purpose, we wish to renormalize our system on $\mc{N}$ into an equivalent system.
In particular, we would like for our renormalized system to satisfy the following:
\begin{itemize}
\item The level spheres of the foliation have nearly identical area.

\item The interval of the foliation is finite rather than infinite.
\end{itemize}
Throughout, we fix a mass value $m \geq 0$, and we suppose that $s_0 > 2 m$.

We will accomplish this renormalization in two steps.
For the first step, we define
\[ \gamma = s^{-2} \mind \in \mc{C}^\infty \ul{T}^0_2 \mc{N} \text{,} \]
i.e., the rescaled horizontal metric.
In particular, this leads to the identities
\[ \gamma_{ab} = s^{-2} \mind_{ab} \text{,} \qquad \gamma^{ab} = s^2 \mind^{ab} \text{,} \qquad \epsilon_{ab} = s^{-2} \vind_{ab} \text{,} \qquad \epsilon^{ab} = s^2 \vind^{ab} \text{.} \]

\begin{remark}
Note that the above is a special case of an $s$-dependent rescaling of $\mind$.
Thus, other geometric quantities, such as $\nabla$, $\lapl$, $\mc{K}$, $\nabla_s$, etc., transform according to the general formulas presented in Sections \ref{sec.geom_riem} and \ref{sec.fol_conf}.
\end{remark}

In the second step, we apply a change of variables to $s$ by defining
\[ t: \mc{N} \rightarrow \R \text{,} \qquad t = 1 - \frac{s_0}{s} \text{.} \]
Note the level set $s = s_0$ corresponds to $t = 0$, while the limit $s \nearrow \infty$ corresponds to $t \nearrow 1$.
We can now consider $\mc{N}$ as a foliation in terms of $t$:
\[ \mc{N} \simeq [0, 1) \times \mc{S} \simeq [0, 1) \times \Sph^2 \text{.} \]
In particular, this reduces the problem of infinite null cones to a finite cylinder.
Note that $s$ and $t$ are related to each other as follows:
\[ s = \frac{s_0}{1 - t} \text{,} \qquad \frac{dt}{ds} = \frac{s_0}{s^2} = \frac{ (1 - t)^2 }{ s_0 } \text{,} \qquad \frac{ds}{dt} = \frac{ s_0 }{ (1 - t)^2 } = \frac{ s^2 }{ s_0 } \text{.} \]
To simplify matters, we let $S_\tau$ denote the corresponding level set of $t$:
\[ S_\tau = \{ z \in \mc{N} \mid t (z) = \tau \} = \mc{S}_{ \frac{s_0}{1 - \tau} } \text{.} \]

Note that this change from $s$ to $t$ leaves $\gamma$, $\epsilon$, $\nabla$, and $\mc{K}$ (with respect to $\gamma$) unchanged.
On the other hand, the vertical Lie and covariant derivatives with respect to $t$ differ from those with respect to $s$:
\[ \mc{L}_t = \frac{ds}{dt} \mc{L}_s = \frac{s^2}{s_0} \mc{L}_s \text{,} \qquad \nabla_t = \frac{ds}{dt} \nabla_s = \frac{s^2}{s_0} \nabla_s \text{.} \]
Here, the covariant operators $\nabla_t$ and $\nabla_s$ are with respect to $\gamma$.

Finally, we renormalize the remaining givens and unknowns of our system.
In other words, we make corresponding renormalizations for the Ricci and curvature coefficients by defining the following quantities:
\begin{itemize}
\item \emph{Renormalized Ricci coefficients:}
\[ H = s_0^{-1} ( \chi - s^{-1} \mind ) \text{,} \qquad Z = s_0^{-1} s \zeta \text{,} \qquad \ul{H} = s^{-1} \ul{\chi} + s^{-2} \paren{ 1 - \frac{2m}{s} } \mind \text{.} \]

\item \emph{Renormalized curvature coefficients:}
\[ A = s_0^{-2} s^2 \alpha \text{,} \qquad B = s_0^{-2} s^3 \beta \text{,} \qquad R = s_0^{-1} [ s^3 ( \rho + i \sigma ) + 2 m ] \text{,} \qquad \ul{B} = s \ul{\beta} \text{.} \]

\item \emph{Renormalized mass aspect function:}
\[ M = s_0^{-1} ( s^3 \mu - 2 m ) \text{.} \]
\end{itemize}
Note these quantities measure the deviation from Schwarzschildean values.

From the above definitions, we can compute how norms change when we switch from the physical ($\mind$, $s$)-system to the renormalized ($\gamma$, $t$)-system.
In particular, given $p, q \in [1, \infty]$ and $\Psi \in \mc{C}^\infty \ul{T}^r_l \mc{N}$, we have
\begin{equation} \label{eq.norm_renorm} \| \Psi \|_{ L^{p, q}_{t, x} } = s_0^\frac{1}{p} \| s^{l - r - \frac{2}{q} - \frac{2}{p}} \Psi \|_{ \sorm{L}^{p, q}_{s, x} } \text{,} \qquad \| \Psi \|_{ L^{q, p}_{x, s} } = s_0^\frac{1}{p} \| s^{l - r - \frac{2}{q} - \frac{2}{p}} \Psi \|_{ \sorm{L}^{q, p}_{x, s} } \text{.} \end{equation}
Other norm relations can be similarly computed.

We can also compute the second fundamental form $k$ associated with our foliation, with respect to the metric $\gamma$ and with respect to $t$.
This yields
\begin{equation} \label{eq.sff_renorm} k = \frac{1}{2} \mf{L}_t \gamma = \frac{1}{2} \frac{s^2}{s_0} \mf{L}_s ( s^{-2} \mind ) = \frac{1}{s_0} ( \chi - s^{-1} \mind ) = H \text{.} \end{equation}

\subsection{The Renormalized Structure Equations} \label{sec.nc_rstr}

We can now reformulate the structure equations in terms of $\gamma$, $t$, and our renormalized quantities.

\begin{proposition} \label{thm.structure_renorm}
Assume all geometric quantities are defined with respect to $\gamma$, $\epsilon$, and $t$.
Then, the following structure equations hold on $\mc{N}$.
\begin{itemize}
\item Evolution equations:
\begin{align}
\label{eq.structure_renorm_ev} \nabla_t H_{ab} &= - \gamma^{cd} H_{ac} H_{bd} - A_{ab} \text{,} \\
\notag \nabla_t Z_a &= -2 \gamma^{bc} H_{ab} Z_c - B_a \text{,} \\
\notag \nabla_t \ul{H}_{ab} &= - ( \nabla_a Z_b + \nabla_b Z_a ) - \frac{1}{2} \gamma^{cd} ( H_{ac} \ul{H}_{bd} + H_{bc} \ul{H}_{ad} ) \\
\notag &\qquad + \paren{ 1 - \frac{2m}{s} } H_{ab} + 2 (1 - t) Z_a Z_b + ( \real R ) \gamma_{ab} \text{.}
\end{align}

\item Elliptic equations:
\begin{align}
\label{eq.structure_renorm_ell} \mc{D}_2 \hat{H}_a &= - (1 - t) B_a + Z_a + \frac{1}{2} \nabla_a ( \trace H ) + \frac{1}{2} (1 - t) ( \trace H ) Z_a \\
\notag &\qquad - (1 - t) \gamma^{bc} \hat{H}_{ab} Z_c \text{,} \\
\notag \mc{D}_1 Z &= - R - M + \frac{1}{2} ( \gamma^{ac} + i \epsilon^{ac} ) \gamma^{bd} \hat{H}_{ab} \ul{\hat{H}}_{cd} \text{.}
\end{align}

\item Gauss-Codazzi equations:
\begin{align}
\label{eq.structure_renorm_gc} \nabla_b H_{ac} - \nabla_c H_{ab} &= - \epsilon_{bc} \epsilon_a{}^d [ (1 - t) B_d - Z_d ] + (1 - t) ( H_{ab} Z_c - H_{ac} Z_b ) \text{,} \\
\notag \mc{K} - 1 &= - (1 - t) (\real R) + \frac{1 - t}{2} \paren{ 1 - \frac{2m}{s} } \trace H - \frac{1}{2} \trace \ul{H} \\
\notag &\qquad + \frac{1 - t}{2} ( \gamma^{ac} \gamma^{bd} - \gamma^{ab} \gamma^{cd} ) H_{ab} \ul{H}_{cd} \text{.}
\end{align}

\item Derivative evolution equations:
\begin{align}
\label{eq.structure_renorm_evd} \nabla_t \nabla_a ( \trace H ) &= - \gamma^{bc} H_{ab} \nabla_c ( \trace H ) - 2 \gamma^{bc} \gamma^{de} H_{bd} \nabla_a H_{ce} \text{,} \\
\notag \nabla_t M &= - \frac{3}{2} ( \trace H ) \paren{ M + \frac{2 m}{s_0} } - 2 (1 - t) \gamma^{ab} Z_a B_b \\
\notag &\qquad + 2 \gamma^{ab} \gamma^{cd} \hat{H}_{ac} \nabla_b Z_d - 2 (1 - t) \gamma^{ab} \gamma^{cd} \hat{H}_{ac} Z_b Z_d \\
\notag &\qquad + 2 \gamma^{ab} Z_b \nabla_a ( \trace H ) + \frac{3}{2} \gamma^{ab} [ (1 - t) \trace H + 2 ] Z_a Z_b \\
\notag &\qquad - \frac{1}{4} \gamma^{ab} \gamma^{cd} \brak{ \trace \ul{H} - 2 \paren{ 1 - \frac{2m}{s} } } \hat{H}_{ac} \hat{H}_{bd} \text{.}
\end{align}

\item Null Bianchi equations:
\begin{align}
\label{eq.structure_renorm_nb} \nabla_t B_a &= (1 - t)^{-1} \mc{D}_2 A_a - 2 ( \trace H ) B_a + \gamma^{bc} Z_b A_{ac} \text{,} \\
\notag \nabla_t R &= \mc{D}_1 B - \frac{3}{2} ( \trace H ) \paren{ R - \frac{2m}{s_0} } \\
\notag &\qquad - ( \gamma^{ab} - i \epsilon^{ab} ) \brak{ (1 - t) Z_a B_b + \frac{1}{2} \gamma^{cd} \ul{H}_{ac} A_{bd} } \text{,} \\
\notag \nabla_t \ul{B}_a &= \mc{D}_1^\ast \bar{R} - ( \trace H ) \ul{B}_a + 3 (1 - t) Z_a \paren{ \real R - \frac{2m}{s_0} } \\
\notag &\qquad - 3 (1 - t) \epsilon_a{}^b Z_b ( \imag R ) + 2 (1 - t) \gamma^{bc} \ul{\hat{H}}_{ab} B_c \text{.}
\end{align}
\end{itemize}
\end{proposition}

\begin{proof}
We show the first equation in \eqref{eq.structure_renorm_ev}; the remaining equations are derived similarly.
Letting $J = \chi - s^{-1} \mind$, the first equation in \eqref{eq.structure_ev} implies
\begin{align*}
\nasla_s J_{ab} &= - \mind^{cd} \chi_{ac} \chi_{bd} - \alpha_{ab} + s^{-2} \mind = - \mind^{cd} J_{ac} J_{bd} - 2 s^{-1} J_{ab} - \alpha_{ab} \text{.}
\end{align*}
The above can be rewritten as
\begin{align*}
\nasla_s ( s^2 J )_{ab} &= - s^2 \mind^{cd} J_{ac} J_{bd} - s^2 \alpha_{ab} = - \gamma^{cd} J_{ac} J_{bd} - s^2 \alpha_{ab} \text{.}
\end{align*}
From the relation \eqref{eq.conf_ev} between $\nasla_s$ and $\nabla_s$, we obtain
\[ \nabla_s J_{ab} = s^{-2} \nasla_s ( s^2 J )_{ab} = s^{-2} \gamma^{cd} J_{ac} J_{bd} - \alpha_{ab} \text{.} \]
Since $H = s_0^{-1} J$ and $\nabla_t = s_0^{-1} s^2 \nabla_s$, we obtain, as desired,
\begin{align*}
\nabla_t H_{ab} &= s_0^{-2} \gamma^{cd} J_{ac} J_{bd} - s_0^{-2} s^2 \alpha_{ab} = \gamma^{cd} H_{ac} H_{bd} - A_{ab} \text{.} \qedhere
\end{align*}
\end{proof}

\begin{remark}
Note in particular that the quantities
\[ \frac{2m}{s_0} \text{,} \qquad 1 - \frac{2m}{s} \text{,} \]
which are constant on any $S_\tau$, always lie between $0$ and $1$.
\end{remark}

\begin{remark}
We also remark that the trace and traceless parts of $H$ and $\ul{H}$ are always taken with respect to the renormalized metric $\gamma$.
\end{remark}

\section{The Main Results} \label{sec.main}

With the definitions and relations in place, we can now state the main results of this paper.
First, we state the main theorem in terms of the \emph{renormalized} setting.
In fact, this \emph{renormalized} result is the one we will directly deal with in the final section of this paper.
Finally, we reverse the renormalization procedure in order to obtain the precise ``physical" version of the main theorem of this paper, i.e., in terms of the spacetime metric and geometry.

Again, we assume the definitions and constructions of Section \ref{sec.nc}.
We also assume $m \geq 0$, and we suppose the initial radius $s_0$ of $\mc{S}$ satisfies $s_0 > 2 m$.

\subsection{The Renormalized Main Theorem} \label{sec.main_rthm}

The first main theorem we state is in the renormalized setting, where all the analysis will take place.

\begin{theorem} \label{thm.nc_renorm}
Assume the constructions of Sections \ref{sec.nc_gf}, \ref{sec.nc_rc}, and \ref{sec.nc_renorm}, and assume all quantities are defined with respect to $\gamma$ and $t$.
Assume also the following:
\begin{itemize}
\item $0 \leq m < s_0 / 2$.

\item $\mc{S} = S_0$, along with the Riemannian metric $\gamma [0]$, satisfies \ass{r2}{C, N}.

\item The following curvature flux bounds hold on $\mc{N}$:
\begin{equation} \label{eq.renorm_ass_flux} \| A \|_{ L^{2, 2}_{t, x} } + \| B \|_{ L^{2, 2}_{t, x} } + \| R \|_{ L^{2, 2}_{t, x} } + \| \ul{B} \|_{ L^{2, 2}_{t, x} } \leq \Gamma \text{,} \end{equation}

\item The following initial value bounds hold on $S_0$:
\begin{align}
\label{eq.renorm_ass_init} \| ( \trace H ) [0] \|_{ L^\infty_x } + \| H [0] \|_{ H^{1/2}_x } + \| Z [0] \|_{ H^{1/2}_x } &\leq \Gamma \text{,} \\
\notag \| \ul{H} [0] \|_{ B^0_x } + \| \nabla ( \trace H ) [0] \|_{ B^0_x } + \| M [0] \|_{ B^0_x } &\leq \Gamma \text{.}
\end{align}
\end{itemize}
If $\Gamma \ll 1$ is sufficiently small, with respect to $C$ and $N$, then:
\begin{itemize}
\item $\mc{N}$, with respect to the $t$-foliation and the renormalized metric $\gamma$, satisfies the conditions \ass{F2}{C, N, D} and \ass{K}{1, D}, where $D \simeq \Gamma \ll 1$.

\item The following bounds hold for the connection coefficients on $\mc{N}$:
\begin{align}
\label{eq.renorm_est} \| \trace H \|_{ L^{\infty, \infty}_{t, x} } &\lesssim \Gamma \text{,} \\
\notag \| H \|_{ N^1_{t, x} \cap L^{\infty, 2}_{x, t} \cap H^{\infty, 1/2}_{t, x} } + \| Z \|_{ N^1_{t, x} \cap L^{\infty, 2}_{x, t} \cap H^{\infty, 1/2}_{t, x} } &\lesssim \Gamma \text{,} \\
\notag \| \nabla_t \nabla ( \trace H ) \|_{ L^{2, 1}_{x, t} } + \| \nabla_t M \|_{ L^{2, 1}_{x, t} } &\lesssim \Gamma \text{,} \\
\notag \| \nabla ( \trace H ) \|_{ L^{2, \infty}_{x, t} \cap B^{\infty, 0}_{t, x} } + \| M \|_{ L^{2, \infty}_{x, t} \cap B^{\infty, 0}_{t, x} } + \| \ul{H} \|_{ L^{2, \infty}_{x, t} \cap B^{\infty, 0}_{t, x} } &\lesssim \Gamma \text{,} \\
\notag \| \nabla H \|_{ N^{0 \star}_{t, x} + B^{2, 0}_{t, x} } + \| \nabla Z \|_{ N^{0 \star}_{t, x} + B^{2, 0}_{t, x} } + \| \nabla_t \ul{H} \|_{ N^{0 \star}_{t, x} + B^{2, 0}_{t, x} } &\lesssim \Gamma \text{,} \\
\notag \| \mc{K} - 1 \|_{ H^{\infty, -1/2}_{t, x} } &\lesssim \Gamma \text{.}
\end{align}

\item Furthermore, for each $0 \leq \tau < 1$, we have the refined curvature estimate
\begin{equation} \label{eq.renorm_est_curv} \| \mc{K} [\tau] - 1 \|_{ H^{-1/2}_x } \lesssim \| \trace \ul{H} [\tau] \|_{ L^2_x } + (1 - \tau) \Gamma \text{.} \end{equation}\
\end{itemize}
\end{theorem}

\begin{remark}
In particular, the conclusions of Theorem \ref{thm.nc_renorm} imply that each $t$-level set $(\Sph^2, \gamma [\tau])$ satisfies the \ass{r2}{} and \ass{k}{} conditions uniformly.
\end{remark}

Roughly speaking, Theorem \ref{thm.nc_renorm} controls the geometry of $\mc{N}$, in the renormalized setting, all the way up to infinity, i.e., $0 \leq t < 1$.
Furthermore, we can use the results of Theorem \ref{thm.nc_renorm} in order to obtain similar control \emph{at infinity}, that is, at $t = 1$.
This is a fairly standard but technical process that involves generating the appropriate limits of our renormalized quantities, though the argument here is a bit further inconvenienced by the fact that the geometries of the $S_\tau$'s, and hence the spaces and norms under consideration, vary with respect to $\tau$.

\begin{corollary} \label{thm.nc_renorm_lim}
Assume the same hypotheses as in Theorem \ref{thm.nc_renorm}.
\begin{itemize}
\item As $\tau \nearrow 1$, the metrics $\gamma [\tau]$ on $\mc{S}$ converge uniformly to a continuous metric $\gamma [1]$ on $\mc{S}$.
Furthermore, the Jacobians $\mc{J} [\tau]$, and hence the volume forms $\epsilon [\tau]$, also converge as $\tau \nearrow 1$ to continuous limits $\mc{J} [1]$ and $\epsilon [1]$.

\item Let $\nabla^\tau$ denote the Levi-Civita connection associated with $\gamma [\tau]$.
Given any $F \in \mc{C}^\infty T^r_l \mc{S}$, the differential $\nabla^\tau F$ has an $L^2$-limit as $\tau \nearrow 1$.
\footnote{As all the $\mc{J} [\tau]$'s, where $0 \leq \tau \leq 1$, are comparable (see \cite[Sect. 4.3]{shao:stt}), then this convergence is in fact independent of the volume form $\epsilon [\tau]$ chosen for the $L^2$-norm.}
In other words, the convergence of $\gamma [\tau]$ to $\gamma [1]$ as $\tau \nearrow 1$ is ``in $H^1$".
\end{itemize}
In addition, we have the following limits for scalar quantities:
\begin{itemize}
\item $\trace H [\tau]$ has a continuous uniform limit, $\trace H [1]$, as $\tau \nearrow 1$.

\item $H$, $Z$, $\ul{H}$, $\nabla (\trace H)$, and $M$ have $L^2$-limits as $\tau \nearrow 1$.
\end{itemize}
\end{corollary}

The proof is given in Section \ref{sec.main_lim}.
Below, we list a few related remarks:
\begin{itemize}
\item With more careful arguments, one can show more regularity for the limits in Corollary \ref{thm.nc_renorm_lim}.
In general, the limit of a quantity at $t = 1$ can have the same regularity as its initial value.
For example, as one assumes $B^0$-control for $\ul{H} [0]$, then $\ul{H} [1]$ should also be a $B^0$-type limit for the $\ul{H} [\tau]$'s.
\footnote{This can shown rigorously using the coordinate-based Besov-type norms from \cite{shao:stt}.}
Similarly, one expects $H^{1/2}$-type limits for $H [\tau]$ and $Z [\tau]$ as $\tau \nearrow 1$.

\item It is also possible to show that $(\mc{S}, \gamma [1])$ satisfies \ass{r2}{}.
This is a consequence of the fact that the $(\mc{S}, \gamma [\tau])$'s satisfy \ass{r2}{} uniformly.

\item From the second equation in \eqref{eq.structure_renorm_gc}, one sees that the Gauss curvatures $\mc{K} [\tau]$ converge, as $\tau \nearrow 1$, weakly to $1 - (\trace \ul{H}) / 2$; see also \eqref{eq.renorm_est_curv}.
\end{itemize}

Some of the limits from Corollary \ref{thm.nc_renorm_lim}, in principle, have physical significance.
For instance, for certain asymptotically spherical foliations, the limits for $M$ and $Z$ can be connected to the Bondi mass and the angular momentum associated with $\mc{N}$.
This will be discussed in further detail in \cite{alex_shao:bondi}.

\subsection{Proof of Corollary \ref{thm.nc_renorm_lim}} \label{sec.main_lim}

Fix $X, Y \in \mc{C}^\infty T^1_0 \mc{S}$.
For convenience, we also denote by $X$ and $Y$ their \emph{equivariant transports}, i.e., the fields
\[ X, Y \in \mc{C}^\infty \ul{T}^1_0 \mc{N} \text{,} \qquad X [\tau] = X \text{.} \]
In addition, fix $0 \leq \tau < \tau^\prime < 1$ and $x \in \mc{S}$.
First, we estimate the difference
\begin{align*}
| \gamma [\tau^\prime] (X, Y) - \gamma [\tau] (X, Y) | |_x &\lesssim \int_\tau^{\tau^\prime} | ( \mf{L}_t \gamma ) [w] (X, Y) | |_x d w \\
&\lesssim \int_\tau^{\tau^\prime} | H [w] | | X | | Y | |_x dw \text{,}
\end{align*}
where the $|\cdot|$'s in the integrand on the right-hand side is with respect to $\gamma [w]$.
Moreover, by the \ass{r0}{} condition uniformly satisfied by the $(\mc{S}, \gamma [w])$'s,
\[ \| X \|_{ L^{\infty, \infty}_{t, x} } + \| Y \|_{ L^{\infty, \infty}_{t, x} } \lesssim 1 \text{,} \]
independently of $w$.
Thus,
\[ | \gamma [\tau^\prime] (X, Y) - \gamma [\tau] (X, Y) | |_x \lesssim ( \tau^\prime - \tau )^\frac{1}{2} \| H \|_{ L^{\infty, 2}_{x, t} } \text{,} \]
and it follows that $\gamma [\tau] (X, Y)$ converges uniformly as $\tau \nearrow 1$ to a continuous limit.
Since this limit clearly depends $\mc{C}^\infty \mc{S}$-linearly on both $X$ and $Y$, it follows that this limit defines a continuous (Riemannian) metric $\gamma [1]$ on $\mc{S}$.

Next, since $\mf{L}_t \mc{J} = \trace H \cdot \mc{J}$ by \eqref{eq.jacobian}, and since $\epsilon [\tau] = \mc{J} [\tau] \epsilon [0]$, by an argument similar to the above, we also obtain analogous continuous limits for $\mc{J}$ and $\epsilon$.

For the limiting connection, we again use $F$ to also denote the equivariant transport of $F$ as in the hypothesis.
Applying the first identity in \eqref{eq.comm}, we obtain
\[ | \mf{L}_t \nabla F | \lesssim | \nabla \mf{L}_t F | + | \nabla k | | F | \lesssim | \nabla k | | F | \text{.} \]
As a result, for any $G \in \mc{C}^\infty T^{l+1}_r \mc{S}$, we can estimate
\begin{align*}
\| \nabla^{\tau^\prime} F (G) - \nabla^\tau F (G) \|_{ L^2_x } &\lesssim \norm{ \int_\tau^{\tau^\prime} | \mf{L}_t \nabla F [w] | | G | dw }_{ L^2_x } \\
&\lesssim \norm{ \int_\tau^{\tau^\prime} | \nabla k [w] | dw }_{ L^2_x } \| F \|_{ L^{\infty, \infty}_{t, x} } \| G \|_{ L^{\infty, \infty}_{t, x} } \\
&\lesssim ( \tau^\prime - \tau )^\frac{1}{2} \| \nabla H \|_{ L^{2, 2}_{t, x} } \text{.}
\end{align*}
Note that since the $\mc{J} [w]$'s are all mutually comparable, the specific choice of $\gamma [w]$ with which to define the above $L^2$-norms does not matter.
From the estimate, it follows that $\nabla^\tau F (G)$ has an $L^2$-limit as $\tau \nearrow 1$.
As before, it is clear that this limit is $\mc{C}^\infty \mc{S}$-linear with respect to $G$, hence we obtain an $L^2$-tensor field $\nabla^1 F$, which represents the ($L^2$-)limit of $\nabla^\tau F$ as $\tau \nearrow 1$.
Furthermore, from distributional considerations, one can show that this operator $\nabla^1$ corresponds to the (weak) Levi-Civita connection associated with $\gamma [1]$; see, for instance, \cite{lefl_mar:dist_curv}.

It remains only to derive limits for the renormalized Ricci coefficients.
The arguments for all the quantities are rather similar.
For example,
\begin{align*}
\| M [\tau^\prime] - M [\tau] \|_{ L^2_x } \lesssim \norm{ \int_\tau^{\tau^\prime} \nabla_t M [w] dw }_{ L^2_x } \lesssim \| \nabla_t M \|_{ L^{2, 1}_{x, t} [\tau, \tau^\prime] } \text{,}
\end{align*}
where the norm on the right-hand is over the segment of $\mc{N}$ corresponding to the region $\tau \leq t \leq \tau^\prime$.
By the estimate for $\nabla_t M$ in \eqref{eq.renorm_est}, the right-hand side becomes arbitrarily small as $\tau^\prime - \tau \rightarrow 0$.
It follows that $M [\tau]$ has an $L^2$-limit as $\tau \nearrow 1$.

The remaining $L^2$-estimates for $H$, $Z$, $\ul{H}$, and $\nabla (\trace H)$ are obtained similarly.
The only complication is the tensorial nature of the limits, which can be handled in a similar manner as in the arguments for $\gamma [1]$ and $\nabla^1$.
Finally, to see that $\trace H$ has a continuous limit, one needs only note that by \eqref{eq.structure_renorm_ev},
\[ \| \nabla_t (\trace H) \|_{ L^{\infty, 1}_{x, t} } \lesssim \| H \|_{ L^{\infty, 2}_{x, t} }^2 < \infty \text{.} \]

\subsection{The Physical Main Theorem} \label{sec.main_pthm}

The final task, besides the proof of Theorem \ref{thm.nc_renorm}, is to translate the contents of Theorem \ref{thm.nc_renorm} back into the physical setting.
In other words, we want to state its hypotheses and conclusions in terms of $\mind$, $s$, the Ricci coefficients ($\chi$, $\ul{\chi}$, $\zeta$), and the curvature components ($\alpha$, $\beta$, $\rho$, $\sigma$, $\ul{\beta}$).

Since the level spheres $\mc{S}_v$, with respect to $\mind$, do not generally have near-unit area, we must, for technical reasons, define our Sobolev and Besov norms in the physical setting slightly differently.
To be more specific, we wish for our norms here to be natural rescalings of the corresponding norms defined in Sections \ref{sec.geom_glp}.

First, given $b \in \R$, $\Psi \in \mc{C}^\infty \ul{T}^r_l \mc{N}$, and $v \geq s_0$, we define the Sobolev norm
\begin{equation} \label{eq.int_sob_ex} \| \Psi [v] \|_{ \sorm{\mc{H}}^b_x } = \| ( v^{-2} I - \lasl )^\frac{b}{2} \Psi [v] \|_{ \sorm{L}^2_x } \text{,} \end{equation}
Similarly, we define the Besov norm
\begin{equation} \label{eq.int_besov_ex} \| \Psi [v] \|_{ \sorm{\mc{B}}^b_x } = v^{-b} \sum_{k \geq 0} 2^{bk} \| P_k \Psi [v] \|_{ \sorm{L}^2_x } + v^{-b} \| P_{< 0} \Psi [v] \|_{ \sorm{L}^2_x } \text{,} \end{equation}
where the $P_k$'s and $P_{< 0}$ are the L-P operators, \emph{with respect to the renormalized metric $\gamma$}.
\footnote{More specifically, we consider the spectral decomposition of $\lapl$, not $\lasl$.}
One can compute directly the following scaling relations:
\begin{equation} \label{eq.norm_renorm_ex} \| \Psi [v] \|_{ H^b_x } = v^{l - r - 1 + b} \| \Psi [v] \|_{ \sorm{\mc{H}}^b_x } \text{,} \qquad \| \Psi [v] \|_{ B^b_x } = v^{l - r - 1 + b} \| \Psi [v] \|_{ \sorm{\mc{B}}^b_x } \text{.} \end{equation}
The left hand sides in \eqref{eq.norm_renorm_ex} are stated with respect to $\gamma$ and $s$.

We can now state fully the physical version of our main theorem.

\begin{theorem} \label{thm.nc_phys}
Assume the constructions of Sections \ref{sec.nc_gf}, \ref{sec.nc_rc}, and \ref{sec.nc_renorm}; and assume all quantities are defined with respect to $\mind$ and $s$.
Assume also the following:
\begin{itemize}
\item $0 \leq m < s_0 / 2$.

\item $\mc{S}$, with the normalized metric $s_0^{-2} \mind [s_0]$, satisfies \ass{r2}{C, N}.

\item The following curvature flux bounds hold:
\begin{align}
\label{eq.phys_ass_flux} s_0^{-\frac{3}{2}} \| s^2 \alpha \|_{ \sorm{L}^{2, 2}_{s, x} } + s_0^{-\frac{3}{2}} \| s^2 \beta \|_{ \sorm{L}^{2, 2}_{s, x} } + s_0^\frac{1}{2} \| \ul{\beta} \|_{ \sorm{L}^{2, 2}_{s, x} } &\leq \Gamma \text{,} \\
\notag s_0^{-\frac{1}{2}} \norm{ s \paren{ \rho + \frac{2 m}{s^3} } }_{ \sorm{L}^{2, 2}_{s, x} } + s_0^{-\frac{1}{2}} \| s \sigma \|_{ \sorm{L}^{2, 2}_{s, x} } &\leq \Gamma \text{,}
\end{align}

\item The following initial value bounds hold on $\mc{S} = \mc{S}_{s_0}$:
\begin{align}
\label{eq.phys_ass_init} s_0 \norm{ \paren{ \trase \chi - \frac{2}{s_0} } [s_0] }_{ \sorm{L}^\infty_x } + s_0^\frac{1}{2} \| ( \chi - s_0^{-1} \mind ) [s_0] \|_{ \sorm{\mc{H}}^{1/2}_x } &\leq \Gamma \text{,} \\
\notag s_0^\frac{1}{2} \| \zeta [s_0] \|_{ \sorm{\mc{H}}^{1/2}_x } + \norm{ \brak{ \ul{\chi} + s_0^{-1} \paren{ 1 - \frac{2m}{s_0} } \mind } [s_0] }_{ \sorm{\mc{B}}^0_x } &\leq \Gamma \text{,} \\
\notag s_0 \| \nasla ( \trase \chi ) [s_0] \|_{ \sorm{\mc{B}}^0_x } + s_0 \norm{ \paren{ \mu - \frac{2m}{s_0^3} } [s_0] }_{ \sorm{\mc{B}}^0_x } &\leq \Gamma \text{.}
\end{align}
\end{itemize}
If $\Gamma$ is sufficiently small, with respect to $C$ and $N$, then
\begin{align}
\label{eq.phys_est} s_0^{-1} \norm{ s^2 \paren{ \trase \chi - \frac{2}{s} } }_{ \sorm{L}^{\infty, \infty}_{s, x} } &\lesssim \Gamma \text{,} \\
\notag s_0^{-\frac{1}{2}} \| s ( \chi - s^{-1} \mind ) \|_{ \sorm{L}^{\infty, 2}_{x, s} } + s_0^{-\frac{1}{2}} \| s \zeta \|_{ \sorm{L}^{\infty, 2}_{x, s} } &\lesssim \Gamma \text{,} \\
\notag s_0^{-1} \| s^\frac{3}{2} ( \chi - s^{-1} \mind ) \|_{ \sorm{L}^{4, \infty}_{x, s} } + s_0^{-1} \| s^\frac{3}{2} \zeta \|_{ \sorm{L}^{4, \infty}_{x, s} } &\lesssim \Gamma \text{,} \\
\notag s_0^{-\frac{3}{2}} \| \nasla_s [ s^2 ( \chi - s^{-1} \mind ) ] \|_{ \sorm{L}^{2, 2}_{s, x} } + s_0^{-\frac{1}{2}} \| s \nasla \chi \|_{ \sorm{L}^{2, 2}_{s, x} } + s_0^{-\frac{1}{2}} \| \chi - s^{-1} \mind \|_{ \sorm{L}^{2, 2}_{s, x} } &\lesssim \Gamma \text{,} \\
\notag s_0^{-\frac{3}{2}} \| \nasla_s ( s^2 \zeta ) \|_{ \sorm{L}^{2, 2}_{s, x} } + s_0^{-\frac{1}{2}} \| s \nasla \zeta \|_{ \sorm{L}^{2, 2}_{s, x} } + s_0^{-\frac{1}{2}} \| \zeta \|_{ \sorm{L}^{2, 2}_{s, x} } &\lesssim \Gamma \text{,} \\
\notag s_0^{-1} \| s^2 \nasla ( \trase \chi ) \|_{ \sorm{L}^{2, \infty}_{x, s} } + s_0^{-1} \norm{ s^2 \paren{ \mu - \frac{2 m}{s^3} } }_{ \sorm{L}^{2, \infty}_{x, s} } &\lesssim \Gamma \text{,} \\
\notag \norm{ \ul{\chi} + s^{-1} \paren{ 1 - \frac{2 m}{s} } \mind }_{ \sorm{L}^{2, \infty}_{x, s} } &\lesssim \Gamma \text{.}
\end{align}
\end{theorem}

Thanks to \eqref{eq.norm_renorm} and \eqref{eq.norm_renorm_ex}, Theorem \ref{thm.nc_phys} follows immediately from Theorem \ref{thm.nc_renorm}.
Indeed, all the assumptions from Theorem \ref{thm.nc_renorm} are the same as those in Theorem \ref{thm.nc_phys}, after applying the aforementioned renormalizations and changes of variables.
Moreover, the conclusions from Theorem \ref{thm.nc_renorm} include those in Theorem \ref{thm.nc_phys}.

\begin{remark}
One can also define variants of the $\sorm{\mc{H}}$- and $\sorm{\mc{B}}$-norms, involving integrals over both $s$ and the spheres.
Then, the Sobolev and Besov estimates in \eqref{eq.renorm_est} will also translate directly to corresponding Sobolev and Besov estimates in \eqref{eq.phys_est}.
\end{remark}

\subsection{Some Extensions and Consequences} \label{sec.main_ext}

Recall that throughout this section, we have assumed that \emph{$\mc{N}$ remains a smooth null hypersurface of $(M, g)$ extending to infinity}.
With a more carefully constructed setup, however, we can actually assume a bit less.
In particular, we can show that null conjugate points cannot form along the null geodesics of $\mc{N}$.
A more precise statement is as follows:

\begin{proposition} \label{thm.nc_conj}
Suppose the assumptions of Theorem \ref{thm.nc_phys}---in particular the bounds \eqref{eq.phys_ass_flux} and \eqref{eq.phys_ass_init}---hold, but only on a \emph{finite} segment of $\mc{N}$,
\begin{align*}
\mc{N}_\delta = \{ p \in \mc{N} \mid t (p) < \delta \} \text{,} \qquad 0 \leq \delta < 1 \text{.}
\end{align*}
In particular, we assume smoothness only for $\mc{N}_\delta$.
\footnote{Note the connection and curvature coefficients need not be well-defined beyond $\mc{N}_\delta$.}
Then, \eqref{eq.phys_est} also holds on $\mc{N}_\delta$.
Furthermore, no null conjugate points can form on the upper boundary $t = \delta$.
\end{proposition}

The proof of Theorem \ref{thm.nc_conj} is analogous to that of Theorem \ref{thm.nc_renorm} (see Section \ref{sec.proof}), except all the estimates are localized to the finite region $t < \delta$.
This results in the estimates \eqref{eq.phys_est} on $\mc{N}_\delta$, and in particular the conclusion that $\trase \chi$ is uniformly bounded on $\mc{N}_\delta$.
That $\trase \chi$ does not blow up implies no null conjugate points can form on the upper boundary $t = \delta$; see \cite{haw_el:gr}.

By combining Theorem \ref{thm.nc_conj} and its proof with a more elaborate bootstrap argument, one can actually prove that $\mc{N}$ remains smooth.
For example, in finite settings, this type of result was essential to the main theorems in \cite{kl_rod:bdc, shao:bdc_nv}.
However, such statements are difficult with formulate cleanly, since one requires smoothness in order to even define the connection and curvature coefficients.

Throughout this section, we have used the affine parameter $s$ as a substitute for the radii $r$ of the level spheres of $s$, where $r$ is defined here as the function
\footnote{Recall that $s$ and $r$ coincide for the Schwarzschild null cones in Section \ref{sec.nc_ms}.}
\[ r: [s_0, \infty) \rightarrow \R \text{,} \qquad r (v) = \paren{ \frac{1}{4 \pi} \int_{ \mc{S}_v } d \vind }^\frac{1}{2} \text{.} \]
In other words, $4 \pi [ r (v) ]^2$ is precisely the area of $\mc{S}_v$, with respect to $\mind$.
Although $s$ is simpler to handle, $r$ is often more relevant for physical considerations (e.g., in the definition of the Hawking mass below).
Consequently, it is important to justify this heuristic that $s$ and $r$ are nearly the same.

To see this, we express $r / s$ in terms of $\gamma$ and $t$:
\[ \frac{r (v)}{ v } = \paren{ \frac{1}{ 4 \pi } \int_{ S_{ 1 - \frac{s_0}{v} } } d \epsilon }^\frac{1}{2} = \paren{ \frac{1}{4 \pi} \int_{ \Sph^2 } \mc{J} [ 1 - s_0 v^{-1} ] \cdot d \epsilon [0] }^\frac{1}{2} \text{.} \]
Recall $\mc{J}$ is the Jacobian factor from \eqref{eq.jacobian}.
Since $\mc{J}$ is uniformly close to $1$ on $\mc{N}$ (see \cite[Prop. 4.5]{shao:stt} or the proof of Corollary \ref{thm.nc_renorm_lim}), it follows that $r / s$ also remains uniformly near $1$.
Furthermore, since $\mc{J}$ has a continuous limit at $t \nearrow 1$ by Corollary \ref{thm.nc_renorm_lim}, it follows that $r / s$ has a limit, which is itself close to $1$, as $s \nearrow \infty$.

Finally, recall that the Hawking mass of the level sphere $\mc{S}_v$ is
\[ \mf{m} (v) = \frac{r (v)}{2} \paren{ 1 + \frac{1}{16 \pi} \int_{ \mc{S}_v } \trase \chi \cdot \trase \ul{\chi} \cdot d \vind } = \frac{r (v)}{8 \pi} \int_{ \mc{S}_v } \mu d \vind \text{,} \]
where $r$ is as before; see \cite{chr_kl:stb_mink}.
Since $M$ has an $L^2$-limit as $t \nearrow 1$ by Corollary \ref{thm.nc_renorm_lim}, it follows that the Hawking masses $\mf{m} (v)$ also have a limit as $v \nearrow \infty$.

In the case that the $\mc{S}_v$'s become asymptotically round, in the sense that the \emph{renormalized} Gauss curvatures $\mc{K}$ converge to $1$ as $v \nearrow \infty$, the limit of the Hawking masses should correspond to the \emph{Bondi energy} associated with $\mc{N}$.
\footnote{In the case that a viable future null infinity $\mc{I}^+$ is defined, the \emph{Bondi energy} represents the energy remaining in the ambient spacetime at the cut where $\mc{N}$ intersects $\mc{I}^+$.}
In our setting, however, there is no reason to expect that these $\mc{S}_v$'s in general actually satisfy this asymptotic roundness property.
In the sequel \cite{alex_shao:bondi} to this paper, we undertake the challenge of \emph{constructing} such an asymptotically round family along $\mc{N}$, given the same assumptions as in Theorem \ref{thm.nc_phys}.
Furthermore, we control the resulting Bondi energy by the weighted flux of curvature and the initial data of $\mc{N}$.

\section{Proof of Theorem \ref{thm.nc_renorm}} \label{sec.proof}

This section is dedicated solely to the proof of Theorem \ref{thm.nc_renorm} (and hence Theorem \ref{thm.nc_phys}).
From now on, we assume the hypotheses of Theorem \ref{thm.nc_renorm}.
Unless otherwise stated, all geometric objects are defined with respect to $\gamma$ and $t$.

In addition, given any $\delta \in (0, 1)$, we define $\mc{N}_\delta$ to be the initial subsegment $[0, \delta] \times \Sph^2$ of $\mc{N}$, i.e., the segment of $\mc{N}$ with $0 \leq t \leq \delta$.

\subsection{Outline} \label{sec.proof_outline}

As in previous works (e.g., \cite{kl_rod:cg, wang:cg}), the foundation of the proof is a grand bootstrap, or continuity, argument.
In this process, we impose a set of \emph{bootstrap assumptions}, stated as estimates, on our system.
We then use these assumptions in order to prove strictly better versions of these estimates.
Below, we give a more precise description of this bootstrap argument.

We say that \ass{BA}{\delta, \Delta} holds iff the following estimate holds on $\mc{N}_\delta$:
\begin{align*}
\| H \|_{ N^1_{t, x} \cap L^{\infty, 2}_{x, t} } + \| Z \|_{ N^1_{t, x} \cap L^{\infty, 2}_{x, t} } + \| \nabla H \|_{ N^{0 \star}_{t, x} + B^{2, 0}_{t, x} } + \| \nabla Z \|_{ N^{0 \star}_{t, x} + B^{2, 0}_{t, x} } &\leq \Delta \text{.}
\end{align*}
\emph{In the above inequality, all the norms are evaluated only over the sub-cone $\mc{N}_\delta$.}
The proof of Theorem \ref{thm.nc_renorm} can now be outlined as follows:
\begin{enumerate}
\item First, note that \ass{BA}{\delta, \Delta} holds trivially when $\delta \searrow 0$.

\item Fix an arbitrary $\delta \in (0, 1)$, and assume \ass{BA}{\delta, \Delta} holds, with $\Gamma \ll \Delta \ll 1$.

\item Using the structure equations, i.e., Proposition \ref{thm.structure_renorm}, we show that the quantities on the left-hand side of the \ass{BA}{\delta, \Delta} condition can in fact be bounded by a constant (depending on $C$ and $N$) times $\Gamma + \Delta^2$.

\item With sufficiently small $\Gamma$ and $\Delta$, the above implies that \ass{BA}{\delta, \Delta/2} holds.

\item The implication \ass{BA}{\delta, \Delta} $\Rightarrow$ \ass{BA}{\delta, \Delta/2} established above for arbitrary $\delta \in (0, 1)$, combined with a standard continuity argument, implies that \ass{BA}{\delta, \Delta/2} holds \emph{unconditionally} for every $\delta \in (0, 1)$.
\end{enumerate}
The above will suffice to complete the proof of Theorem \ref{thm.nc_renorm}.

The points (1), (4), (5) are standard features of continuity arguments.
Thus, in the remainder of this section, we will focus entirely on the main steps (2) and (3) of the bootstrap argument: establishing improved estimates from the bootstrap assumptions.
We fix throughout $\delta \in (0, 1)$, and we work entirely on the segment $\mc{N}_\delta$\emph{---in particular, all integral norms over both $t$ and $x$ are assumed to be over $\mc{N}_\delta$.}
Moreover, we assume that \ass{BA}{\delta, \Delta} holds on $\mc{N}_\delta$.
The objective, then, is to show that \ass{BA}{\delta, \Delta^\prime} holds, with $\Delta^\prime \lesssim_{C, N} \Gamma + \Delta^2$.

Next, we present a brief outline of the process behind establishing (2) and (3).
The first task is to establish the regularity conditions \ass{F2}{} and \ass{K}, as defined in Sections \ref{sec.fol_reg} and \ref{sec.fol_curv}.
This step is essential, as it validates all the technical tools described throughout Sections \ref{sec.geom} and \ref{sec.fol}.
These regularity conditions are consequences of the bootstrap assumptions, and are a main reason why these bootstrap assumptions are fundamental to the overall argument.

The \ass{F0}{} condition follows immediately from the \ass{BA}{} assumption, namely, from the smallness of the second fundamental form $k = H$.
For the full \ass{F2}{} condition, we have to work a bit harder; in particular, we need control for the curl $\mf{C}$ of $H$.
Fortunately, $\mf{C}$ has special structure, given by the Codazzi equations in \eqref{eq.structure_renorm_gc}, which, along with \ass{BA}{}, grants it the required control.
\footnote{In particular, the structure equation for $\mf{C}$ contains neither $\nabla (\trace H)$ nor $M$ (both of which would require $L^2$-Besov control) on its right-hand side.}
Finally, for the \ass{K}{} condition, we combine the Gauss equation in \eqref{eq.structure_renorm_gc} along with the second equation in \eqref{eq.structure_renorm_ell}.
This, along with \ass{BA}{}, suffices to decompose the Gauss curvature $\mc{K}$ as in the \ass{K}{} condition.
In particular, the only term in the equation for $\mc{K}$ which is not $L^2$-controlled is the divergence of $Z$.

With the regularity conditions established, we next focus on obtaining the required improved estimates.
Many of the quantities, e.g., $\trace H$, $\ul{H}$, $M$, can be controlled by exploiting the evolution equations that they satisfy; see \eqref{eq.structure_renorm_ev} and \eqref{eq.structure_renorm_evd}.
For $H$ and $Z$, however, one also requires the elliptic estimates of \eqref{eq.structure_renorm_ell}.
The main idea is essentially the same as in \cite{kl_rod:cg, wang:cg}: the structure equations \eqref{eq.structure_renorm_ev}, \eqref{eq.structure_renorm_ell}, \eqref{eq.structure_renorm_evd} for the connection coefficients generally have the schematic form
\begin{align}
\label{eq.structure_generic_form} \mc{A} = \mc{R} + \mc{A}^2 \text{,}
\end{align}
where $\mc{A}$ and $\mc{R}$ represent any of the connection and curvature coefficients, respectively.
\footnote{There are of course exceptions to this, in which the right-hand side of the equation contains terms linear in $\mc{A}$. These must be addressed throughout the proof by being more careful with the order in which the various structure equations are used.}
Assuming that the appropriate norms can be used, then \eqref{eq.structure_generic_form}, along with \eqref{eq.renorm_ass_flux} and the bootstrap assumption \ass{BA}{\delta, \Delta}, yields the improved estimates
\begin{align*}
\| \mc{A} \| \lesssim \Gamma + \Delta^2 \text{.}
\end{align*}

Next, to control $\nabla H$ and $\nabla Z$ in the sum norm (see the \ass{BA}{} condition), we resort again to the elliptic equations \eqref{eq.structure_renorm_ell}.
The main idea is similar to that in \cite{kl_rod:cg, wang:cg}, except that we no longer require an infinite decomposition.
Roughly, looking at the right-hand sides in \eqref{eq.structure_renorm_ell}, the curvature quantities $B$ and $R$ must be controlled in the $N^{0 \star}$-norm using \eqref{eq.structure_renorm_nb}, while $M$ and $\nabla (\trace H)$ must be controlled in Besov norms using \eqref{eq.structure_renorm_evd}.
Our use of a sum norm in the \ass{BA}{} assumption allows us to take only one iteration of this decomposition and then simply reapply the \ass{BA}{} assumption.
\footnote{In contrast, in \cite{kl_rod:cg, shao:bdc_nv, wang:cg}, one must iterate this decomposition infinitely many times.}
This significantly simplifies the decomposition process.

In the above, we have obtained improved estimates for
\[ \| H \|_{ N^1_{t, x} } \text{,} \qquad \| \nabla H \|_{ N^{0\star}_{t, x} + B^{2, 0}_{t, x} } \text{,} \qquad \| Z \|_{ N^1_{t, x} } \text{,} \qquad \| \nabla Z \|_{ N^{0\star}_{t, x} + B^{2, 0}_{t, x} } \text{.} \]
The remaining improved estimates for $H$ and $Z$ in the sharp trace norm can now be obtained from Corollary \ref{thm.sharp_trace_ex}.
This completes the proof of steps (2) and (3), and hence concludes the bootstrap argument.
Moreover, all the desired estimates in \eqref{eq.renorm_est} will have been established within this bootstrap argument.

In the remainder of this section, we accomplish steps (2) and (3) in full detail.

\subsection{Regularity Conditions} \label{sec.proof_reg}

As mentioned before, the first goal is the \ass{F2}{} condition (see Section \ref{sec.fol_reg}).
We begin by establishing some trivial properties.

\begin{lemma} \label{thmp.rf0}
Suppose \ass{BA}{\delta, \Delta} holds.
Then:
\begin{itemize}
\item $(\mc{N}_\delta, \gamma)$ satisfies \ass{F0}{C, N, \Delta}.

\item The following estimates hold:
\begin{equation} \label{eqp.hz_sob} \| H \|_{ L^{4, \infty}_{x, t} } + \| Z \|_{ L^{4, \infty}_{x, t} } \lesssim_{C, N} \Delta \text{.} \end{equation}
\end{itemize}
\end{lemma}

\begin{proof}
Since $k = H$, then \ass{BA}{\delta, \Delta} combined with the fact that $( \Sph^2, \gamma [0] )$ satisfies \ass{r0}{C, N} implies the \ass{F0}{C, N, \Delta} condition on $(\mc{N}_\delta, \gamma)$.
Using this \ass{F0}{} condition, along with \eqref{eq.nsob_ineq} and \ass{BA}{\delta, \Delta}, we obtain the bounds
\[ \| H \|_{ L^{4, \infty}_{x, t} } \lesssim \| H [0] \|_{ L^4_x } + \Delta \text{,} \qquad \| Z \|_{ L^{4, \infty}_{x, t} } \lesssim \| Z [0] \|_{ L^4_x } + \Delta \text{.} \]
This implies \eqref{eqp.hz_sob}, since by \eqref{eq.sob_frac_sh},
\[ \| H [0] \|_{ L^4_x } \lesssim \| H [0] \|_{ H^{1/2}_x } \lesssim \Gamma \text{,} \qquad \| Z [0] \|_{ L^4_x } \lesssim \| Z [0] \|_{ H^{1/2}_x } \lesssim \Gamma \text{.} \qedhere \]
\end{proof}

The next goal is the full \ass{F2}{} condition.
For this, though, we require first some commutator estimates between $\nabla$ and the integral $\cint_0^t$.

\begin{lemma} \label{thmp.rf2}
Suppose \ass{BA}{\delta, \Delta} holds.
\begin{itemize}
\item There exists $\mf{D} \in \mc{C}^\infty \ul{T}^0_3 \mc{N}$ such that
\begin{equation} \label{eqp.sff_curl} \nabla_t \mf{D} = \mf{C} \text{,} \qquad \| \mf{D} \|_{ N^1_{t, x} \cap L^{\infty, 2}_{x, t} \cap L^{4, \infty}_{x, t} } \lesssim_{C, N} \Delta \text{.} \end{equation}

\item If $\Psi \in \mc{C}^\infty \ul{T}^r_l \mc{N}$, then the following commutator estimates hold:
\begin{align}
\label{eqp.comm_est} \| \nabla \cint_0^t \Psi \|_{ L^{2, 2}_{t, x} } &\lesssim_{C, N, r, l} \| \cint_0^t \nabla \Psi \|_{ L^{2, 2}_{t, x} } + \Delta \| \Psi \|_{ L^{2, 2}_{t, x} } \text{,} \\
\notag \| \cint_0^t \nabla \Psi \|_{ L^{2, 2}_{t, x} } &\lesssim_{C, N, r, l} \| \nabla \cint_0^t \Psi \|_{ L^{2, 2}_{t, x} } + \Delta \| \Psi \|_{ L^{2, 2}_{t, x} } \text{.}
\end{align}

\item $(\mc{N}_\delta, \gamma)$ satisfies \ass{F2}{C, N, \Delta}.
\end{itemize}
\end{lemma}

\begin{proof}
Recalling that $k = H$, the Codazzi equations in \eqref{eq.structure_renorm_gc} imply
\begin{align*}
\mf{C}_{abc} &= - \epsilon_{bc} \epsilon_a{}^d [ (1 - t) B_d - Z_d ] + (1 - t) ( H_{ab} Z_c - H_{ac} Z_b ) \text{.}
\end{align*}
Combining this with the evolution equation for $Z$ in \eqref{eq.structure_renorm_ev} yields
\begin{align*}
\mf{C}_{abc} &= \nabla_t [ (1 - t) \epsilon_{bc} \epsilon_a{}^d Z_d ] + (1 - t) ( H_{ab} Z_c - H_{ac} Z_b + 2 \epsilon_{bc} \epsilon_a{}^d \gamma^{ef} H_{de} Z_f ) \\
&\qquad + 2 \epsilon_{bc} \epsilon_a{}^d Z_d \\
&= \nabla_t I_1 + I_2 \text{,}
\end{align*}
where the terms $I_1$ and $I_2$ are defined from the preceding formula.
If we define $\mf{D} = I_1 + \cint_0^t I_2$, then $\nabla_t \mf{D} = \mf{C}$, as desired.
Moreover, by \ass{BA}{\delta, \Delta},
\begin{equation} \label{eqpl.sff_curl_1} \| \nabla_t \mf{D} \|_{ L^{2, 2}_{t, x} } = \| \mf{C} \|_{ L^{2, 2}_{t, x} } \lesssim \| \nabla H \|_{ L^{2, 2}_{t, x} } \lesssim \Delta \text{.} \end{equation}

Similarly, for the $L^{\infty, 2}_{x, t}$-bound for $\mf{D}$, we first control
\[ \| I_1 \|_{ L^{\infty, 2}_{x, t} } \lesssim \| Z \|_{ L^{\infty, 2}_{x, t} } \lesssim \Delta \text{.} \]
Next, each term in $I_2$, except for the last term ($\sim Z$), can be written as a product $J_1 \cdot J_2$, where by \eqref{eq.int_ineq} and \ass{BA}{\delta, \Delta}, the factors can be bounded
\[ \| J_1 \|_{ L^{\infty, 2}_{x, t} } + \| J_2 \|_{ L^{\infty, 2}_{x, t} } \lesssim \Delta \text{,} \qquad \| \cint_0^t ( J_1 \cdot J_2 ) \|_{ L^{\infty, 2}_{x, t} } \lesssim \| J_1 \|_{ L^{\infty, 2}_{x, t} } \| J_2 \|_{ L^{\infty, 2}_{x, t} } \lesssim \Delta^2 \text{.} \]
It follows that
\[ \| \cint_0^t I_2 \|_{ L^{\infty, 2}_{x, t} } \lesssim \| \cint_0^t Z \|_{ L^{\infty, 2}_{x, t} } + \Delta^2 \lesssim \Delta \text{.} \]
As a result, recalling also the \ass{F0}{C, N, \Delta} condition, we obtain
\begin{equation} \label{eqpl.sff_curl_2} \| \mf{D} \|_{ L^{2, 2}_{t, x} } \lesssim \| \mf{D} \|_{ L^{\infty, 2}_{x, t} } \lesssim \| I_1 \|_{ L^{\infty, 2}_{x, t} } + \| \cint^t_0 I_2 \|_{ L^{\infty, 2}_{x, t} } \lesssim \Delta \text{.} \end{equation}

To control $\nabla \mf{D}$, though, we must first prove the commutator estimate \eqref{eqp.comm_est}.
Recalling the commutation identity \eqref{eq.comm} along with \eqref{eq.int_ineq}, we obtain
\begin{align*}
\| \nabla \cint_0^t \Psi - \cint_0^t \nabla \Psi \|_{ L^{2, 2}_{t, x} } &\lesssim \| \cint_0^t ( k \otimes \nabla \cint_0^t \Psi ) \|_{ L^{2, 2}_{t, x} } + \| \cint_0^t ( \mf{C} \otimes \cint_0^t \Psi ) \|_{ L^{2, 2}_{t, x} } \\
&\lesssim \| H \|_{ L^{\infty, 2}_{x, t} } \| \nabla \cint_0^t \Psi \|_{ L^{2, 2}_{t, x} } + \| \cint_0^t ( \nabla_t \mf{D} \otimes \cint_0^t \Psi ) \|_{ L^{2, 2}_{t, x} } \text{.}
\end{align*}
Applying \ass{F0}{C, N, \Delta}, integrating by parts, and recalling \eqref{eq.int_ineq}, then
\begin{align*}
\| \nabla \cint_0^t \Psi - \cint_0^t \nabla \Psi \|_{ L^{2, 2}_{t, x} } &\lesssim \Delta \| \nabla \cint_0^t \Psi \|_{ L^{2, 2}_{t, x} } + \| \mf{D} \otimes \cint_0^t \Psi \|_{ L^{2, 2}_{t, x} } + \| \cint_0^t ( \mf{D} \otimes \Psi ) \|_{ L^{2, 2}_{t, x} } \\
&\lesssim \Delta \| \nabla \cint_0^t \Psi \|_{ L^{2, 2}_{t, x} } + \| \mf{D} \|_{ L^{\infty, 2}_{x, t} } \| \Psi \|_{ L^{2, 2}_{t, x} } \\
&\lesssim \Delta \| \nabla \cint_0^t \Psi \|_{ L^{2, 2}_{t, x} } + \Delta \| \Psi \|_{ L^{2, 2}_{t, x} } \text{.}
\end{align*}
The estimates in \eqref{eqp.comm_est} follow immediately from this and the assumption $\Delta \ll 1$.

We can now control $\nabla \mf{D}$.
First of all, from \ass{BA}{\delta, \Delta},
\[ \| \nabla I_1 \|_{ L^{2, 2}_{t, x} } \lesssim \| \nabla Z \|_{ L^{2, 2}_{t, x} } \lesssim \Delta \text{.} \]
For $\nabla I_2$, note that $J_1$ and $J_2$, defined as before, satisfy
\[ \| J_1 \|_{ L^{\infty, 2}_{x, t} \cap L^{4, \infty}_{x, t} } + \| \nabla J_1 \|_{ L^{2, 2}_{t, x} } + \| J_2 \|_{ L^{\infty, 2}_{x, t} \cap L^{4, \infty}_{x, t} } + \| \nabla J_2 \|_{ L^{2, 2}_{t, x} } \lesssim \Delta \text{,} \]
where we used \ass{BA}{\delta, \Delta} and \eqref{eqp.hz_sob}.
Thus, applying \eqref{eqp.comm_est} and \eqref{eq.int_ineq}, we obtain
\begin{align*}
\| \nabla \cint_0^t ( J_1 \cdot J_2 ) \|_{ L^{2, 2}_{t, x} } &\lesssim \| \cint_0^t \nabla ( J_1 \cdot J_2 ) \|_{ L^{2, 2}_{t, x} } + \| J_1 \cdot J_2 \|_{ L^{2, 2}_{t, x} } \\
&\lesssim \| \nabla J_1 \|_{ L^{2, 2}_{t, x} } \| J_2 \|_{ L^{\infty, 2}_{x, t} } + \| J_1 \|_{ L^{\infty, 2}_{x, t} } \| \nabla J_2 \|_{ L^{2, 2}_{t, x} } + \| J_1 \|_{ L^{\infty, 2}_{x, t} } \| J_2 \|_{ L^{2, \infty}_{x, t} } \\
&\lesssim \Delta^2 \text{.}
\end{align*}
Therefore, applying \eqref{eqp.comm_est} again,
\[ \| \nabla \cint^t_0 I_2 \|_{ L^{2, 2}_{t, x} } \lesssim \| \nabla Z \|_{ L^{2, 2}_{t, x} } + \| Z \|_{ L^{2, 2}_{t, x} } + \Delta^2 \lesssim \Delta \text{.} \]
Combining the above, we obtain
\begin{equation} \label{eqpl.sff_curl_3} \| \nabla \mf{D} \|_{ L^{2, 2}_{t, x} } \lesssim \Delta \text{.} \end{equation}

From the above estimates \eqref{eqpl.sff_curl_1}-\eqref{eqpl.sff_curl_3}, along with \eqref{eq.nsob_ineq}, we obtain \eqref{eqp.sff_curl}.
The \ass{F2}{C, N, \Delta} condition now follows from \eqref{eqp.sff_curl} and \ass{BA}{\delta, \Delta}.
\end{proof}

\subsection{The Curvature Condition}

For the necessary elliptic estimates, we will also need the \ass{K}{} condition.
The first step is a number of transport equation estimates.

\begin{lemma} \label{thmp.transport}
If \ass{BA}{\delta, \Delta} holds, then the following estimates hold:
\begin{align}
\label{eqp.transport} \| \trace H \|_{ L^{\infty, \infty}_{t, x} } &\lesssim_{C, N} \Gamma + \Delta^2 \text{,} \\
\notag \| H \|_{ H^{\infty, 1/2}_{t, x} } + \| Z \|_{ H^{\infty, 1/2}_{t, x} } &\lesssim_{C, N} \Delta \text{,} \\
\notag \| \nabla_t \ul{H} \|_{ L^{2, 2}_{t, x} } + \| \ul{H} \|_{ L^{2, \infty}_{x, t} } &\lesssim_{C, N} \Delta \text{,} \\
\notag \| \nabla_t \nabla ( \trace H ) \|_{ L^{2, 1}_{x, t} } + \| \nabla ( \trace H ) \|_{ L^{2, \infty}_{x, t} } &\lesssim_{C, N} \Gamma + \Delta^2 \text{,} \\
\notag \| \nabla_t M \|_{ L^{2, 1}_{x, t} } + \| M \|_{ L^{2, \infty}_{x, t} } &\lesssim_{C, N} \Gamma + \Delta^2 \text{.}
\end{align}
\end{lemma}

\begin{proof}
First, taking the trace of the first equation in \eqref{eq.structure_renorm_ev} yields
\[ \nabla_t ( \trace H ) = - \gamma^{ab} \gamma^{cd} H_{ac} H_{bd} \text{.} \]
Applying $\cint_0^t$ to the above and recalling \eqref{eq.norm_comp_int} and \ass{BA}{\delta, \Delta}, we obtain
\[ \| \trace H \|_{ L^{\infty, \infty}_{t, x} } \lesssim \| \trace H [0] \|_{ L^\infty_x } + \| H \|_{ L^{\infty, 2}_{x, t} }^2 \lesssim \Gamma + \Delta^2 \text{.} \]
This proves the first estimate in \eqref{eqp.transport}.
Moreover, from the \ass{F2}{C, N, \Delta} condition (see Lemma \ref{thmp.rf2}) and \eqref{eq.nsob_trace}, we obtain the second estimate in \eqref{eqp.transport}:
\[ \| H \|_{ H^{\infty, 1/2}_{t, x} } + \| Z \|_{ H^{\infty, 1/2}_{t, x} } \lesssim \| H [0] \|_{ H^{1/2}_x } + \| Z [0] \|_{ H^{1/2}_x } + \Delta \lesssim \Delta \text{.} \]

Applying \eqref{eq.renorm_ass_flux}, \ass{BA}{\delta, \Delta}, and \eqref{eqp.hz_sob} to the last equation in \eqref{eq.structure_renorm_ev} yields
\begin{align*}
\| \nabla_t \ul{H} \|_{ L^{2, 2}_{t, x} } &\lesssim \| R \|_{ L^{2, 2}_{t, x} } + \| \nabla Z \|_{ L^{2, 2}_{t, x} } + \| H \|_{ L^{\infty, 2}_{x, t} } ( 1 + \| \ul{H} \|_{ L^{2, \infty}_{x, t} } ) + \| Z \|_{ L^{\infty, 2}_{x, t} } \| Z \|_{ L^{2, \infty}_{x, t} } \\
&\lesssim \Delta ( 1 + \| \ul{H} \|_{ L^{2, \infty}_{x, t} } ) \text{.}
\end{align*}
Moreover, recalling \eqref{eq.norm_comp_int} and the fact that $\delta \leq 1$, we have
\[ \| \ul{H} \|_{ L^{2, \infty}_{x, t} } \lesssim \| \ul{H} [0] \|_{ L^2_x } + \| \nabla_t \ul{H} \|_{ L^{2, 2}_{t, x} } \lesssim \Gamma + \| \nabla_t \ul{H} \|_{ L^{2, 2}_{t, x} } \text{.} \]
Combining the above and recalling that $\Delta \ll 1$ yields the third estimate in \eqref{eqp.transport}.

Next, we take an $L^{2, 1}_{x, t}$-norm of the first equation in \eqref{eq.structure_renorm_evd}:
\begin{align*}
\| \nabla_t \nabla ( \trace H ) \|_{ L^{2, 1}_{x, t} } &\lesssim \| H \otimes \nabla H \|_{ L^{2, 1}_{x, t} } \lesssim \| H \|_{ L^{\infty, 2}_{x, t} } \| \nabla H \|_{ L^{2, 2}_{t, x} } \lesssim \Delta^2 \text{.}
\end{align*}
In particular, we applied \ass{BA}{\delta, \Delta}.
Keeping in mind \eqref{eq.norm_comp_int}, then
\begin{align*}
\| \nabla ( \trace H ) \|_{ L^{2, \infty}_{x, t} } &\lesssim \| \nabla ( \trace H ) [0] \|_{ L^2_x } + \| \cint_0^t \nabla_t \nabla ( \trace H ) \|_{ L^{2, \infty}_{x, t} } \\
&\lesssim \| \nabla ( \trace H ) [0] \|_{ L^2_x } + \| \nabla_t \nabla ( \trace H ) \|_{ L^{2, 1}_{x, t} } \\
&\lesssim \Gamma + \Delta^2 \text{.}
\end{align*}
This proves the fourth part of \eqref{eqp.transport}.

The remaining estimates for $M$ are analogous.
From the second equation in \eqref{eq.structure_renorm_evd}, we see that $\nabla_t M$ can be decomposed as $I_0 + I_1 + I_2$, where:
\begin{itemize}
\item $I_0 = -3 m s_0^{-1} (\trace H)$ satisfies, due to \eqref{eqp.transport},
\[ \| I_0 \|_{ L^{\infty, \infty}_{t, x} } \lesssim \| \trace H \|_{ L^{\infty, \infty}_{t, x} } \lesssim \Gamma + \Delta^2 \text{.} \]

\item $I_1$ can be expressed as sums of terms of the form $J_1 \cdot J_2$, where
\[ \| J_1 \|_{ L^{\infty, 2}_{x, t} } \lesssim \Delta \text{,} \qquad \| J_2 \|_{ L^{2, 2}_{t, x} } \lesssim \Delta \text{,} \]

\item $I_2$ can be written as sums of terms of the form $K_1 \cdot K_2 \cdot K_3$, with
\[ \| K_1 \|_{ L^{\infty, 2}_{x, t} } \lesssim \Delta \text{,} \qquad \| K_2 \|_{ L^{\infty, 2}_{x, t} } \lesssim \Delta \text{,} \qquad \| K_3 \|_{ L^{2, \infty}_{x, t} } \lesssim \Delta \text{.} \]
\end{itemize}
As a result, recalling \eqref{eq.norm_comp_int}, \eqref{eq.renorm_ass_init}, and \ass{BA}{\Delta, \delta}, we obtain
\begin{align*}
\| \nabla_t M \|_{ L^{2, 1}_{x, t} } &\lesssim \| I_0 \|_{ L^{\infty, \infty}_{t, x} } + \| J_1 \|_{ L^{\infty, 2}_{x, t} } \| J_2 \|_{ L^{2, 2}_{t, x} } + \| K_1 \|_{ L^{\infty, 2}_{x, t} } \| K_2 \|_{ L^{\infty, 2}_{x, t} } \| K_3 \|_{ L^{2, \infty}_{x, t} } \\
&\lesssim \Gamma + \Delta^2 \text{,} \\
\| M \|_{ L^{2, \infty}_{x, t} } &\lesssim \| M [0] \|_{ L^2_x } + \| \nabla_t M \|_{ L^{2, 1}_{x, t} } \\
&\lesssim \Gamma + \Delta^2 \text{.}
\end{align*}
This proves the final estimate in \eqref{eqp.transport}.
\end{proof}

With these transport equation estimates, we can now prove the \ass{K}{} condition.

\begin{lemma} \label{thmp.gc}
Suppose \ass{BA}{\delta, \Delta} holds.
Then:
\begin{itemize}
\item $(\mc{N}_\delta, \gamma)$ satisfies \ass{K}{1, D} for some $D \lesssim \Delta$, with $D \ll 1$.

\item The following estimates hold for $\mc{K}$:
\begin{equation} \label{eqp.gc_weak} \| \mc{K} - 1 \|_{ L^{2, 2}_{t, x} } \lesssim \Delta \text{,} \qquad \| \mc{K} - 1 \|_{ H^{\infty, -1/2}_{t, x} } \lesssim \Delta \text{.} \end{equation}
\end{itemize}
\end{lemma}

\begin{proof}
Combining the equations for $\mc{K}$ and $\mc{D}_1 Z$ in Proposition \ref{thm.structure_renorm}, we see that
\[ \mc{K} - 1 = - \frac{1}{2} \trace \ul{H} + (1 - t) \left[ \gamma^{ab} \nabla_a Z_b - M + \frac{1}{2} \paren{ 1 - \frac{2m}{s} } \trace H - \frac{1}{4} \trace H \trace \ul{H} \right] \text{.} \]
In terms of the definition of the \ass{K}{} condition from Section \ref{sec.fol_curv}, we take $f \equiv 1$, $V = (1 - t) Z$, and $W$ the remaining terms of the above equation.
By \eqref{eqp.transport},
\begin{align*}
\| V \|_{ H^{\infty, 1/2}_{t, x} } &\lesssim \| Z \|_{ H^{\infty, 1/2}_{t, x} } \lesssim \Delta \text{,} \\
\| W \|_{ L^{2, \infty}_{x, t} } &\lesssim \| \trace H \|_{ L^{2, \infty}_{x, t} } + \| M \|_{ L^{2, \infty}_{x, t} } + ( 1 + \| \trace H \|_{ L^{\infty, \infty}_{t, x} } ) \| \trace \ul{H} \|_{ L^{2, \infty}_{x, t} } \lesssim \Delta \text{.}
\end{align*}
This yields \ass{K}{1, D}, with $D \lesssim \Delta$.
For \eqref{eqp.gc_weak}, we use the second equation in \eqref{eq.structure_renorm_gc} (along with \eqref{eq.renorm_ass_flux}, \ass{BA}{\delta, \Delta}, and \eqref{eqp.transport}) and then \eqref{eq.curv_sob}.
\end{proof}

\subsection{Sobolev Estimates}

The next step is to obtain $N^1$-estimates for $H$ and $Z$.
This is the first part of the improved versions of the bootstrap assumptions.

\begin{lemma} \label{thmp.hz_sob_imp}
If \ass{BA}{\delta, \Delta} holds, then
\begin{align}
\label{eqp.hz_sob_imp} \| H \|_{ N^1_{t, x} } + \| Z \|_{ N^1_{t, x} } &\lesssim_{C, N} \Gamma + \Delta^2 \text{.}
\end{align}
\end{lemma}

\begin{proof}
From the first equation in \eqref{eq.structure_renorm_ev}, \eqref{eq.renorm_ass_flux}, \ass{BA}{\delta, \Delta}, and \eqref{eqp.hz_sob}, we have
\begin{align*}
\| \nabla_t H \|_{ L^{2, 2}_{t, x} } &\lesssim \| H \|_{ L^{\infty, 2}_{x, t} } \| H \|_{ L^{2, \infty}_{x, t} } + \Gamma \lesssim \Gamma + \Delta^2 \text{.}
\end{align*}
A similar process using the second equation in \eqref{eq.structure_renorm_ev} yields
\[ \| \nabla_t Z \|_{ L^{2, 2}_{t, x} } \lesssim \Gamma + \Delta^2 \text{.} \]

Next, applying \ass{K}{1, D} and \eqref{eq.hodge_est_D1} to the second equation in \eqref{eq.structure_renorm_ell}, we obtain
\[ \| \nabla Z \|_{ L^{2, 2}_{t, x} } + \| Z \|_{ L^{2, 2}_{t, x} } \lesssim \| R \|_{ L^{2, 2}_{t, x} } + \| M \|_{ L^{2, 2}_{t, x} } + \| H \|_{ L^{\infty, 2}_{x, t} } \| \ul{H} \|_{ L^{2, \infty}_{x, t} } \text{.} \]
By \ass{BA}{\delta, \Delta}, \eqref{eq.renorm_ass_flux}, and \eqref{eqp.transport}, we can bound
\begin{equation} \label{eqpl.hz_sob_imp_1} \| \nabla Z \|_{ L^{2, 2}_{t, x} } + \| Z \|_{ L^{2, 2}_{t, x} } \lesssim \Gamma + \Delta^2 \text{.} \end{equation}
Similarly, using \eqref{eq.hodge_est_D2} on the first equation in \eqref{eq.structure_renorm_ell} yields
\begin{align*}
\| \nabla H \|_{ L^{2, 2}_{t, x} } + \| H \|_{ L^{2, 2}_{t, x} } &\lesssim \| \nabla \hat{H} \|_{ L^{2, 2}_{t, x} } + \| \hat{H} \|_{ L^{2, 2}_{t, x} } + \| \nabla ( \trace H ) \|_{ L^{2, \infty}_{x, t} } + \| \trace H \|_{ L^{\infty, \infty}_{t, x} } \\
&\lesssim \Gamma + \Delta^2 + \| B \|_{ L^{2, 2}_{t, x} } + \| Z \|_{ L^{2, 2}_{t, x} } + \| H \|_{ L^{\infty, 2}_{x, t} } \| Z \|_{ L^{2, \infty}_{x, t} } \text{,}
\end{align*}
where we also applied \eqref{eqp.transport}.
Applying \ass{BA}{\delta, \Delta}, \eqref{eq.renorm_ass_flux}, \eqref{eqp.hz_sob}, and \eqref{eqpl.hz_sob_imp_1} yields
\[ \| \nabla H \|_{ L^{2, 2}_{t, x} } + \| H \|_{ L^{2, 2}_{t, x} } \lesssim \Gamma + \Delta^2 \text{.} \]
By the definition of the $N^1_{t, x}$-norm, the proof is complete.
\end{proof}

Next, we obtain some additional estimates for the curvature coefficients.
Recall the special covariant integral operator $\cint^t_\star$, which was defined in \eqref{eq.cint_ex}.

\begin{lemma} \label{thmp.curv_int}
If \ass{BA}{\delta, \Delta} holds, then
\begin{align}
\label{eqp.curv_int} \| \cint^t_\star A \|_{ L^{\infty, 2}_{t, x} } + \| \cint^t_\star B \|_{ L^{\infty, 2}_{t, x} } &\lesssim_{C, N} \Delta \text{,} \\
\notag \| \cint^t_\star B \|_{ N^{1i}_{t, x} } + \| \cint^t_\star R \|_{ N^{1i}_{t, x} } &\lesssim_{C, N} \Gamma + \Delta^2 \text{,} \\
\notag \| B \|_{ N^{0 \star}_{t, x} } + \| R \|_{ N^{0 \star}_{t, x} } &\lesssim_{C, N} \Gamma + \Delta^2 \text{.}
\end{align}
\end{lemma}

\begin{proof}
It suffices to prove the first two estimates in \eqref{eqp.curv_int} with $\cint^t_\star \Psi$ replaced by $\cint^t_0 ( \eta_+ \Psi )$, since the remaining estimate for $\cint^t_\delta ( \eta_- \Psi )$ can be obtained analogously.

First, using the first equation in \eqref{eq.structure_renorm_ev}, we can write
\begin{align*}
\| \cint^t_0 ( \eta_+ A ) \|_{ L^{\infty, 2}_{x, t} } &\lesssim \| \cint^t_0 ( \eta_+ \nabla_t H ) \|_{ L^{\infty, 2}_{x, t} } + \| \cint^t_0 ( \eta_+ | H | | H | ) \|_{ L^{\infty, 2}_{x, t} } = X_0 + X_1 \text{.}
\end{align*}
Using H\"older's inequality, we have
\[ X_1 \lesssim \| H \|_{ L^{\infty, 2}_{x, t} }^2 \lesssim \Delta^2 \text{.} \]
Moreover, integrating by parts and recalling that $\eta_+$ vanishes at $t = 0$ yields
\begin{align*}
X_0 &\lesssim \| \eta_+ H \|_{ L^{\infty, 2}_{x, t} } + \| \cint^t_0 ( \eta_+^\prime H ) \|_{ L^{\infty, 2}_{x, t} } \lesssim \| H \|_{ L^{\infty, 2}_{x, t} } + \delta^{-1} \| \cint^t_0 | H | \|_{ L^{\infty, 2}_{x, t} } \lesssim \Delta \text{.}
\end{align*}
Combining the above steps with a completely analogous computation using the second equation in \eqref{eq.structure_renorm_ev}, we obtain the estimate
\[ \| \cint^t_0 ( \eta_+ A ) \|_{ L^{\infty, 2}_{x, t} } + \| \cint^t_0 ( \eta_+ B ) \|_{ L^{\infty, 2}_{x, t} } \lesssim \Delta \text{.} \]
From the preceding remarks, the first estimate of \eqref{eqp.curv_int} follows.

We now consider the second part of \eqref{eqp.curv_int}.
By \eqref{eq.renorm_ass_flux}, we immediately obtain
\[ \| \nabla_t \cint^t_0 ( \eta_+ B ) \|_{ L^{2, 2}_{t, x} } + \| \cint^t_0 ( \eta_+ B ) \|_{ L^{2, 2}_{t, x} } \lesssim \| B \|_{ L^{2, 2}_{t, x} } \lesssim \Gamma \text{.} \]
Next, applying \eqref{eq.hodge_est_D1}, \eqref{eq.renorm_ass_flux}, and \eqref{eqp.comm_est}, we obtain
\footnote{Since $\nabla_t$ annihilates both $\gamma$ and $\epsilon$, then $\cint^t_0$ commutes with any of the symmetric Hodge operators in the same way as it does with $\nabla$.  As a result, \eqref{eqp.comm_est} applies here.}
\[ \| \nabla \cint^t_0 ( \eta_+ B ) \|_{ L^{2, 2}_{t, x} } \lesssim \| \mc{D}_1 \cint^t_0 ( \eta_+ B ) \|_{ L^{2, 2}_{t, x} } \lesssim \| \cint^t_0 ( \eta_+ \mc{D}_1 B ) \|_{ L^{2, 2}_{t, x} } + \Delta^2 \text{.} \]
From the second equation in \eqref{eq.structure_renorm_nb},
\begin{align*}
\| \cint^t_0 ( \eta_+ \mc{D}_1 B ) \|_{ L^{2, 2}_{t, x} } &\lesssim \| \cint^t_0 ( \eta_+ \nabla_t R ) \|_{ L^{2, 2}_{t, x} } + \| \cint^t_0 ( \eta_+ | H | ) \|_{ L^{2, 2}_{t, x} } + \| \cint^t_0 ( \eta_+ | H | | R | ) \|_{ L^{2, 2}_{t, x} } \\
&\qquad + \| \cint^t_0 ( \eta_+ | Z | | B | ) \|_{ L^{2, 2}_{t, x} } + \| \cint^t_0 ( \eta_+ \ul{H} \otimes A ) \|_{ L^{2, 2}_{t, x} } \\
&= I_0 + I_1 + I_2 + I_3 + I_4 \text{.}
\end{align*}
By \eqref{eqp.hz_sob_imp}, we obtain $I_1 \lesssim \Gamma + \Delta^2$.
Also, $I_2$ and $I_3$ can be controlled using \eqref{eq.int_ineq}:
\[ I_2 \lesssim \| H \|_{ L^{\infty, 2}_{x, t} } \| R \|_{ L^{2, 2}_{t, x} } \lesssim \Delta^2 \text{,} \qquad I_3 \lesssim \| Z \|_{ L^{\infty, 2}_{x, t} } \| B \|_{ L^{2, 2}_{t, x} } \lesssim \Delta^2 \text{.} \]
For $I_4$, we first integrate by parts:
\begin{align*}
I_4 &= \| \cint^t_0 ( \eta_+ \ul{H} \otimes \nabla_t \cint^t_\star A ) \|_{ L^{2, 2}_{t, x} } \\
&\lesssim \| \ul{H} \otimes \cint^t_\star A \|_{ L^{2, 2}_{t, x} } + \| \delta^{-1} \cint_0^t ( | \ul{H} | | \cint^t_\star A | ) \|_{ L^{2, 2}_{t, x} } + \| \cint^t_0 ( | \nabla_t \ul{H} | | \cint^t_\star A | ) \|_{ L^{2, 2}_{t, x} } \text{.}
\end{align*}
Applying \eqref{eq.int_ineq}, H\"older's inequality, \eqref{eqp.transport}, and the first part of \eqref{eqp.curv_int}, we obtain
\[ I_4 \lesssim \| \ul{H} \|_{ L^{2, \infty}_{x, t} } \| \cint^t_\star A \|_{ L^{\infty, 2}_{x, t} } + \| \nabla_t \ul{H} \|_{ L^{2, 2}_{t, x} } \| \cint^t_\star A \|_{ L^{\infty, 2}_{x, t} } \lesssim \Delta^2 \text{.} \]
Finally, for $I_0$, we integrate by parts and apply \eqref{eq.int_ineq} and \eqref{eq.renorm_ass_flux}:
\begin{align*}
I_0 &\lesssim \| R \|_{ L^{2, 2}_{t, x} } + \| \delta^{-1} \cint^t_0 R \|_{ L^{2, 2}_{t, x} } \lesssim \Gamma \text{.}
\end{align*}

It follows from the above that
\[ \| \cint^t_\star B \|_{ N^1_{t, x} } \lesssim \Gamma + \Delta^2 \text{.} \]
Finally, for the initial value bound, we apply \eqref{eq.nsob_trace} to obtain
\[ \| ( \cint^t_\star B ) [0] \|_{ H^{1/2}_x } \lesssim \| \cint^t_\delta ( \eta_- B ) [0] \|_{ H^{1/2}_x } \lesssim \| \cint^t_\delta ( \eta_- B ) [\delta] \|_{ H^{1/2}_x } + \| \cint^t_\delta ( \eta_- B ) \|_{ N^1_{t, x} } \text{.} \]
The first term on the right-hand side vanishes by definition, while the second term can be bounded as in the preceding argument.
As a result, we have
\[ \| ( \cint^t_\star B ) [0] \|_{ H^{1/2}_x } \lesssim \| \cint^t_1 ( \eta_- B ) \|_{ N^1_{t, x} } \lesssim \Gamma + \Delta^2 \text{.} \]

The above completes the Sobolev estimate for $\cint^t_\star B$ in \eqref{eqp.curv_int}.
The corresponding estimate for $\cint^t_\star R$ can be established using completely analogous methods.
This proves the second estimate in \eqref{eqp.curv_int}.
The remaining estimate in \eqref{eqp.curv_int} follows immediately from the second estimate, since $\cint^t_\star \Psi$ is a $\nabla_t$-antiderivative of $\Psi$.
\end{proof}

\subsection{Transport-Besov Estimates}

The next task is to obtain some basic Besov and decomposition estimates from evolution relations.

\begin{lemma} \label{thmp.main_decomp_ev}
If \ass{BA}{\delta, \Delta} holds, then
\begin{align}
\label{eqp.main_decomp_ev} \| \nabla ( \trace H ) \|_{ B^{\infty, 0}_{t, x} } + \| M \|_{ B^{\infty, 0}_{t, x} } &\lesssim_{C, N} \Gamma + \Delta^2 \text{,} \\
\notag \| H \otimes \ul{H} \|_{ B^{2, 0}_{t, x} } &\lesssim_{C, N} \Delta^2 \text{,} \\
\notag \| \nabla_t \ul{H} \|_{ N^{0 \star}_{t, x} + B^{2, 0}_{t, x} } + \| \ul{H} \|_{ B^{\infty, 0}_{t, x} } &\lesssim_{C, N} \Delta \text{.}
\end{align}
\end{lemma}

\begin{proof}
From the first equation in \eqref{eq.structure_renorm_evd}, along with \eqref{eq.norm_comp_besov} and \eqref{eq.est_trace_sh_ex}, we obtain
\begin{align*}
\| \nabla ( \trace H ) \|_{ B^{\infty, 0}_{t, x} } &\lesssim \| \nabla ( \trace H ) [0] \|_{ B^0_x } + \| \cint_0^t ( H \otimes \nabla H ) \|_{ B^{\infty, 0}_{t, x} } \\
&\lesssim \Gamma + \| H \|_{ N^{1 i}_{t, x} \cap L^{2, \infty}_{x, t} } \| \nabla H \|_{ N^{0 \star}_{t, x} + B^{2, 0}_{t, x} } \\
&\lesssim \Gamma + \Delta^2 \text{,}
\end{align*}
where we also applied \eqref{eq.renorm_ass_init} and \ass{BA}{\delta, \Delta}.
The corresponding estimate for $M$ is similar, but requires first the remaining estimates for $\ul{H}$.

Applying the fundamental theorem of calculus to $\ul{H}$, we can write
\[ \| H \otimes \ul{H} \|_{ B^{2, 0}_{t, x} } \lesssim \| H \otimes \mf{p} ( \ul{H} [0] ) \|_{ B^{2, 0}_{t, x} } + \| H \otimes \cint^t_0 \nabla_t \ul{H} \|_{ B^{2, 0}_{t, x} } = I_1 + I_2 \text{,} \]
where $\mf{p}$ denotes the parallel transport operator defined in Section \ref{sec.fol_ev}.
The term $I_1$ can be controlled using \eqref{eq.est_prod_imp}, while $I_2$ is bounded using \eqref{eq.est_trace_shp_ex}:
\begin{align}
\label{eqpl.main_decomp_ev_1} \| H \otimes \ul{H} \|_{ B^{2, 0}_{t, x} } &\lesssim \| H \|_{ N^{1i}_{t, x} \cap L^{\infty, 2}_{x, t} } ( \| \ul{H} [0] \|_{ B^0_x } + \| \nabla_t \ul{H} \|_{ N^{0 \star}_{t, x} + B^{2, 0}_{t, x} } ) \\
\notag &\lesssim \Delta^2 + \Delta \| \nabla_t \ul{H} \|_{ N^{0 \star}_{t, x} + B^{2, 0}_{t, x} } \text{.}
\end{align}
In the last step, we applied \ass{BA}{\delta, \Delta}.

Next, we use the last equation in \eqref{eq.structure_renorm_ev}:
\begin{align*}
\| \nabla_t \ul{H} \|_{ N^{0 \star}_{t, x} + B^{2, 0}_{t, x} } &\lesssim \| \nabla Z \|_{ N^{0 \star}_{t, x} + B^{2, 0}_{t, x} } + \| H \|_{ B^{2, 0}_{t, x} } + \| H \otimes \ul{H} \|_{ B^{2, 0}_{t, x} } \\
&\qquad + \| Z \otimes Z \|_{ B^{2, 0}_{t, x} } + \| R \|_{ N^{0 \star}_{t, x} } \\
&= I_1 + I_2 + I_3 + I_4 + I_5 \text{.}
\end{align*}
Using \eqref{eq.est_prod_ex}, \eqref{eq.renorm_ass_init}, and \eqref{eqp.hz_sob_imp}, we can bound
\[ I_2 \lesssim \Delta \text{,} \qquad I_4 \lesssim \Delta^2 \text{,} \]
while \eqref{eqp.curv_int} and \ass{BA}{\delta, \Delta} yield
\[ I_5 \lesssim \Gamma + \Delta^2 \text{,} \qquad I_1 \lesssim \Delta \text{.} \]
As a result, we obtain
\begin{equation} \label{eqpl.main_decomp_ev_2} \| \nabla_t \ul{H} \|_{ N^{0 \star}_{t, x} + B^{2, 0}_{t, x} } \lesssim \Delta + \| H \otimes \ul{H} \|_{ B^{2, 0}_{t, x} } \text{.} \end{equation}

Combining \eqref{eqpl.main_decomp_ev_1} and \eqref{eqpl.main_decomp_ev_2} yields
\[ \| H \otimes \ul{H} \|_{ B^{2, 0}_{t, x} } \lesssim \Delta^2 \text{,} \qquad \| \nabla_t \ul{H} \|_{ N^{0 \star}_{t, x} + B^{2, 0}_{t, x} } \lesssim \Delta \text{.} \]
Moreover, applying \eqref{eq.est_trace_ex} and \eqref{eq.renorm_ass_init}, we have
\begin{align*}
\| \ul{H} \|_{ B^{\infty, 0}_{t, x} } &\lesssim \| \ul{H} [0] \|_{ B^0_x } + \| \cint^t_0 \nabla_t \ul{H} \|_{ B^{\infty, 0}_{t, x} } \lesssim \Gamma + \| \nabla_t \ul{H} \|_{ N^{0 \star}_{t, x} + B^{2, 0}_{t, x} } \lesssim \Delta \text{.}
\end{align*}
This proves the second and third inequalities of \eqref{eqp.main_decomp_ev}.

It remains to prove the Besov bound for $M$.
As in the proof of \eqref{eqp.transport}, we decompose $\nabla_t M$ as $I_0 + I_1 + I_2$; again, by the second equation in \eqref{eq.structure_renorm_evd}:
\begin{itemize}
\item $I_0 = -3 m s_0^{-1} (\trace H)$ satisfies
\begin{align*}
\| I_0 \|_{ B^{\infty, 0}_{t, x} } &\lesssim \| \nabla ( \trace H ) \|_{ L^{\infty, 2}_{t, x} } + \| \trace H \|_{ L^{\infty, 2}_{t, x} } \lesssim \Gamma + \Delta^2 \text{,}
\end{align*}
where we applied \eqref{eqp.transport} along with \eqref{eq.sobolev}.

\item $I_1$ can be expressed as sums of terms of the form $J_1 \cdot J_2$, where
\footnote{In particular, for the term of the form $Z \cdot B$, we can take $J_2$ to be $B$ by applying \eqref{eqp.curv_int}.}
\[ \| J_1 \|_{ N^{1 i}_{t, x} \cap L^{\infty, 2}_{x, t} } \lesssim \Delta \text{,} \qquad \| J_2 \|_{ N^{0 \star}_{t, x} + B^{2, 0}_{t, x} } \lesssim \Delta \text{,} \]

\item $I_2$ can be written as sums of terms of the form $K_1 \cdot K_2 \cdot K_3$, with
\[ \| K_1 \|_{ N^{1 i}_{t, x} \cap L^{\infty, 2}_{x, t} } \lesssim \Delta \text{,} \qquad \| K_2 \cdot K_3 \|_{ N^{0 \star}_{t, x} + B^{2, 0}_{t, x} } \lesssim \Delta \text{.} \]
More specifically, in most cases, we can bound, using \eqref{eq.est_prod_ex} and \ass{BA}{\delta, \Delta},
\[ \| K_2 \cdot K_3 \|_{ B^{2, 0}_{t, x} } \lesssim \| K_2 \|_{ N^{1 i}_{t, x} } \| K_3 \|_{ N^{1 i}_{t, x} } \lesssim \Delta^2 \text{.} \]
The exception is when $K_2 \cdot K_3 \simeq H \cdot \ul{H}$, for which we apply the second part of \eqref{eqp.main_decomp_ev}, which was established in the preceding paragraph.
\end{itemize}
As a result, applying \eqref{eq.norm_comp_besov} and \eqref{eq.est_trace_sh_ex} yields
\begin{align*}
\| M \|_{ B^{\infty, 0}_{t, x} } &\lesssim \| M [0] \|_{ B^0_x } + \| I_0 \|_{ B^{\infty, 0}_{t, x} } + \| J_1 \|_{ N^{1 i}_{t, x} \cap L^{\infty, 2}_{x, t} } \| J_2 \|_{ N^{0 \star}_{t, x} + B^{2, 0}_{t, x} } \\
&\qquad + \| K_1 \|_{ N^{1 i}_{t, x} \cap L^{\infty, 2}_{x, t} } \| K_2 \cdot K_3 \|_{ N^{0 \star}_{t, x} + B^{2, 0}_{t, x} } \\
&\lesssim \Gamma + \Delta^2 \text{.}
\end{align*}
This completes the proof of the first inequality of \eqref{eqp.transport}.
\end{proof}

\subsection{Advanced Curvature Estimates}

Using Lemma \ref{thmp.main_decomp_ev}, we can now prove additional estimates for the renormalized curvature components.
We begin with the following preliminary trace-type estimates.

\begin{lemma} \label{thmp.curv_int_tr}
If \ass{BA}{\delta, \Delta} holds, then
\begin{equation} \label{eqp.curv_int_tr} \| (1 - t) \mc{D}_2^{-1} B \|_{ L^{\infty, 2}_{x, t} } + \| \mc{D}_1^{-1} R \|_{ L^{\infty, 2}_{x, t} } \lesssim_{C, N} \Delta \text{.} \end{equation}
\end{lemma}

\begin{proof}
We apply the first equation in \eqref{eq.structure_renorm_ell} to write
\begin{align*}
\| (1 - t) \mc{D}_2^{-1} B \|_{ L^{\infty, 2}_{x, t} } &\lesssim \| \hat{H} \|_{ L^{\infty, 2}_{x, t} } + \| \mc{D}_2^{-1} Z \|_{ L^{2, \infty}_{t, x} } + \| \mc{D}_2^{-1} \nabla ( \trace H ) \|_{ L^{2, \infty}_{t, x} } \\
&\qquad + \| \mc{D}_2^{-1} ( H \cdot Z ) \|_{ L^{2, \infty}_{t, x} } \text{.}
\end{align*}
The first term on the right-hand side is bounded by $\Delta$ using \ass{BA}{\delta, \Delta}.
The remaining terms can be controlled using \eqref{eq.besov_impr_sh} and \eqref{eq.besov_impr_bdd}:
\begin{align*}
\| (1 - t) \mc{D}_2^{-1} B \|_{ L^{\infty, 2}_{x, t} } &\lesssim \Delta + \| Z \|_{ B^{2, 0}_{t, x} } + \| \nabla ( \trace H ) \|_{ B^{2, 0}_{t, x} } + \| H \otimes Z \|_{ B^{2, 0}_{t, x} } \text{.}
\end{align*}
The second term on the right-hand side is trivially bounded, while the third and fourth terms are bounded using \eqref{eqp.main_decomp_ev} and \eqref{eq.est_prod_ex}, respectively.
This yields
\begin{align*}
\| (1 - t) \mc{D}_2^{-1} B \|_{ L^{\infty, 2}_{x, t} } &\lesssim \Delta + \Gamma + \Delta^2 \lesssim \Delta \text{,}
\end{align*}
as desired.
By a similar argument using the second equation in \eqref{eq.structure_renorm_ell}, we obtain
\begin{align*}
\| \mc{D}_1^{-1} R \|_{ L^{\infty, 2}_{x, t} } &\lesssim \| Z \|_{ L^{\infty, 2}_{x, t} } + \| \mc{D}_1^{-1} M \|_{ L^{2, \infty}_{t, x} } + \| \mc{D}_1^{-1} ( \hat{H} \cdot \ul{\hat{H}} ) \|_{ L^{2, \infty}_{t, x} } \\
&\lesssim \Delta + \| M \|_{ B^{2, 0}_{t, x} } + \| H \otimes \ul{H} \|_{ B^{2, 0}_{t, x} } \text{.}
\end{align*}
Using \eqref{eqp.main_decomp_ev} on the above completes the proof of \eqref{eqp.curv_int_tr}.
\end{proof}

Our goal is the subsequent estimate, for which Lemma \ref{thmp.curv_int_tr} is one part of its proof.

\begin{lemma} \label{thmp.curv_int_ex}
If \ass{BA}{\delta, \Delta} holds, then
\begin{align}
\label{eqp.curv_int_ex} \| \cint^t_\star [ (1 - t) \nabla \mc{D}_2^{-1} B ] \|_{ N^{1i}_{t, x} } + \| \cint^t_\star \nabla \mc{D}_1^{-1} R \|_{ N^{1i}_{t, x} } &\lesssim_{C, N} \Gamma + \Delta^2 \text{,} \\
\notag \| (1 - t) \nabla \mc{D}_2^{-1} B \|_{ N^{0 \star}_{t, x} } + \| \nabla \mc{D}_1^{-1} R \|_{ N^{0 \star}_{t, x} } &\lesssim_{C, N} \Gamma + \Delta^2 \text{.}
\end{align}
\end{lemma}

\begin{proof}
Again, we need only prove the first part of \eqref{eqp.curv_int_ex} with $\cint^t_\star \Psi$ replaced by $\cint^t_0 ( \eta_+ \Psi )$, as the corresponding estimates for $\cint^t_\delta ( \eta_- \Psi)$ are completely analogous.
As in the proof of Lemma \ref{thmp.curv_int}, we can apply \eqref{eq.hodge_inv_est} and \eqref{eq.renorm_ass_flux} in order to obtain
\begin{align*}
\| \nabla_t \cint^t_0 [ (1 - t) \eta_+ \nabla \mc{D}_2^{-1} B ] \|_{ L^{2, 2}_{t, x} } + \| \cint^t_0 [ (1 - t) \eta_+ \nabla \mc{D}_2^{-1} B ] \|_{ L^{2, 2}_{t, x} } &\lesssim \| \nabla \mc{D}_2^{-1} B \|_{ L^{2, 2}_{t, x} } \lesssim \Gamma \text{.}
\end{align*}
An analogous estimate also yields
\begin{align*}
\| \nabla_t \cint^t_0 ( \eta_+ \nabla \mc{D}_1^{-1} R ) \|_{ L^{2, 2}_{t, x} } + \| \cint^t_0 ( \eta_+ \nabla \mc{D}_1^{-1} R ) \|_{ L^{2, 2}_{t, x} } &\lesssim \| \nabla \mc{D}_1^{-1} R \|_{ L^{2, 2}_{t, x} } \lesssim \Gamma \text{.}
\end{align*}

For the spatial gradient, we let
\[ I_B = \| \nabla \cint^t_0 [ (1 - t) \eta_+ \nabla \mc{D}_2^{-1} B ] \|_{ L^{2, 2}_{t, x} } \text{,} \qquad I_R = \| \nabla \cint^t_0 ( \eta_+ \nabla \mc{D}_1^{-1} R ) \|_{ L^{2, 2}_{t, x} } \text{.} \]
Applying \eqref{eq.div_curl_est} and commuting $\nabla$ and $\cint^t_0$ using \eqref{eqp.comm_est} yields
\begin{align*}
I_B &\lesssim \| \gamma^{ab} \nabla_a \cint^t_0 [ (1 - t) \eta_+ \nabla_b \mc{D}_2^{-1} B ] \|_{ L^{2, 2}_{t, x} } + \| \epsilon^{ab} \nabla_a \cint^t_0 [ (1 - t) \eta_+ \nabla_b \mc{D}_2^{-1} B ] \|_{ L^{2, 2}_{t, x} } \\
&\qquad + \| \cint^t_0 [ (1 - t) \eta_+ \nabla_b \mc{D}_2^{-1} B ] \|_{ L^{2, 2}_{t, x} } \\
&\lesssim \| \cint^t_0 [ (1 - t) \eta_+ \lapl \mc{D}_2^{-1} B ] \|_{ L^{2, 2}_{t, x} } + \| \cint^t_0 [ (1 - t) | \mc{K} | | \mc{D}_2^{-1} B | ] \|_{ L^{2, 2}_{t, x} } + \| \nabla \mc{D}_2^{-1} B \|_{ L^{2, 2}_{t, x} } \\
&\lesssim \| \cint^t_0 [ (1 - t) \eta_+ \lapl \mc{D}_2^{-1} B ] \|_{ L^{2, 2}_{t, x} } + \| \cint^t_0 [ (1 - t) | \mc{K} - 1 | | \mc{D}_2^{-1} B | ] \|_{ L^{2, 2}_{t, x} } \\
&\qquad + \| \cint^t_0 (1 - t) | \mc{D}_2^{-1} B | \|_{ L^{2, 2}_{t, x} } + \| \nabla \mc{D}_2^{-1} B \|_{ L^{2, 2}_{t, x} } \text{.}
\end{align*}
Recalling the identities \eqref{eq.hodge_sq} and applying \eqref{eq.hodge_inv_est}, \eqref{eq.renorm_ass_flux}, and \eqref{eqp.gc_weak}, then
\begin{align*}
I_B &\lesssim \| \cint^t_0 [ (1 - t) \eta_+ \mc{D}_2^\ast B ] \|_{ L^{2, 2}_{t, x} } + \| \mc{K} - 1 \|_{ L^{2, 2}_{t, x} } \| (1 - t) \mc{D}_2^{-1} B \|_{ L^{\infty, 2}_{x, t} } + \Gamma \\
&\lesssim \| \cint^t_0 [ (1 - t) \nabla_t \cint^t_0 \eta_+ \mc{D}_2^\ast B ] \|_{ L^{2, 2}_{t, x} } + \Delta \| (1 - t) \mc{D}_2^{-1} B \|_{ L^{\infty, 2}_{x, t} } + \Gamma \\
&\lesssim \| \cint^t_0 ( \eta_+ \mc{D}_2^\ast B ) \|_{ L^{2, 2}_{t, x} } + \Delta \| (1 - t) \mc{D}_2^{-1} B \|_{ L^{\infty, 2}_{x, t} } + \Gamma \text{.}
\end{align*}
An analogous argument also produces the estimate
\begin{align*}
I_R &\lesssim \| \cint^t_0 ( \eta_+ \mc{D}_1^\ast R ) \|_{ L^{2, 2}_{t, x} } + \Delta \| \mc{D}_1^{-1} R \|_{ L^{\infty, 2}_{x, t} } + \Gamma \text{.}
\end{align*}

Commuting once again using \eqref{eqp.comm_est} and recalling \eqref{eq.renorm_ass_flux} and \eqref{eqp.curv_int_tr} yields
\begin{align*}
I_B &\lesssim \| \nabla \cint^t_0 ( \eta_+ B ) \|_{ L^{2, 2}_{t, x} } + \Gamma + \Delta^2 \text{,} \\
I_R &\lesssim \| \nabla \cint^t_0 ( \eta_+ R ) \|_{ L^{2, 2}_{t, x} } + \Gamma + \Delta^2 \text{.}
\end{align*}
The first terms on the right-hand sides can be bounded in the same manner as in the proof of \eqref{eqp.curv_int}.
Consequently, we obtain, as desired,
\[ \| \cint^t_\star [ (1 - t) \nabla \mc{D}_2^{-1} B ] \|_{ N^1_{t, x} } + \| \cint^t_\star \nabla \mc{D}_1^{-1} R \|_{ N^1_{t, x} } \lesssim \Gamma + \Delta^2 \text{.} \]

For the initial value bounds, we apply \eqref{eq.nsob_trace} like in the proof of \eqref{eqp.curv_int},
\begin{align*}
\| \cint^t_\star [ (1 - t) \nabla \mc{D}_2^{-1} B ] [0] \|_{ H^{1/2}_x } &\lesssim \| \cint^t_\delta [ (1 - t) \eta_- \nabla \mc{D}_2^{-1} B ] \|_{ N^1_{t, x} } \lesssim \Gamma + \Delta^2 \text{,} \\
\| \cint^t_\star \nabla \mc{D}_1^{-1} R [0] \|_{ H^{1/2}_x } &\lesssim \| \cint^t_\delta ( \eta_- \nabla \mc{D}_1^{-1} R ) \|_{ N^1_{t, x} } \lesssim \Gamma + \Delta^2 \text{,}
\end{align*}
where in the last steps, we applied the $\cint^t_\delta$-analogue of the preceding argument.
Combining all the above estimates completes the proof of the first estimate in \eqref{eqp.curv_int_ex}.
The second estimate follows immediately from the first, since $\nabla_t \cint^t_\star \Psi = \Psi$.
\end{proof}

\subsection{Completion of the Proof}

We can now complete the proof of Theorem \ref{thm.nc_renorm} by proving the strictly improved version of the bootstrap assumptions, \ass{BA}{\delta, \Delta/2}.
A part of this has already been done in \eqref{eqp.hz_sob_imp}.
Here, we will prove the remaining improved estimates associated with the \ass{BA}{\delta, \Delta/2} condition.

\begin{lemma} \label{thm.decomp_imp}
If \ass{BA}{\delta, \Delta} holds, then
\begin{align}
\label{eqp.decomp_imp} \| \nabla H \|_{ N^{0 \star}_{t, x} + B^{2, 0}_{t, x} } + \| \nabla Z \|_{ N^{0 \star}_{t, x} + B^{2, 0}_{t, x} } &\lesssim_{C, N} \Gamma + \Delta^2 \text{,} \\
\notag \| H \|_{ L^{\infty, 2}_{x, t} } + \| Z \|_{ L^{\infty, 2}_{x, t} } &\lesssim_{C, N} \Gamma + \Delta^2 \text{.}
\end{align}
\end{lemma}

\begin{proof}
First of all, since $2 H = 2 \hat{H} + ( \trace H ) \gamma$, we can bound
\begin{align*}
\| \nabla H \|_{ N^{0 \star}_{t, x} + B^{2, 0}_{t, x} } &\lesssim \| \nabla \hat{H} \|_{ N^{0 \star}_{t, x} + B^{2, 0}_{t, x} } + \| \nabla ( \trace H ) \|_{ B^{2, 0}_{t, x} } \\
&\lesssim \| \nabla \hat{H} \|_{ N^{0 \star}_{t, x} + B^{2, 0}_{t, x} } + \Gamma + \Delta^2 \text{,}
\end{align*}
where in the last step, we applied \eqref{eqp.transport}.
For the term with $\hat{H}$, we recall the identity $\mc{D}_2^{-1} \mc{D}_2 = I$ and apply the first equation in \eqref{eq.structure_renorm_ell} along with \eqref{eq.besov_impr_bdd}:
\begin{align*}
\| \nabla \hat{H} \|_{ N^{0\star}_{t, x} + B^{2, 0}_{t, x} } &\lesssim \| \nabla \mc{D}_2^{-1} \mc{D}_2 \hat{H} \|_{ N^{0 \star}_{t, x} + B^{2, 0}_{t, x} } \\
&\lesssim \| (1 - t) \nabla \mc{D}_2^{-1} B \|_{ N^{0 \star}_{t, x} } + \| Z \|_{ B^{2, 0}_{t, x} } + \| \nabla ( \trace H ) \|_{ B^{2, 0}_{t, x} } \\
&\qquad + \| (1 - t) ( H \otimes Z ) \|_{ B^{2, 0}_{t, x} } \text{.}
\end{align*}
The third, second, and first term on the right-hand side are bounded using \eqref{eqp.transport}, \eqref{eqp.hz_sob_imp}, and \eqref{eqp.curv_int_ex}, respectively.
From this, we have
\[ \| \nabla \hat{H} \|_{ N^{0 \star}_{t, x} + B^{2, 0}_{t, x} } \lesssim \Gamma + \Delta^2 + \| H \otimes Z \|_{ B^{\infty, 0}_{t, x} } \text{.} \]
For the last term, we apply \eqref{eq.est_prod_ex} and \eqref{eqp.hz_sob_imp}:
\[ \| H \otimes Z \|_{ B^{\infty, 0}_{t, x} } \lesssim \| H \|_{ N^{1i}_{t, x} } \| Z \|_{ N^{1i}_{t, x} } \lesssim \Delta^2 \text{.} \]
Combining above yields the desired bound for $\nabla H$ in \eqref{eqp.decomp_imp}.

By a similar argument using the second equation in \eqref{eq.structure_renorm_ell}, we obtain
\begin{align*}
\| \nabla Z \|_{ N^{0 \star}_{t, x} + B^{2, 0}_{t, x} } &\lesssim \| \nabla \mc{D}_1^{-1} R \|_{ N^{0 \star}_{t, x} } + \| M \|_{ B^{2, 0}_{t, x} } + \| H \otimes \ul{H} \|_{ B^{2, 0}_{t, x} } \lesssim \Gamma + \Delta^2 \text{,}
\end{align*}
where in the last step, we applied \eqref{eqp.transport}, \eqref{eqp.main_decomp_ev}, and \eqref{eqp.curv_int_ex}.
The above considerations complete the proof of the first estimate in \eqref{eqp.decomp_imp}.

For the second estimate, we apply \eqref{eq.sharp_trace_ex} and recall that $k = H$:
\begin{align*}
\| H \|_{ L^{\infty, 2}_{x, t} } &\lesssim ( 1 + \| H \|_{ L^{2, \infty}_{x, t} } ) ( \| H \|_{ N^{1i}_{t, x} } + \| \nabla H \|_{ N^{0 \star}_{t, x} + B^{2, 0}_{t, x} } ) \text{,} \\
\| Z \|_{ L^{\infty, 2}_{x, t} } &\lesssim ( 1 + \| H \|_{ L^{2, \infty}_{x, t} } ) ( \| Z \|_{ N^{1i}_{t, x} } + \| \nabla Z \|_{ N^{0 \star}_{t, x} + B^{2, 0}_{t, x} } ) \text{.}
\end{align*}
Recalling \eqref{eq.nsob_ineq}, \eqref{eqp.hz_sob_imp}, and the first part of \eqref{eqp.decomp_imp}, then, as desired,
\[ \| H \|_{ L^{\infty, 2}_{x, t} } + \| Z \|_{ L^{\infty, 2}_{x, t} } \lesssim ( 1 + \Gamma + \Delta^2 ) ( \Gamma + \Delta^2 ) \lesssim \Gamma + \Delta^2 \text{.} \qedhere \]
\end{proof}

Applying \eqref{eqp.hz_sob_imp} and \eqref{eqp.decomp_imp} and taking $\Gamma$ (and hence $\Delta$) sufficiently small yields the strictly improved condition \ass{BA}{\delta, \Delta/2}.
Thus, the bootstrap argument implies that \ass{BA}{\delta, \Delta} holds even without the bootstrap assumption.
To obtain the estimates \eqref{eq.renorm_est}, we simply recall the estimates \eqref{eqp.transport}, \eqref{eqp.gc_weak}, \eqref{eqp.hz_sob_imp}, \eqref{eqp.main_decomp_ev}, and \eqref{eqp.decomp_imp}, and we recall that $\Delta$ is simply $\Gamma$ times some (possibly large) constant.
Furthermore, that the \ass{F2}{} and \ass{K}{} conditions hold for $(\mc{N}, \gamma)$ follows from Lemmas \ref{thmp.rf2} and \ref{thmp.gc}.

It remains only to prove the refined curvature estimate \eqref{eq.renorm_est_curv}.
For this, we recall the proof of Lemma \ref{thmp.gc}, in particular, the identity
\[ \mc{K} - 1 = - \frac{1}{2} \trace \ul{H} + (1 - t) \left[ \gamma^{ab} \nabla_a Z_b - M + \frac{1}{2} \paren{ 1 - \frac{2m}{s} } \trace H - \frac{1}{4} \trace H \trace \ul{H} \right] \text{.} \]
Given $0 \leq \tau < 1$, we can restrict our attention to a small segment $\tau \leq t < \tau + \varepsilon$ of the null cone $\mc{N}$.
Again, in terms of the definition of the \ass{K}{} condition from Section \ref{sec.fol_curv}, we take $f \equiv 1$, $V = (1 - t) Z$, and $W$ the remaining terms of the above equation.
From the estimates \eqref{eq.renorm_est}, we see that \ass{K}{1, D^\prime} holds, with
\[ D^\prime \lesssim \sup_{\tau \leq \tau^\prime < \tau + \varepsilon} \| \trace \ul{H} [\tau^\prime] \|_{ L^2_x } + (1 - \tau) \Gamma \text{.} \]
The estimate \eqref{eq.renorm_est_curv} now follows from \eqref{eq.curv_sob} and by taking $\varepsilon \searrow 0$.

This concludes the proof of Theorem \ref{thm.nc_renorm}.

\raggedright
\bibliographystyle{amsplain}
\bibliography{articles,books,misc}

\end{document}